\documentclass[11pt]{article}

\usepackage[utf8]{inputenc}

\usepackage{geometry} 
\usepackage{graphicx}
\usepackage[font=sf]{caption}
\usepackage[font=sf, size=smaller]{subcaption}
\usepackage{hyperref}
\hypersetup{
    colorlinks=false,
    linkcolor=blue,
    filecolor=blue,      
    urlcolor=blue,
}

\usepackage{array} 
\usepackage{paralist} 
\usepackage{verbatim} 
\usepackage{lipsum}
\usepackage{amsmath, amsfonts, amssymb, amsxtra, amsthm, mathrsfs}
\usepackage{enumitem}
\usepackage{color}
\usepackage{tikz,tikz-3dplot,pgfplots}
\usepackage{multirow}
\usepackage{hyperref}
\hypersetup{
 colorlinks = true,
     linkcolor = blue,
     anchorcolor = blue,
     citecolor = brown,
     filecolor = blue,
     urlcolor = blue
}
\usepackage{tocloft}
\usepackage{authblk}

\theoremstyle{plain}
\newtheorem{thm}{Theorem}[section]
\newtheorem{prop}[thm]{Proposition}
\newtheorem{lm}[thm]{Lemma}
\newtheorem{cor}[thm]{Corollary}

\newtheorem{clm}{Claim}
\newtheorem{clmA}{Claim}

\theoremstyle{definition}

\newtheorem{rmk}[thm]{Remark}

\newtheorem{alg}[thm]{Algorithm}

\DeclareMathOperator{\isom}{Isom}

\DeclareMathOperator{\hull}{hull}
\DeclareMathOperator{\core}{core}

\DeclareMathOperator{\psl}{PSL}

\DeclareMathOperator{\stab}{Stab}
\DeclareMathOperator{\Int}{Int}

\DeclareMathOperator{\bolda}{\mathbf{a}}
\DeclareMathOperator{\boldb}{\mathbf{b}}
\definecolor{cof}{RGB}{219,144,71}
\definecolor{pur}{RGB}{186,146,162}
\definecolor{greeo}{RGB}{91,173,69}
\definecolor{greet}{RGB}{52,111,72}
\usetikzlibrary{calc}

\pgfplotsset{compat=1.14}

\usepackage[calcwidth]{titlesec}
\titlespacing*{\paragraph}{0pt}{0pt}{1em}
\newcommand{\periodafter}[1]{#1.}
\titleformat{\paragraph}[runin]{\bfseries}{\theparagraph}{}{\periodafter}
\titleformat{\part}[hang]{\Large\bfseries}{Part \thepart: }{1ex}{}{}

\numberwithin{equation}{section}

\cftpagenumbersoff{part}

\title{\bfseries Geodesic planes in a geometrically finite end and the halo of a measured lamination}
\author[1]{Tina Torkaman}
\author[2]{Yongquan Zhang}
\affil[1]{\small Department of Mathematics, University of Chicago}
\affil[2]{\small Institute for Mathematical Sciences, Stony Brook University}
\date{August 15, 2024}

\begin{document}

\maketitle

\begin{abstract}
Recent works \cite{MMO1,MMO2,acy_geom_finite,exotic_plane} have shed light on the topological behavior of geodesic planes in the convex core of a geometrically finite hyperbolic 3-manifold $M$ of infinite volume. In this paper, we focus on the remaining case of geodesic planes outside the convex core of $M$, giving a complete classification of their closures in $M$.

In particular, we show that the behavior is different depending on whether exotic roofs exist or not. Here an \emph{exotic roof} is a geodesic plane contained in an end $E$ of $M$, which limits on the convex core boundary $\partial E$, but cannot be separated from the core by a support plane of $\partial E$.

A necessary condition for the existence of exotic roofs is the existence of exotic rays for the bending lamination. Here an \emph{exotic ray} is a geodesic ray that has a finite intersection number with a measured lamination $\mathcal{L}$ but is not asymptotic to any leaf nor eventually disjoint from $\mathcal{L}$. We establish that exotic rays exist if and only if $\mathcal{L}$ is not a multicurve. The proof is constructive, and the ideas involved are important in the construction of exotic roofs.

We also show that the existence of geodesic rays satisfying a stronger condition than being exotic, phrased only in terms of the hyperbolic surface $\partial E$ and the bending lamination, is sufficient for the existence of exotic roofs. As a result, we show that geometrically finite ends with exotic roofs exist in every genus. Moreover, in genus $1$, when the end is homotopic to a punctured torus, a generic one (in the sense of Baire category) contains uncountably many exotic roofs.
\end{abstract}

\tableofcontents
\thispagestyle{empty}
\newpage

\clearpage
\pagenumbering{arabic}

\section{Introduction}
In this paper, we investigate two related problems. The first problem concerns
\begin{enumerate}[label={(\Roman*)}, topsep=0mm, itemsep=0mm]
    \item The topological behavior of geodesic planes outside the convex core of a geometrically finite hyperbolic three manifold $M$ of infinite volume.
\end{enumerate}
Our goal is to classify the closure of such geodesic planes, following a program initiated by recent breakthroughs \cite{MMO1,MMO2}. We answer this question completely (see Theorem~\ref{thm: minimal_roof} and Propositions~\ref{prop: atomic} and \ref{prop:several_comp}), and find a new and unexpected phenomenon that must be dealt with first. One of our main results is the following:
\begin{thm}\label{thm:main_intro}
There exists a quasifuchsian manifold with uncountably many exotic roofs.
\end{thm}
Here an \emph{exotic roof} is a geodesic plane disjoint from $\core(M)$, limiting on $\partial\core(M)$, yet cannot be separated from it by a support plane of $\partial\core(M)$. See Corollary~\ref{cor:higher_genus} for a more precise version.

The existence of exotic roofs relies on finer structures of the bending lamination $\mathcal{L}$ of the convex core boundary, and in particular is related to the existence of \emph{exotic rays} for $\mathcal{L}$. This leads to our second problem, which concerns
\begin{enumerate}[label={(\Roman*)}, topsep=0mm, itemsep=0mm, resume]
    \item The growth rate of the intersection number with a given measured geodesic lamination $\mathcal{L}$ along a geodesic ray on a hyperbolic surface $X$ of finite area.
\end{enumerate}
Exotic rays for $\mathcal{L}$ have the unexpected property that they have finite intersection number with $\mathcal{L}$, even though they meet $\mathcal{L}$ recurrently. Our main result in this regard is the following:
\begin{thm} \label{main1}
Provided that $\mathcal{L}$ is not a multicurve, there exists an exotic ray for $\mathcal{L}$.
\end{thm}
Problems (I) and (II) are closely related: the existence of exotic rays for the bending lamination is necessary for the existence of exotic roofs, and the existence of a geodesic ray for which the growth in (II) is \emph{super-exponentially} slow is sufficient (see Theorem~\ref{thm: sufficient_intro}).

\paragraph{Geodesic planes in hyperbolic 3-manifolds}
We start by describing the motivational problem (I) in more detail. Let $M=\Gamma\backslash\mathbb{H}^3$ be a complete, oriented, hyperbolic 3-manifold of constant curvature $-1$, specified by a Kleinian group $\Gamma\subset\isom^+(\mathbb{H}^3)\cong\psl(2,\mathbb{C})$. Let $\Lambda\subset S^2\cong\hat{\mathbb{C}}$ be its limit set. The \emph{convex core} of $M$ is given by
$$\core(M)=\Gamma\backslash\hull(\Lambda)\subset M,$$
where $\hull(\Lambda)\subset\mathbb{H}^3$ is the convex hull of $\Lambda$. Equivalently, $\core(M)$ is the smallest closed convex subset of $M$ so that the inclusion induces a homotopy equivalence. A hyperbolic 3-manifold $M$ is said to be \emph{convex cocompact} if $\core(M)$ is compact, and \emph{geometrically finite} if the unit neighborhood of $\core(M)$ has finite volume.

A \emph{geodesic plane} in a hyperbolic 3-manifold $M$ is a totally geodesic isometric immersion $f:\mathbb{H}^2\to M$. We often identify $f$ with its image $P=f(\mathbb{H}^2)$ and call the latter a geodesic plane as well. The map $f$ lifts to a map $\tilde f:\mathbb{H}^2\to\mathbb{H}^3$, whose image is a totally geodesic plane in $\mathbb{H}^3$, which in turn determines a circle $C$ in $S^2$, its boundary at infinity. Different lifts of $f$ give the orbit of $C$ under $\Gamma$. Conversely, any oriented circle in $S^2$ determines a geodesic plane in $M$. We will make use of this correspondence repeatedly, and refer to $C$ as a \emph{boundary circle} of $P$.

We are interested in the topological behavior of geodesic planes in $M$. When $M$ has finite volume, any geodesic plane is either closed or dense, independently due to Ratner \cite{ratner} and Shah \cite{shah}. In particular, this is a consequence of Ratner's much more general theorem on orbit closures of unipotent subgroups. In the infinite volume setting, when $M$ is convex cocompact and \emph{acylindrical}, geodesic planes intersecting the interior of $\core(M)$ satisfy strong rigidity property in the spirit of Ratner: they are either closed or dense in $\Int(\core(M))$ \cite{MMO1, MMO2}. This is generalized in \cite{acy_geom_finite} to certain geometrically finite acylindrical manifolds as well.

A geodesic plane's behavior near the convex core boundary, on the other hand, is more subtle. In \cite{exotic_plane}, the second named author constructed an explicit example of a geodesic plane in a convex cocompact acylindrical hyperbolic 3-manifold that is closed in the interior of the convex core, but not in the whole manifold.

We also refer to \cite{khalil2019geodesic, apollonian_elem} for some recent developments in the geometrically finite setting.

\paragraph{Geodesic planes in a geometrically finite end}
In this paper, we focus on the final piece missing from the discussions above and study geodesic planes outside the convex core of a geometrically finite 3-manifold with incompressible boundary. An \emph{end} of $M$ is a connected component of $M\backslash\Int(\core(M))$. Since the planes of interest are contained entirely in an end, we may as well assume $M$ is a quasifuchsian manifold (see \S\ref{sec: quasifuchsian}). In this case, $M$ has two ends $E_\pm$, corresponding to the two components $\Omega_\pm$ of the domain of discontinuity. The boundary $X_\pm=\partial E_\pm$ is isometric to a hyperbolic surface of finite area, piecewise totally geodesic, bent along a \emph{bending lamination} $\mathcal{L}_\pm$. Note that we allow cusps for $X_\pm$.

Each complement of $E_\pm-\mathcal{L}_\pm$ is called a \emph{flat piece}, bordered by complete geodesics in $\mathcal{L}_\pm$. Geometrically, a flat piece is either an ideal polygon, or contains closed curves that are not homotopically trivial (see Figure~\ref{fig:flat_pieces}). We refer to \S\ref{sec: quasifuchsian} and \cite{bending} for more details on the geometry of convex core boundary.

\begin{figure}[htp]
    \centering
    \begin{minipage}{0.4\linewidth}
        \begin{subfigure}{\linewidth}
        \centering
          \includegraphics[width=0.7\textwidth]{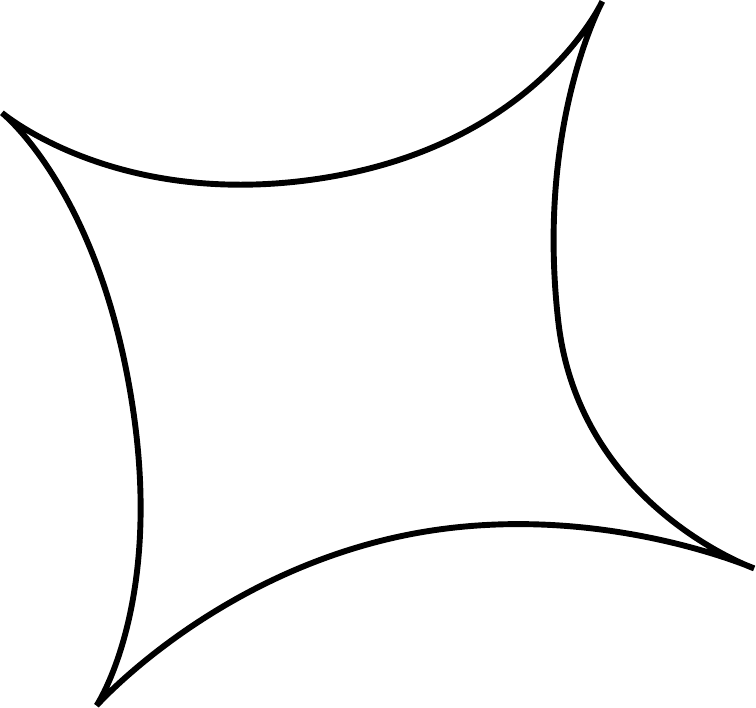}
          \caption{An ideal polygon}
    \end{subfigure}
    \end{minipage}
    \begin{minipage}{0.4\linewidth}
    \begin{subfigure}{\linewidth}
        \centering
          \includegraphics[width=\textwidth]{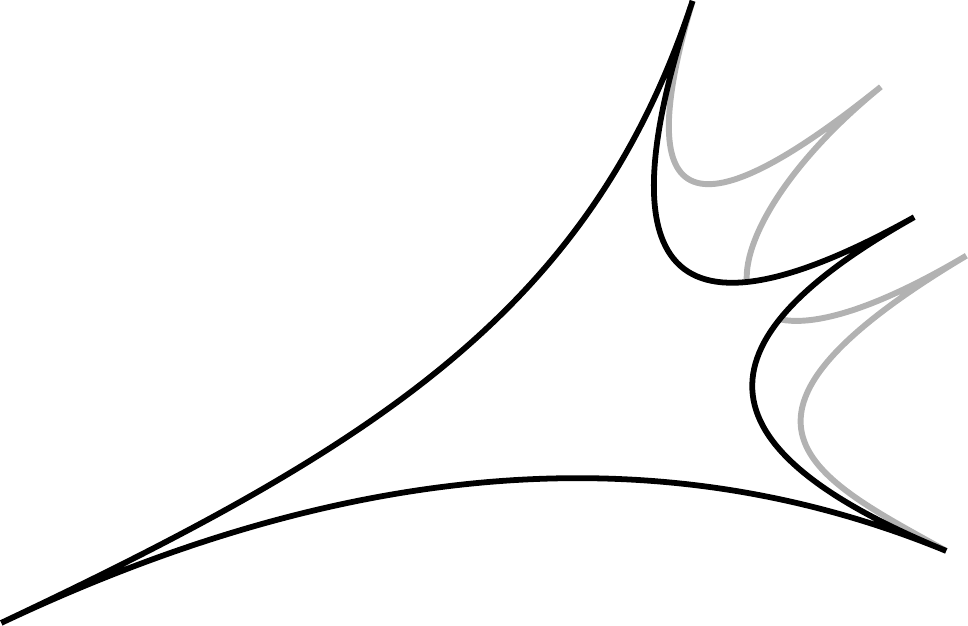}
          \caption{An punctured ideal polygon}
          \label{subfig:punctured_polygon}
    \end{subfigure}
    \end{minipage}
    \begin{minipage}{0.4\linewidth}
    \begin{subfigure}{\linewidth}
    \centering
          \includegraphics[width=0.7\linewidth]{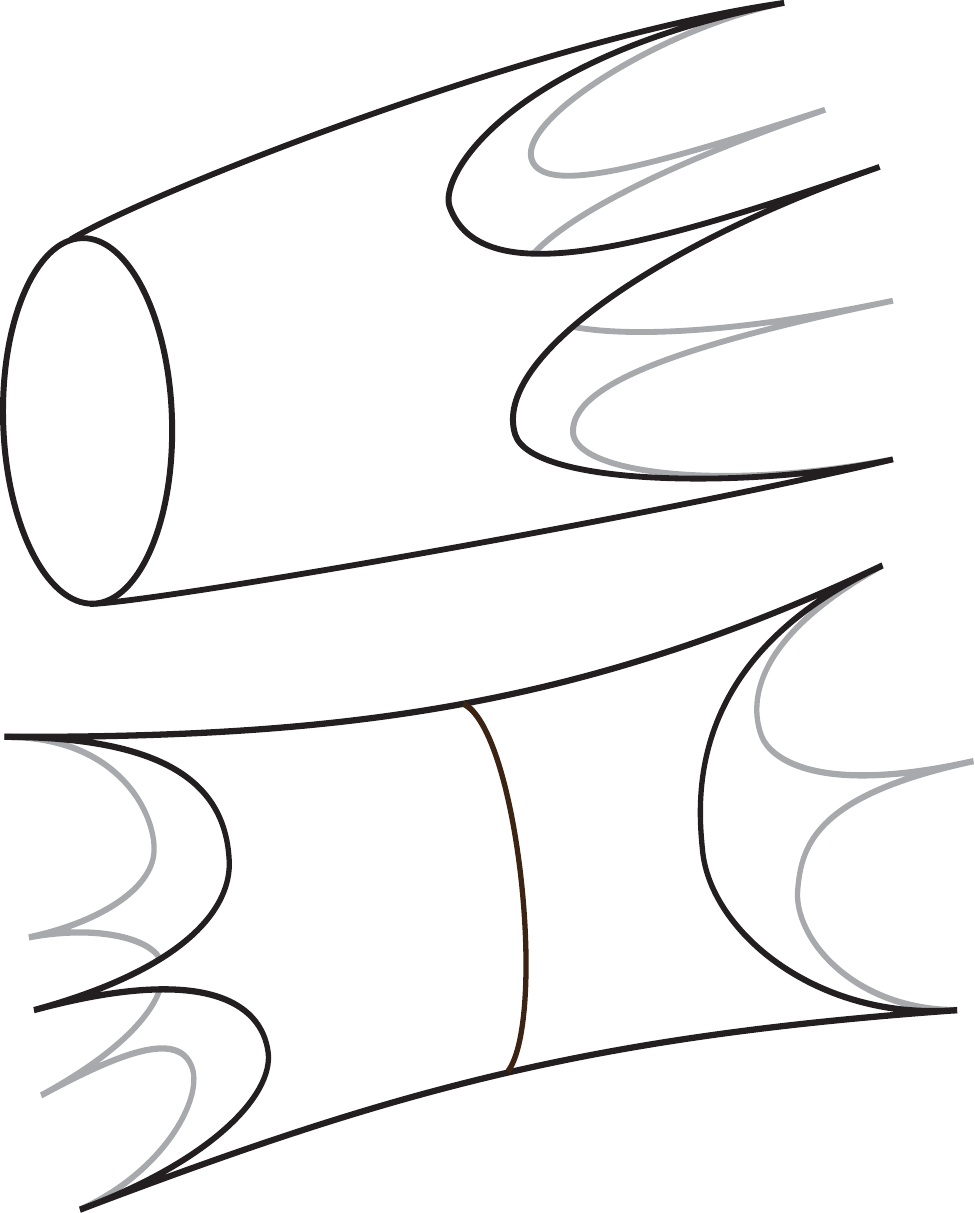}
          \caption{A crown or double crown}
          \label{subfig:crown_double_crown}
    \end{subfigure}
    \end{minipage}
    \begin{minipage}{0.4\linewidth}
    \begin{subfigure}{\linewidth}
    \centering
          \includegraphics[width=\linewidth]{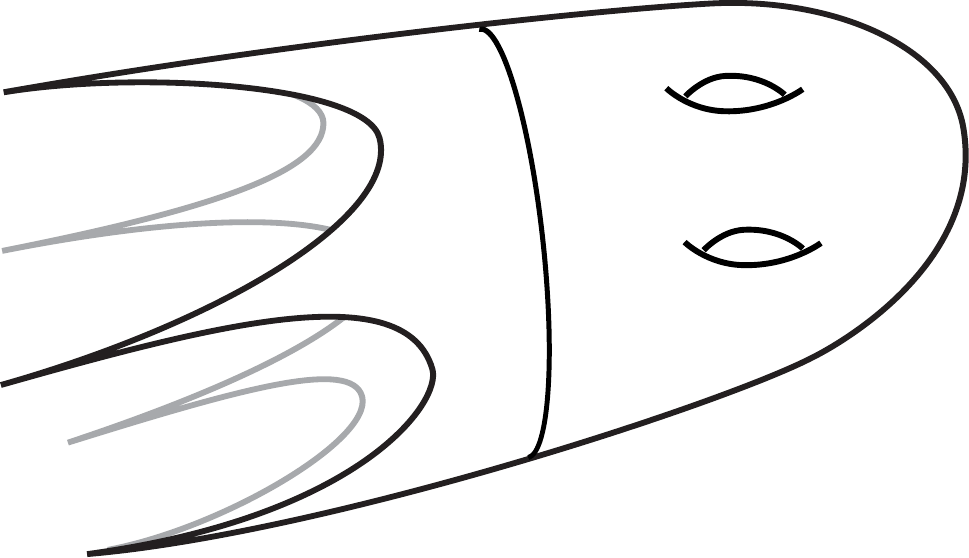}  
          \caption{A nonelementary surface, decorated with a crown along its geodesic boundary}
          \label{subfig:nonelementary}
    \end{subfigure}
    \end{minipage}
    \caption{A menagerie of possible shapes of flat pieces}
    \label{fig:flat_pieces}
\end{figure}

Let $P$ be a geodesic plane contained in $E_+$, and $C$ a boundary circle of $P$. Then $C$ is contained completely in $\overline{\Omega_+}$. We have the following possibilities:
\begin{enumerate}[topsep=0 mm, label={(\roman*)}, itemsep=0 mm]
\item\label{case: support} $P$ meets the convex core boundary, and $|C\cap\Lambda|\ge2$ (such a plane is called a \emph{support plane}, and the corresponding circle a \emph{support circle}). Note that a support plane meets the convex core boundary in a flat piece or a leaf of the bending lamination.
\item\label{case: horo} $P$ is disjoint from but accumulates on the convex core boundary, and $|C\cap\Lambda|=1$ (such a plane is called an \emph{asymptotic plane}, and the corresponding circle an \emph{asymptotic circle}).
\item\label{case: bounded} $P$ is bounded away from the convex core boundary, and $C$ is disjoint from the limit set.
\end{enumerate}

For Case~\ref{case: bounded}, since the action of $\Gamma$ is properly discontinuous on the domain of discontinuity, the orbit of $C$ only accumulates on the limit set. This implies that $P$ is closed in $M$.

Up to conformal conjugation, assume $\Omega_+$ is bounded in $\mathbb{C}$. Given a geodesic plane $P$ in $E_+$ and a boundary circle $C$ of $P$, a smaller circle $C'$ contained in the closed disk bounded by $C$ gives another geodesic plane $P'$. We say $P'$ is \emph{shadowed} by $P$ (and $C'$ is shadowed by $C$). We call $P'$ a \emph{roof} if it is not shadowed by any other geodesic plane. In particular, all support planes are roofs, and no planes in Case~\ref{case: bounded} are roofs. A roof is said to be \emph{exotic} if it is of Case~\ref{case: horo}. We call boundary circles of exotic roofs \emph{exotic circles}.

\paragraph{Classification of closure of geodesic planes} We are now ready to state our main results on geodesic planes outside convex cores. The picture is distinctly different for atomic bending laminations and non-atomic ones. Proposition~\ref{prop: atomic} below deals with the case where the bending lamination is purely atomic (i.e.\ it is a \emph{multicurve}, consisting of disjoint simple closed geodesics), and Theorem~\ref{thm: minimal_roof} deals with the case where it is minimal and non-atomic. Finally, Proposition~\ref{prop:several_comp} describes how to combine the two cases to obtain a classification for any general bending lamination.

We start with the purely atomic case:
\begin{prop}\label{prop: atomic}
Let $M$ be a quasifuchsian manifold, $\Lambda$ its limit set, and suppose that the bending lamination $\mathcal{L}$ of one of its ends $E_+$ is purely atomic. Let $f:\mathbb{H}^2\to M$ be a geodesic plane contained in $E_+$, $P$ its image, and $C$ a boundary circle of $P$. Then one of the following happens:
\begin{enumerate}[label=\normalfont{(\arabic*)}, topsep=0mm, itemsep=0mm]
\item $C\cap\Lambda=\varnothing$, and $P$ is closed.
\item\label{item:cylinder} $|C\cap\Lambda|=2$, and $P$ is closed. The map $f$ factors through a cylinder $\langle g\rangle\backslash\mathbb{H}^2$, where $g$ is a hyperbolic element generating $\stab_{\Gamma}(C)$.
\item\label{item:nonelementary} $C\cap\Lambda$ is a Cantor set, and $P$ is closed. The map $f$ factors through a nonelementary geometrically finite hyperbolic surface of infinite volume $\stab_{\Gamma}(C)\backslash\mathbb{H}^2$.
\item\label{item:asy_cylinder} $|C\cap\Lambda|=1$, and $P$ is shadowed by a support plane $P'$ of type \ref{item:cylinder} above. Moreover, $\overline{P}=P\cup P'$.
\item\label{item:asy_nonelem} $|C\cap\Lambda|=1$ where $p=C\cap\Lambda$ is not the fixed point of a parabolic element, and $P$ is shadowed by a support plane $P'$ (with a boundary circle $C'$) of type \ref{item:nonelementary} above. Moreover, $\overline{P}=\stab_\Gamma(C')\backslash H(C')$.
\item\label{item:asy_parabolic} $|C\cap\Lambda|=1$ where $p=C\cap\Lambda$ is the fixed point of a parabolic element in $\Gamma$, and $P$ is closed.
\end{enumerate}
\end{prop}
Here in \ref{item:asy_nonelem}, $H(C')$ is the half-space bounded by the hyperbolic plane determined by the circle $C'$ and contained in the lift of $E_+$. We also remark that the distinction between \ref{item:asy_nonelem} and \ref{item:asy_parabolic} is only necessary when $X_+=\partial E_+$ has cusps; otherwise $\Gamma$ contains no parabolic elements. These cases can be nicely illustrated in terms of the $\Gamma$-orbit of the boundary circle $C$; see Figure~\ref{fig:atomic_limit_set}.

\begin{figure}[htp]
\captionsetup{width=.9\textwidth}
    \centering
    \begin{minipage}{0.45\linewidth}
    \begin{subfigure}{\linewidth}
    \includegraphics[trim={1cm 7cm 1cm 7cm},clip,width=\textwidth]{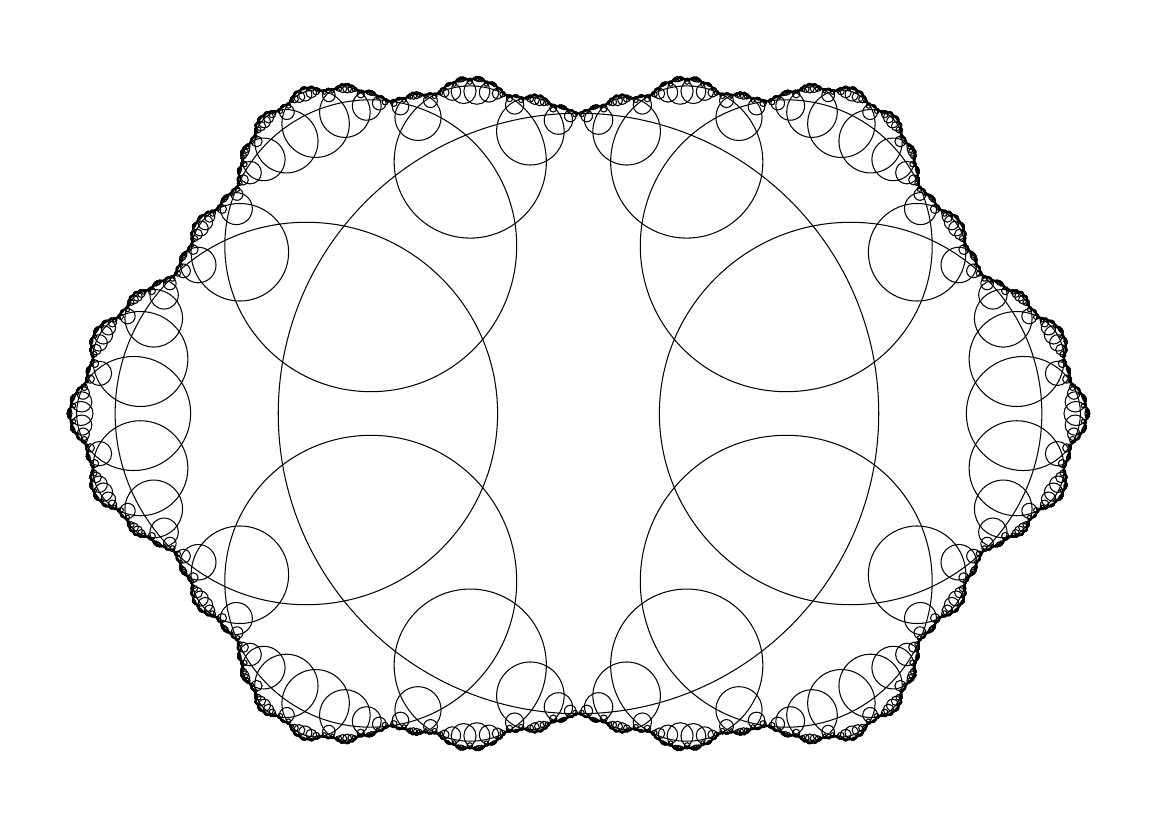}
    \caption{A support plane of type \ref{item:cylinder}}
    \label{subfig:cylinder}
    \end{subfigure}
    \end{minipage}
    \begin{minipage}{0.45\linewidth}
    \begin{subfigure}{\linewidth}
    \includegraphics[trim={1cm 7cm 1cm 7cm},clip,width=\textwidth]{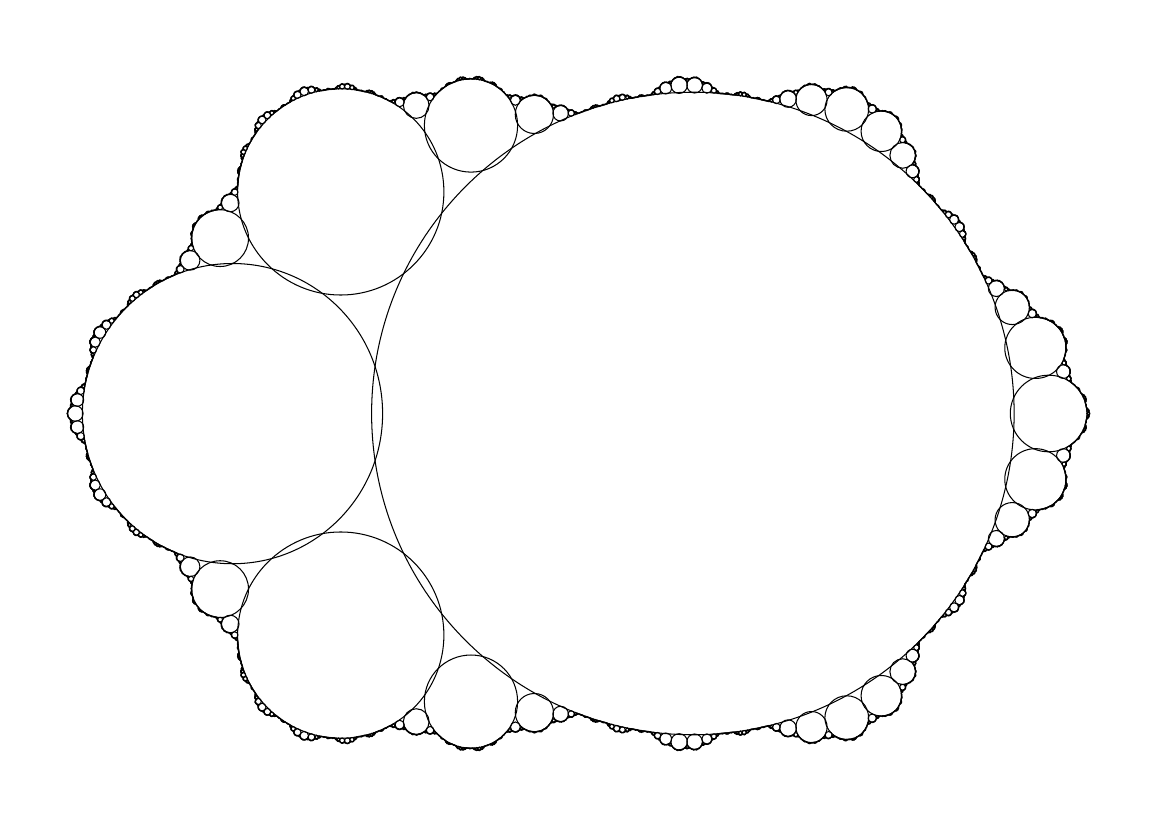}
    \caption{A support plane of type \ref{item:nonelementary}}
    \label{subfig:support}
    \end{subfigure}
    \end{minipage}
    
    \begin{minipage}{0.45\linewidth}
    \begin{subfigure}{\linewidth}
    \captionsetup{width=.6\textwidth}
    \includegraphics[trim={1cm 7cm 1cm 7cm},clip,width=\textwidth]{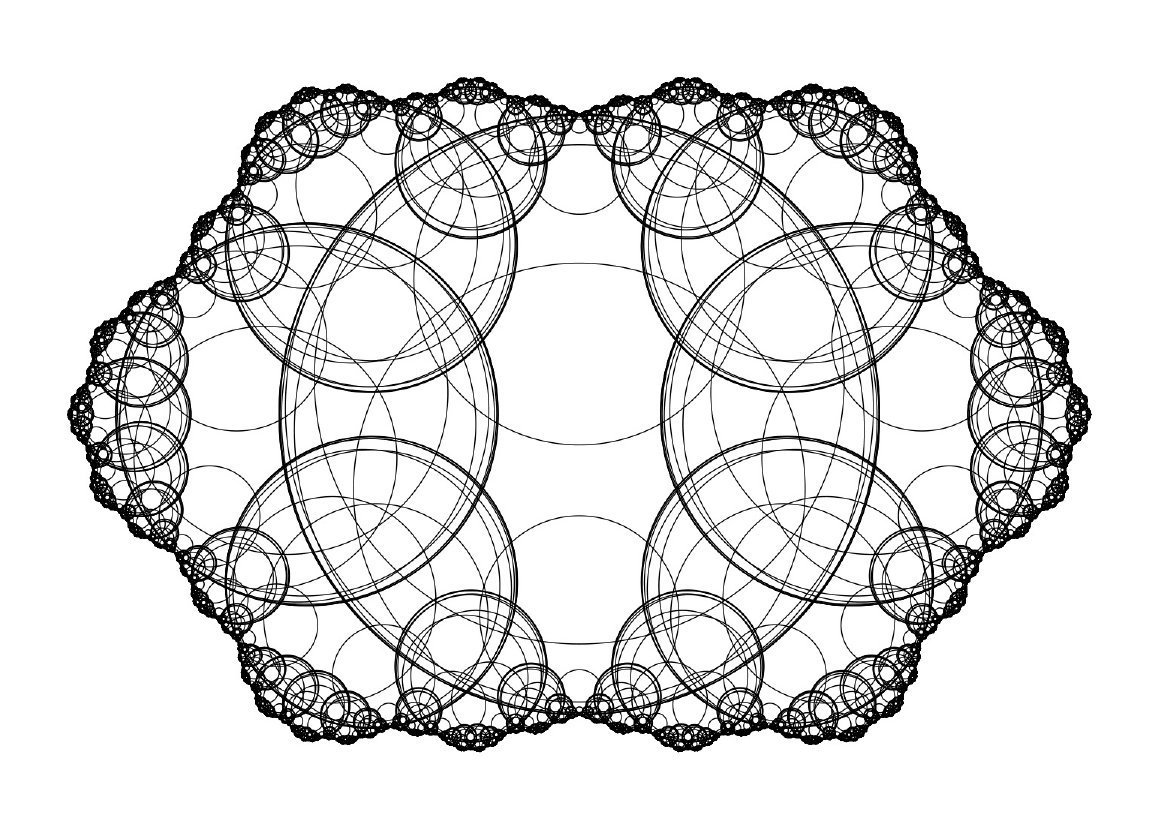}
    \caption{An asymptotic plane of type \ref{item:asy_cylinder}, shadowed by the example in Figure~\ref{subfig:cylinder}}
    \end{subfigure}
    \end{minipage}
    \begin{minipage}{0.45\linewidth}
    \begin{subfigure}{\linewidth}
    \captionsetup{width=.6\textwidth}
    \includegraphics[trim={1cm 7cm 1cm 7cm},clip,width=\textwidth]{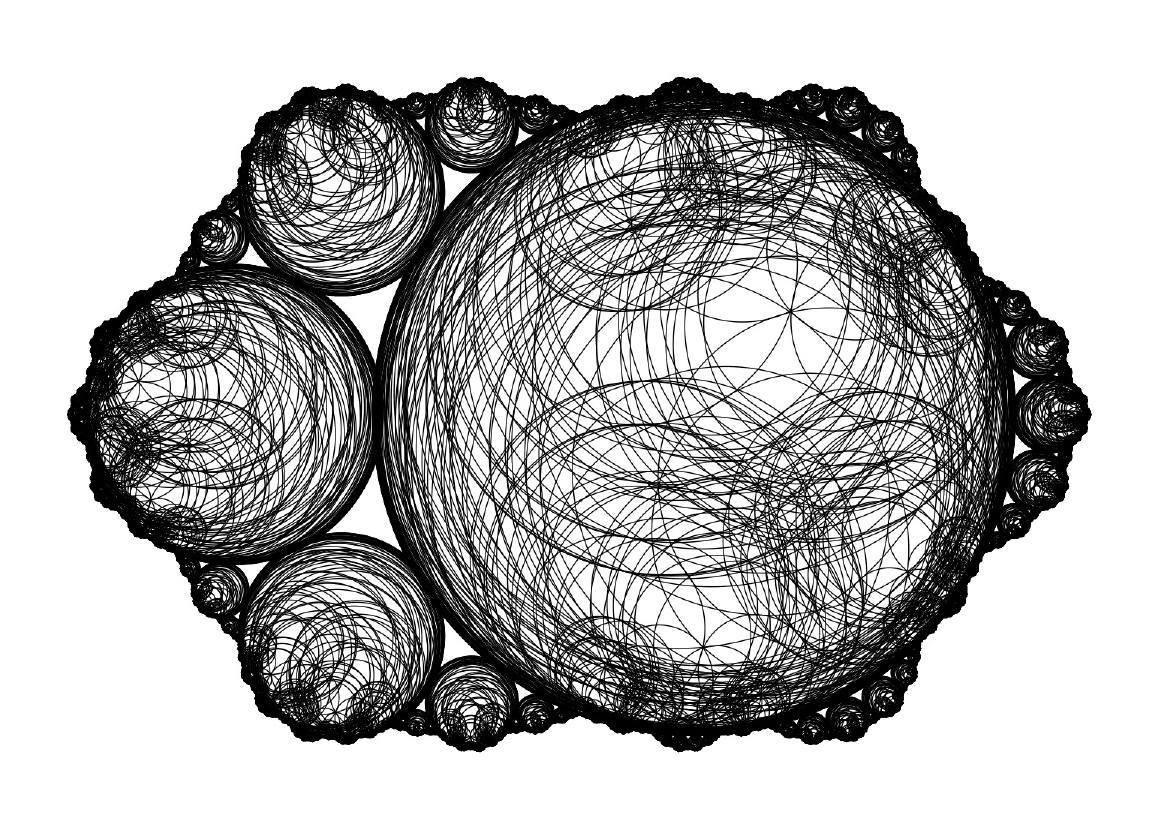}
    \caption{An asymptotic plane of type \ref{item:asy_nonelem}, shadowed by the example in Figure~\ref{subfig:support}}
    \end{subfigure}
    \end{minipage}
    \begin{minipage}{0.45\linewidth}
    \begin{subfigure}{\linewidth}
    \captionsetup{width=.6\textwidth}
    \includegraphics[trim={1cm 7cm 1cm 7cm},clip,width=\textwidth]{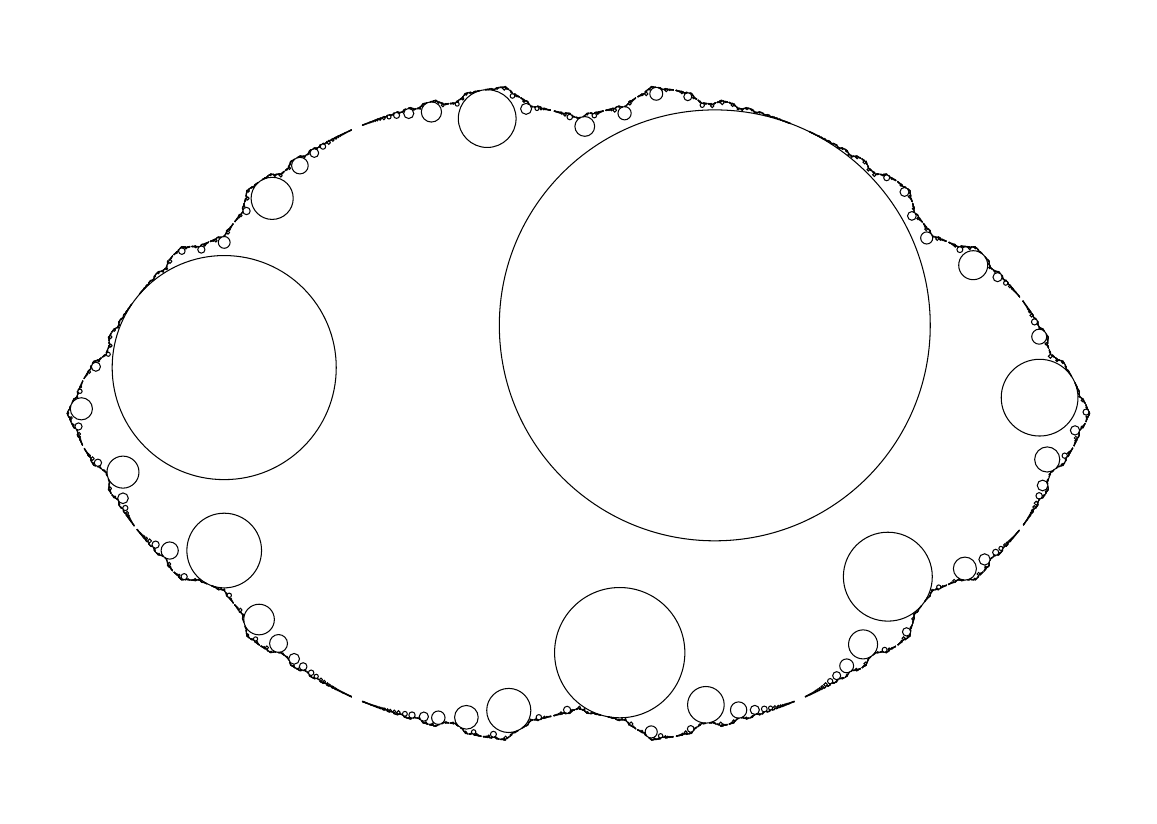}
    \caption{An asymptotic plane of type \ref{item:asy_parabolic}}
    \end{subfigure}
    \end{minipage}
    \caption{Illustrations of Cases (2)-(6) in Proposition~\ref{prop: atomic}. (a)-(d) are drawn in the same quasifuchsian manifold without cusps, and (e) is drawn in a quasifuchsian manifold with cusps}
    \label{fig:atomic_limit_set}
\end{figure}

We next consider minimal laminations without atom. Let $C'$ be an asymptotic circle shadowed by a support circle $C$. If $C'\cap\Lambda$ is the endpoint of a leaf, we refer to $C'$ as \emph{stemmed}, otherwise \emph{enclosed}. Note that enclosed asymptotic circles only exist for $C$ with nontrivial stabilizer in $\Gamma$. Indeed, when the stabilizer is trivial, the corresponding plane $\tilde P$ intersects the convex core boundary either in a leaf or an ideal polygon bounded by leaves (see the discussion above on the geometry of convex core boundary, or in greater detail \S\ref{sec: quasifuchsian} and \cite{bending}), and hence $C$ meets the limit set in finitely many points, which are all endpoints of leaves. 

On the other hand, when the stabilizer is nontrivial, by \cite[Thm.~5.1]{MMO2} and \cite[Lem.~3.7]{acy_geom_finite}, a point of $C\cap\Lambda$ either lies in the limit of $\stab_\Gamma(C)$, or is the tip of a crown. The former are not endpoints of leaves while the latter are, so any enclosed circle $C'$ shadowed by $C$ must meet the $\Lambda$ at a point in the limit set of $\stab_\Gamma(C)$.

We have the following classification in this case.
\begin{thm}\label{thm: minimal_roof}
With the same notation as the previous proposition, assume $\mathcal{L}$ is minimal and atom-free. Then we have the following possibilities:
\begin{enumerate}[label=\normalfont{(\arabic*)}, topsep=0mm, itemsep=0mm]
    \item $C\cap\Lambda=\varnothing$, and $P$ is closed;
    \item\label{item:thm:minimal_roof_parabolic} $|C\cap\Lambda|=1$ where $p=C\cap\Lambda$ is the fixed point of a parabolic element in $\Gamma$, and $P$ is closed;
    \item\label{item:thm:mimial_roof_dense} $|C\cap\Lambda|\ge1$, $C$ is not of the case above, and $\overline{P}=E_+$.
    
    Suppose further that $E_+$ contains no exotic roofs. Then we have:
\begin{itemize}
    \item If $|C\cap\Lambda|\ge 2$, then the closure of $\Gamma\cdot C$ is exactly the union $\mathcal{S}$ of all support circles and enclosed asymptotic circles;
    \item If $|C\cap\Lambda|=1$, then $\overline{\Gamma\cdot C}=\Gamma\cdot C\cup\mathcal{S}$.
\end{itemize}

Suppose otherwise that $E_+$ contains an exotic roof. Then we have:
\begin{itemize}
    \item If $|C\cap\Lambda|\ge1$, then $\overline{\Gamma\cdot C}$ consists of all support and asymptotic circles.
\end{itemize}
\end{enumerate}
\end{thm}
Note that in \ref{item:thm:mimial_roof_dense}, the two situations (with or without exotic roofs) are indistinguishable if we only look at the closures of planes in the 3-manifold, as we always have $\overline{P}=E_+$. Instead, we look at the corresponding orbits of circles, which provide finer details of the closure. Equivalently, we can look at the closure of tangent frames to the planes in the frame bundle over the manifold (see \cite{MMO1} for more details on circles, planes and frames). When exotic roofs exist, the closure of the set of frames tangent to a support or asymptotic plane is larger: it contains all possible frames lying over all support and asymptotic planes, including all stemmed asymptotic planes.

This highlights one key difference between stemmed and enclosed asymptotic circles / planes. As we will see in the proof, enclosed asymptotic circles are always included in the closure of the orbit $\Gamma\cdot C$ due to the dynamics of $\stab_\Gamma(C)$, as a consequence of Dal'bo's result on circles based at points in the limit set \cite{horocycle}. Stemmed asymptotic circles, on the other hand, are based at points disjoint from the limit set of $\stab_\Gamma(C)$, and hence the dynamics of the stabilizer is not enough to include these circles in the closure.

Finally, the case of a general bending lamination is a mix of the two cases above. Write $\mathcal{L}=\mathcal{L}_a\sqcup\mathcal{L}_m$ as the union of atomic parts $\mathcal{L}_a$ and non-atomic parts $\mathcal{L}_m$. Decompose $X_+\backslash\mathcal{L}_a=\bigsqcup_{t} X'_t$ into connected components, and set $X_t$ to be the closure of $X'_t$ in $X$. For each $t$, set $\mathcal{L}_t=\mathcal{L}_m\cap X_t$. Write $\mathcal{L}_t=\bigsqcup_\alpha\mathcal{L}_t^\alpha$ into minimal components. Let $X_t^\alpha \subset X_t$ be the smallest closed subsurface with geodesic boundary containing $\mathcal{L}_t^\alpha$, and set $X_t^0=\overline{X_t\backslash\cup_\alpha X_t^\alpha}$. We have

\begin{prop}\label{prop:several_comp}
\begin{enumerate}[label=\normalfont{(\arabic*)}, topsep=0mm, itemsep=0mm]
    \item If $s\neq t$, then any support plane or asymptotic plane for $X_s$ has closure disjoint from that for $X_t$;
    \item  For each $\alpha$, depending on whether $\mathcal{L}_t^\alpha$ exhibits an exotic roof or not, the closure of a support plane for $X_t$ either contains all support planes and asymptotic planes for $X_t^\alpha$, or all support planes and enclosed asymptotic planes for $X_t^\alpha$ respectively;
    \item The closure of a support plane for $X_t$ contains all support planes and enclosed asymptotic planes for $X_t^0$.
\end{enumerate}
\end{prop}

\paragraph{Dynamics vs geometry}
Before we move on, we make some remarks on techniques used in previous works, and the difficulties we face for the current problem. Previous works \cite{MMO1,MMO2,acy_geom_finite} made use of the structure of the limit set $\Lambda$ when $M$ is acylindrical, and in particular it is proved that for any boundary circle $C$ of a geodesic plane intersecting the interior of the convex core, $C\cap\Lambda$ contains a \emph{thick} cantor set. This then implies sufficiently frequent return of the unipotent flow to the convex core, via \emph{unipotent blowup}, and is enough for the ``closed or dense" dichotomy.

When the convex core boundary is totally geodesic, its fundamental group is a lattice in a conjugate of $\psl(2,\mathbb{R})$. The chaotic dynamics of its horocycle flow is also particularly useful in \cite{MMO1}.

On the other hand, for the current problem, the boundary circles of interest interact minimally with the limit set. Moreover, generically, the complement of the bending lamination in the convex core boundary consists of ideal polygons, so there do not exist nonelementary subgroups of some conjugate of $\psl(2,\mathbb{R})$ whose dynamics we can make use of. Therefore, instead of using dynamics, we study the corresponding geodesic planes via the geometry of the convex core boundary, in particular its bending lamination.

However, we make the following observation. let $M=\Gamma\backslash\mathbb{H}^3$ be a hyperbolic 3-manifold covered by a quaisfuchsian manifold $M_0=\Gamma_0\backslash\mathbb{H}^3$. Assume the boundary of an end $E$ of $M_0$ projects down to a properly immersed pleated surface $X$ in $M$. We have
\begin{prop}
   Assume $E$ contains an exotic roof, and the bending lamination on $\partial E$ is minimal. Suppose also that $X$ is not part of the convex core boundary of $M$. Then any geodesic plane $P$ limiting on $X$ is dense in $M$.
\end{prop}
\begin{proof}
    Since $\Gamma$ contains a Zariski dense subgroup $\Gamma_0$, it is itself Zariski dense. Since the closure of any support plane of $E$ in $M_0$ contains all support and asymptotic planes on that side, the corresponding circles sweep out an open set of the domain of discontinuity. The projection of these planes is contained in the closure of $P$. Moreover, these planes all meet the limit set of $M$, since $X$ is not part of the convex core boundary. Since the circles sweep out an open subset of the limit set of $\Gamma$, the result follows from \cite[Cor.~4.2]{MMO1}.
\end{proof}
In other words, pleated surfaces with exotic roofs could play the role of closed geodesic planes in previous works, in the sense that they \emph{scatter} any nearby plane to be dense. However, there might still exist geodesic planes accumulating on the convex core boundary that are not dense, even with an exotic roof (and we expect a concrete example could be constructed using methods from \cite{exotic_plane}).

\paragraph{Horocycles}
In \cite{mcmullen2016horocycles}, classification of the behavior of every geodesic plane in acylindrical, convex cocompact hyperbolic 3-manifolds with totally geodesic boundary is used to classify the behavior of every horocycle as well. In the survey \cite{Oh+2022+506+566}, it was observed that to generalize this to hyperbolic 3-manifolds with quasifuchsian boundary, it is necessary to study the behavior of support and asymptotic planes, which we have done here. We hope to address the behavior of horocycles in future works.

\paragraph{Exotic rays}
In Theorem~\ref{thm: minimal_roof}, the behavior of a geodesic plane is different depending on the existence of exotic roofs. It is not obvious whether exotic roofs exist. As discussed above, the answer depends on the bending lamination of the end. As a matter of fact, exotic roofs are closely related to the existence of \emph{exotic rays} for the bending lamination. Before describing their connections, we pivot to Problem (II), which is formulated completely in terms of surfaces.

Let $X:=\Gamma\backslash\mathbb{H}^2$ be a complete hyperbolic surface of finite area, and $\mathcal{L}$ a measured geodesic lamination on $X$. By a geodesic ray on $X$ we mean a geodesic isometric immersion $r:[0,\infty)\to X$.

Define the \emph{intersection number} $I({\mathcal{L}},r)$ as the transverse measure of $r$ with respect to $\mathcal{L}$. For almost every ray $r$, $I({\mathcal{L}},r)$ is infinite (see Theorem~\ref{thm:linear}). On the other hand, $I({\mathcal{L}},r)$ is finite when:
\begin{itemize}[topsep=0mm, itemsep=0mm]
    \item $r$ is asymptotic to a leaf of $\mathcal{L}$ (see Proposition~\ref{prop:point_at_infinity}); or
    \item $r$ is eventually disjoint from $\mathcal{L}$.
\end{itemize}
We call a ray \emph{exotic for $\mathcal{L}$} if $I(\mathcal{L},r)$ is finite but it belongs to neither of these cases. Our question is then the following: given a measured geodesic lamination $\mathcal{L}$, does there exist an exotic ray for $\mathcal{L}$?

If $\mathcal{L}$ is a multicurve, i.e.\ $\mathcal{L}=a_1\gamma_1+\dots+a_k\gamma_k$
where $a_i \in \mathbb{R}^+$ and $\gamma_i$'s are simple closed geodesics, then there is no exotic ray for $\mathcal{L}$; any ray with $I(r,\mathcal{L})<\infty$ must eventually be disjoint from $\mathcal{L}$. Our main result Theorem~\ref{main1} addresses the remaining cases.

\paragraph{The halo of a measured lamination}
We now put Theorem~\ref{main1} into a broader context. Let $\mathcal{M}$ be a measured geodesic lamination on $\mathbb{H}^2$. Let $\partial\mathcal{M}\subset S^1$ be the set of endpoints of geodesics in $\mathcal{M}$. We can define \emph{exotic rays} for $\mathcal{M}$ as above. The \emph{halo} of $\mathcal{M}$, denoted by $h\mathcal{M}$, is the set of endpoints of exotic rays for $\mathcal{M}$. By definition, $h\mathcal{M}\cap \partial\mathcal{M}=\varnothing$. Moreover, if $p\in h\mathcal{M}$, then any geodesic ray ending at $p$ is exotic (Proposition~\ref{prop:point_at_infinity}). When $\mathcal{M}$ is the lift $\widetilde{\mathcal{L}}$ of a measured lamination $\mathcal{L}$ on $X$ to $\mathbb{H}^2$, we simply refer to $h\widetilde{\mathcal{L}}$ as the halo of $\mathcal{L}$, and denote it by $h\mathcal{L}$.

We have the following stronger version of Theorem~\ref{main1} framed in terms of halos:
\begin{thm}\label{mainprime}
Let $\mathcal{L}$ be a measured geodesic lamination on a complete hyperbolic surface $X=\Gamma\backslash\mathbb{H}^2$ of finite area. Then the halo $h\mathcal{L}$ is either empty or uncountable. Moreover, it is uncountable if and only if $\mathcal{L}$ is not a multicurve.
\end{thm}

\paragraph{A necessary condition for exotic roofs}
We now return to the study of exotic roofs. Let $\Omega$ be a Jordan domain in $\hat{\mathbb{C}}=\mathbb{C}\cup\{\infty\}$, and let $\Lambda$ be its boundary. Let $\hull(\Lambda)$ be the convex hull of $\Lambda$ in $\mathbb{H}^3$. Unless $\Lambda$ is a round circle, $\partial\hull(\Lambda)$ consists of two connected components, each isometric to $\mathbb{H}^2$ in the metric inherited from $\mathbb{H}^3$ and bent along a bending lamination (see e.g.\ \cite{bending}). Note that by taking $\Omega=\Omega_+$, the case of a quasifuchsian manifold considered above is included in the discussion.

Suppose $H$ is the component of $\partial\hull(\Lambda)$ on the side of $\Omega$, and $\mathcal{M}$ its bending lamination. The isometry $H\cong\mathbb{H}^2$ extends to a homeomorphism between $\Lambda$ and $S^1$. In particular, we may treat the halo $h\mathcal{M}$ as either a subset of $S^1$ or $\Lambda$. Exotic roofs and circles for $\Omega$ can also be defined analogously in this context. See Figure~\ref{fig:ellipse} for an example that does not arise as the limit set of a quasifuchsian manifold.
\begin{figure}[htp]
    \centering
    \captionsetup{width=.8\linewidth}
    \includegraphics{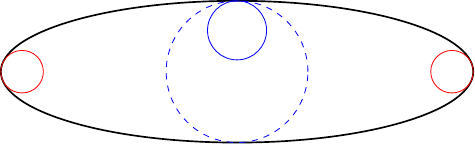}
    \caption{A domain bounded by an ellipse, its exotic circles (in red), and a non-exotic circle (in solid blue) shadowed by the dotted circle. The bending lamination $\mathcal{M}$ gives a foliation of $\mathbb{H}^2$ by geodesics, and $h\mathcal{M}=h'\mathcal{M}$ consists of two points.}
    \label{fig:ellipse}
\end{figure}

We define
$$h'\mathcal{M}:=\{p\in\Lambda:p=C\cap\Lambda\text{ for some exotic circle }C\}.$$
Using Gauss-Bonnet theorem, we can show:
\begin{thm}\label{prop:exotic_circle_necessary}
If $C$ is an exotic circle in $\Omega$, then the point $p=C\cap\Lambda$ is contained in the halo of the bending lamination $\mathcal{M}$. In other words, $h'\mathcal{M}\subset h\mathcal{M}$.
\end{thm}
Therefore exotic circles and roofs do not exist when the halo is empty. In particular, in the case of a quasifuchsian manifold $M$, if the bending lamination on one end is a multicurve, exotic circles do not exist on that side. This is compatible with the simple classification detailed in Proposition~\ref{prop: atomic}.

\paragraph{Leaf approximations and a sufficient condition}
In the case when the bending lamination $\mathcal{L}=\mathcal{L}_+$ is not a multicurve, Thoerem~\ref{mainprime} guarantees $h\mathcal{L}\neq\varnothing$. However, \emph{a priori}, $h'\mathcal{L}$ is only a subset of $h\mathcal{L}$. Our method of proving $h\mathcal{L}\neq\varnothing$ does not directly translate to show $h'\mathcal{L}\neq\varnothing$, so knowing only the former is not sufficient to conclude the latter at this point.

On the other hand, we are able to show $h'\mathcal{L}\neq\varnothing$ for laminations that are very well approximated by closed geodesics in some sense. More precisely, we introduce the notion of \emph{good leaf approximations} for laminations on a hyperbolic surface $X$. Essentially, a good sequence of leaf approximations is a sequence of closed curves $\gamma_n$ consisting of a long leaf segment of length $d_n$, whose end is connected by a short transverse segment of transverse measure $o(e^{-d_n})$. Such an approximating sequence is called \emph{exotic} if concatenating segments give an exotic ray. We have
\begin{thm}\label{thm: sufficient_intro}
Suppose the bending lamination $\mathcal{L}$ of an end $E_+$ of a quasifuchsian manifold has an exotic and good sequence of leaf approximations. Then $E_+$ contains uncountably many exotic roofs. Equivalently, the corresponding component $\Omega_+$ of the domain of discontinuity contains uncountably many orbits of exotic circles.
\end{thm}
In \S\ref{sec:sufficient}, we actually introduce the more general \emph{CESAG} sequences of leaf approximations and show that the existence of such a sequence is sufficient to generate exotic roofs. See \S\ref{sec:sufficient} and Theorem~\ref{thm: sufficient} for details. We remark that in the proof, it is really necessary to check the exotic property, for otherwise we may only get an asymptotic plane shadowed by a support plane.

The conditions on leaf approximations are somewhat competing: the exotic property requires the transverse segments to be small, which means that leaf segments have to be long; on the other hand, leaf segments cannot be too long compared to the transverse measure, as explained above. Balancing these two sides proves to be difficult. However, in the special case where the boundary of the end $X_+$ is isometric to a punctured torus, this is sometimes possible using rational approximation of irrational numbers, as we will explain now.

\paragraph{The case of punctured torus}
Let $X$ be a complete hyperbolic torus with one puncture. A measured lamination on $X$ is determined up to scale by its \emph{slope} $\theta\in\mathbb{P}H^1(X,\mathbb{R})\cong\mathbb{RP}^1$. Closed curves are identified with the rational points $\mathbb{QP}^1$. The correspondence is apparent in the flat picture: for a flat torus $\mathbb{T}_\tau:=\mathbb{C}/(\mathbb{Z}\oplus\mathbb{Z}\tau)$ (where $\tau\in\mathbb{H}$), a measured lamination of slope $\theta$ is foliated by straight lines in the direction of $1+\theta\tau$.

A fundamental domain of the action of $\pi_1(X)$ on $\mathbb{H}^2$ can be chosen as an ideal quadrilateral $Q$, as in Figure~\ref{fig:symbolic_coding}. Using the theory of boundary expansion (see e.g.~\cite{boundary_expansion, simple}), geodesics on $X$ can be coded in terms of a set of generators of $\pi_1(X)$. Among them, simple geodesics (i.e leaves of laminations) can be coded by Sturmian words (see e.g.\ \cite{sturmian1, sturmian2}, and for an exposition \cite{sturmian}). We refer to \S\ref{sec:tori} for more details.
\begin{figure}[htp]
    \centering
    \includegraphics[trim={0cm 0.5cm 0cm 0.5cm}, clip, width=0.4\linewidth]{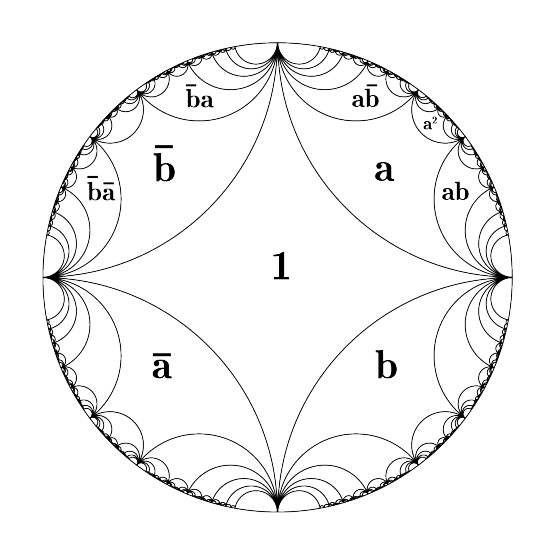}
    \caption{Symbolic coding with the quadrilateral $Q$ labelled $1$}
    \label{fig:symbolic_coding}
\end{figure}

Let $r=[c_0; c_1, c_2, \ldots]$ be an irrational number in terms of its continued fraction, and $p_k/q_k=[c_0; c_1, \ldots, c_k]$ its $k$-th convergent. We say $r$ is \emph{well-approximated} if there exists an increasing sequence $\{k_n\}$ of natural numbers so that for any constant $C>0$ we have $\exp(Cq_{k_n})/q_{k_n+1}\to 0$. Such irrational numbers exist; for example, we may take inductively $c_k=\lceil e^{q_{k-1}^2}\rceil$. For simplicity, we also call the corresponding lamination well-approximated. We have
\begin{thm}\label{thm:sewa}
Any well-approximated lamination on a punctured torus exhibits an exotic and good sequence of leaf approximations.
\end{thm}
As a consequence, we have:
\begin{cor}
Let $M$ be a quasifuchsian manifold whose convex core boundary is isometric to a pair of punctured tori $X_\pm$. If the bending lamination $\mathcal{L}_\pm$ on $X_\pm$ is given by a well-approximated irrational number, then there exist uncountably many exotic roofs in the corresponding end.
\end{cor}
We note that every lamination on a punctured torus can be realized as the bending lamination of some quasifuchsian manifold, by \cite[Th\'eor\`eme 1]{bonahon2004laminations}.

In Appendix~\ref{app:generic}, we show that well-approximated irrational numbers are generic in the sense of Baire category. Together with the corollary above, we have
\begin{cor}
Generically (in the sense of Baire category), a quasifuchsian manifold homotopic to a punctured torus has exotic roofs in both ends.
\end{cor}

\paragraph{Four-punctured spheres}
Similar results also hold for a four-punctured sphere, on which a measured lamination is again determined up to scale by its slope. As a matter of fact, the map $z\mapsto -z$ descends to an involution $\iota$ on a flat torus $\mathbb{T}_\tau$, and the quotient gives a sphere with four singular points. A lamination of slope $\theta$ on $\mathbb{T}_\tau$ then descends to a lamination of slope $\theta$ on the quotient. Essentially the same constructions in the proof of Theorem~\ref{thm:sewa} then give the analogous result for four-punctured spheres.

\paragraph{Higher genus}
The results stated above remain valid for any one-holed torus or four-holed sphere. Any compact surface of genus $g\ge2$ contains a one-holed torus (or a four-holed sphere) as a subsurface. By considering any well-approximated measured lamination supported on this one-holed torus (and note that again by \cite[Th\'eor\`eme 1]{bonahon2004laminations}, such a lamination is realizable as a bending lamination of some quasifuchsian manifold), we immediately have
\begin{cor}\label{cor:higher_genus}
For any $g\ge2$, there exist uncountably many quasifuchsian manifolds homotopic to a surface of genus $g$ containing uncountably many exotic roofs in one end. 
\end{cor}
As a matter of fact, any compact surface can be cut along simple closed curves into a union of one-holed tori and four-holed spheres. Each of these contains a well-approximated lamination. Moreover, with a different decomposition into tori and spheres and choices of well-approximated laminations, we can obtain two measured laminations that fill the surface (see Figure~\ref{fig:decomp}). By \cite[Th\'eor\`eme 1]{bonahon2004laminations} again, we have a quasifuchsian manifold that contains exotic roofs in both ends. Consequently, for such a quasifuchsian manifold, the closure of any support or asymptotic plane consists of all support and asymptotic planes in the same end.
\begin{figure}[htp]
    \centering
    \captionsetup{width=0.8\linewidth}
    \includegraphics[width=0.8\linewidth]{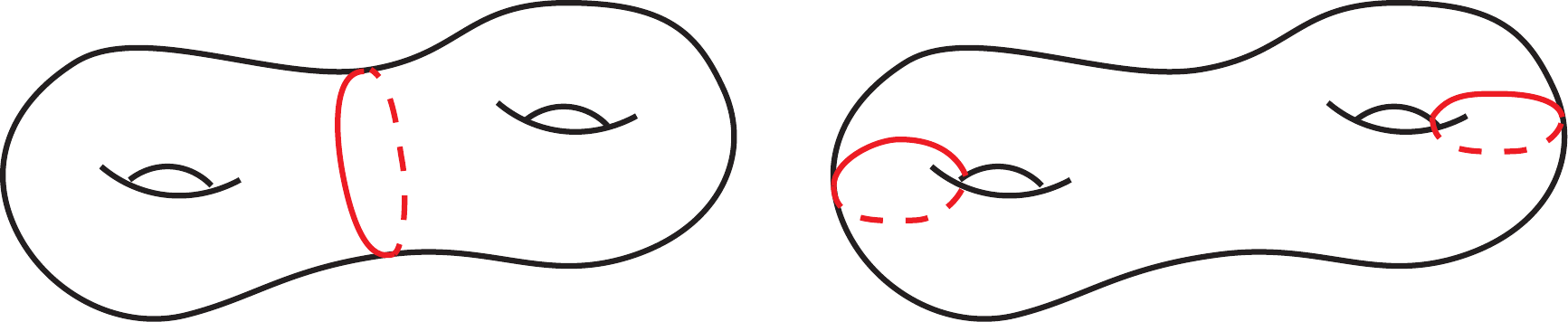}
    \caption{Two different decompositions of a genus 2 surface into one-holed tori and 4-holed spheres, so that choosing an irrational lamination on each component gives a pair of laminations whose union fills the surface.}
    \label{fig:decomp}
\end{figure}

We want to stress here that while we are able to give examples of quasifuchsian manifolds with exotic roofs in every genus, it remains unknown whether there is an example of non-atomic bending lamination without any exotic roof.

\paragraph{Questions}
Here are some unanswered questions that naturally arise from our discussions.
\begin{enumerate}[label=\normalfont{\arabic*.}, topsep=0mm, itemsep=0mm]
    \item Our construction of exotic rays makes use of the recurrence property of measured laminations on surfaces. In general, a measured lamination on $\mathbb{H}^2$ does not have such a property. For what measured laminations on $\mathbb{H}^2$ do exotic rays exist?
    \item How \emph{exotic} are exotic roofs? Or more precisely, does every end with an irrational bending lamination contains an exotic roof? We have shown the existence for well-approximated laminations, and such laminations are generic in the sense of Baire category for punctured tori. Does genericity hold in higher genera? Do exotic planes exist for \emph{badly} approximated laminations? For example, does the lamination with slope $\gamma=(\sqrt{5}+1)/2=[1;1,1,\cdots]$ (arguably the \emph{worst} approximated irrational number) on the punctured torus exhibit an exotic roof?
    \item More generally, for a bending lamination $\mathcal{M}$, how different are the halo $h\mathcal{M}$ and its subset $h'\mathcal{M}$ (the set of points where an exotic circle intersects the Jordan curve)? Question 2 above asks whether it is possible to have $h'\mathcal{M}=\varnothing$ while $h\mathcal{M}\neq\varnothing$, but it is interesting to understand the difference even when both are nonempty.
\end{enumerate}

\paragraph{Notes and references}
We list some previous works on the shape of convex core boundary and its bending lamination \cite{bending,bridgeman1998average,epstein2004quasiconformal,bridgeman2003bounds,bridgeman2005bounding,bridgeman2016improved}.

In \S\ref{sec:complementary} we also consider $I(\mathcal{L},r(t))$ as a function of $t$ for a geodesic ray $r:[0,\infty)\to \infty$ in general. We show that $I(\mathcal{L},r(t))$ is always sublinear (Proposition~\ref{Prop: linr}). Moreover, there exists a constant $c_{\mathcal{L}}$ such that for almost every geodesic ray we have $I(\mathcal{L},r(t))\asymp c_{\mathcal{L}}t$ (for more details see Theorem~\ref{thm:generic}). Finally, given any sublinear function $f$, we show there is a ray $r$ such that $I(\mathcal{L},r(t)) \asymp f(t)$ (see Theorem~\ref{thm: sublinear}).

The theory of Sturmian words has been generalized to the regular octagon and even all regular $2n$-gons, see \cite{octagon}. Our discussion in the case of punctured tori is also inspired and guided by \cite{simple}, where properties of simple words of a punctured hyperbolic surface are studied.

\paragraph{Outline of the paper}
The paper is organized into two parts. Part~\ref{part:halo} is focused on the side of surfaces. In \S\ref{sec:background}, we give an exposition of some closely related geometric and combinatorial objects on surfaces, including measured laminations, geodesic currents, and train tracks. In \S\ref{sec:entropy}, we study train paths and their symbolic coding, and show that two train paths are asymptotic at infinity if and only if their words share the same tail. For a measured lamination carried by a train track, we then construct closed train paths of arbitrarily small transverse measures whose words do not appear in the words for any leaves. Concatenating such words then gives exotic rays in \S\ref{sec:general_proof}. Here we also discuss some general properties of the halo. Finally, in \S\ref{sec:complementary}, we discuss the \emph{generic} behavior of the intersection number of a geodesic ray with a measured lamination, as a complement to \emph{exotic} behavior that has been the focus of the paper.

Part~\ref{part:qf} is focused on the side of quasifuchsian manifolds. In \S\ref{sec: quasifuchsian}, we review the description of convex core boundary of quasifuchsian manifolds following \cite{bending}. In \S\ref{sec:geodesic_planes_atomic} and \S\ref{sec:geodesic_planes_irrational}, we prove Proposition~\ref{prop: atomic} and Theorem~\ref{thm: minimal_roof} respectively, describing the behavior of geodesic planes outside convex cores. Parts of Theorem~\ref{thm: minimal_roof} assume the existence of exotic roofs, so before getting to the proof, we relate their existence to that of exotic rays. In \S\ref{sec: roof}, we show that the existence of exotic rays is necessary for the existence of exotic roofs, and in \S\ref{sec:sufficient}, we give a sufficient condition for existence, in terms of the hyperbolic surface underlying the convex core boundary and its bending lamination. Finally, in \S\ref{sec:tori}, we check that this condition is satisfied for any punctured torus with a well-approximated irrational lamination.

\paragraph{Acknowledgements}
We would like to thank C.~McMullen for his continuous support, enlightening discussions, and suggestions. Figures~\ref{fig:atomic_limit_set} and \ref{fig:symbolic_coding} were produced using his program \textbf{lim}. We would also like to thank K.~Winsor for helpful comments on a previous draft, and A.~Nolte for pointing out a mistake in the proof of Theorem~\ref{thm:generic}. Thanks also go to the anonymous referees, whose detailed comments and suggestions greatly improved the clarity of the exposition. Y.~Z.\ would like to acknowledge the support of Max Planck Institute for Mathematics, where part of the research was done.

\part{The halo of a measured lamination}\label{part:halo}
\section{Background on surfaces}\label{sec:background}
In this section, we briefly introduce some objects that will be used later in the proofs.

\paragraph{Measured laminations}
 Given a complete hyperbolic surface $X$ of finite area, a \emph{measured (geodesic) lamination} $\mathcal{L}$ is a compact subset of $X$ foliated by simple geodesics, together with a transverse invariant measure, which assigns a measure for any arc transverse to the lamination. The total mass of a transverse arc $\Lambda$ with respect to this measure is the \emph{intersection number} of the arc with $\mathcal{L}$, which we denote by $I(\mathcal{L},\Lambda)$.

A measured lamination $\mathcal{L}$ is called \emph{minimal} if every leaf is dense in $\mathcal{L}$. In general, $\mathcal{L}$ consists of finitely many minimal components. For a multicurve, every minimal component is a simple closed geodesic; the measure of any transverse arc is then simply a weighted count of intersections with $\mathcal{L}$.

\paragraph{Geodesic currents}
For the proof of Theorem~\ref{thm:linear}, we need some basic facts about geodesic currents. We refer to \cite{Bon.gc, Bon.Tch} for details.

Given a hyperbolic surface $X$, a \emph{geodesic current} is a locally finite measure on $T_1(X)$, invariant under the geodesic flow $\phi_t$ and the involution $\iota$ flipping the direction of the tangent vector.
As an example, the \emph{Liouville measure} $\lambda$ is a geodesic current, which locally decomposes as the product of the area measure on $X$ and the uniform measure of total mass $\pi/2$ in the bundle direction\footnote{We adopt the same normalization of $\lambda$ as \cite{Bon.gc}, so that $i(C,\lambda)=\ell(C)$ for any geodesic current $C$.}. Closed geodesics also give examples of geodesic currents: for a closed geodesic $\gamma$ of length $T$, given by two arclength parametrizations with opposite orientation $r_\pm:[0,T]\to X$, the associated geodesic current is then $((r'_+)_*dt+(r'_-)_*dt)/2$, the average of the pushforward of the Lebesgue measure $dt$ on $[0,T]$ under $r'_\pm:[0,T]\to T_1(X)$.

Let $\mathcal{C}(X)$ be the set of all geodesic currents on $X$. We can define notions of length and intersection number for geodesic currents, which extend the usual length and intersection number for closed geodesics. Indeed, the \emph{length} of a geodesic current $C$ is defined to be the total $C-$mass of $T_1(X)$ and is denoted by $\ell(C)$. The \emph{intersection number} of two geodesic currents $C_1$ and $C_2$, denoted by $i(C_1,C_2)$, is harder to define succinctly; for the precise definition see \cite{Bon.gc}. Here are some properties of $i(\cdot,\cdot)$ needed in the proof of Theorem~\ref{thm:generic}:
\begin{itemize}[topsep=0mm, itemsep=0mm]
    \item For any closed geodesics $\gamma_1$ and $\gamma_2$, $i(\gamma_1,\gamma_2)$ gives the number of times they intersect on $X$ with multiplicity.
    \item For any geodesic current $C$, $i(C,\lambda)=\ell(C)$.
    \item $i(\lambda,\lambda)=\ell(\lambda)=\pi^2(2g-2+n)$, and so $\lambda/(\pi^2(2g-2+n))$ is a probability measure on $T_1(X)$.
\end{itemize}

One fundamental fact concerning the intersection number is continuity. Let $K$ be a compact subset of $X$. We denote by $\mathcal{C}_K(X)$ the collection of geodesic currents whose support in $T_1(X)$ projects down to a set contained in $K$. Then we have \cite[\S4.2]{Bon.gc}
\begin{thm}[\cite{Bon.gc}]
For any compact set $K\subset X$, the intersection number
$$i:\mathcal{C}(X)\times\mathcal{C}_K(X)\to\mathbb{R}$$
is continuous.
\end{thm}

Finally, we remark that measured laminations also give geodesic currents. For a closed geodesic $\gamma$ and a measured lamination $\mathcal{L}$, two definitions of intersection number (total transverse measure $I(\mathcal{L},\gamma)$ and $i(\mathcal{L},\gamma)$) agree. Moreover, the set of measured laminations is equal to the \emph{light cone}, i.e.\ the set of geodesic currents with zero self-intersection number \cite[Prop 4.8]{Bon.gc}.

\paragraph{Measured train tracks}
A \emph{train track} $\tau$ is an embedded $1$-complex on $X$ consisting of the set of vertices $V_{\tau}$ (which we call \emph{switches}) and the set of edges $E_{\tau}$ (which we call \emph{branches}), satisfying the following properties:
\begin{itemize}[topsep=0mm, itemsep=0mm]
    \item Each branch is a smooth path on $X$; moreover, branches are tangent at the switches.
    \item Every connected component of $\tau$ that is a simple closed curve has a unique switch of degree two. All other switches have degree at least three. At each switch $v$ if we fix a compatible orientation for branches connected to $v$, there is at least one incoming and one outgoing branch.
    \item For each component $C$ of $X-\tau$, the surface obtained from doubling $C$ along its boundary $\partial C$, has negative Euler characteristic if we treat non-smooth points on the boundary as punctures.
\end{itemize}
A \emph{measured train track} $(\tau,\omega)$ is a train track $\tau$ and a weight function on edges $\omega:E_\tau\to\mathbb{R}_{\ge0}$ satisfying the following equation for each switch $v\in V_\tau$:
$$\omega(e_1)+\dots\omega(e_i)=\omega(e'_1)+\dots+\omega(e'_j)$$
where $e_1,\dots,e_i$ are incoming branches at $v$ and $e'_1, \dots, e'_j$ are outgoing branches, for a fixed compatible orientation of branches connected to $v$. Note that the equation does not change if we choose the other compatible orientation of the branches.

It is well-known that each measured train track $(\tau,\omega)$ corresponds to a measured lamination $\lambda_{\tau}$ on the hyperbolic surface $X$; see e.g.\ \cite[Construction~1.7.7]{com.tr}.  

Let $\mathcal{I}$ be an interval. A \emph{train path} is a smooth immersion $r:\mathcal{I}\to\tau\subset X$ starting and ending at a switch. We say a ray (or a multicurve, or a train track) $\gamma$ is \emph{carried} by $\tau$ if there exists a $C^1$ map $\phi: X \rightarrow X$ homotopic to the identity so that $\phi(\gamma)\subset\tau$ and the differential $d\phi_p$ restricted to the tangent line at any $p$ on the ray (or a multicurve, or a train track) is nonzero. In the case of a ray or an oriented curve, the image under $\phi$ is a train path of $\tau$, and in the case of a train track $\tau'$, $\phi$ maps a train path of $\tau'$ to a train path of $\tau$. We say a geodesic lamination $\mathcal{L}$ is carried by $\tau$ if every leaf of $\mathcal{L}$ is carried by $\tau$.

From the constructions in \S1.7 of \cite{com.tr}, we can approximate a minimal non-atomic measured lamination $\mathcal{L}$ by a sequence of birecurrent measured train tracks $(\tau_n, \omega_n)$. These train tracks are deformation retracts of smaller and smaller neighborhoods of $\mathcal{L}$.

Moreover, we can choose a sequence of approximating train tracks so that they are connected and satisfy the following properties: 
\begin{itemize}[topsep=0mm, itemsep=0mm]
\item They are not simple closed curves;
\item $\tau_{n+1}$ is carried by $\tau_n$ for $n \geq 1$;
\item $\omega_n$ of each branch is a positive number less than $1/{2^n}$;
\item $\mathcal{L}$ is carried by $\tau_n$ for $n \geq 1$.
\end{itemize}
For more information on train tracks, see \cite{com.tr}.

\section{Train tracks and symbolic coding}\label{sec:entropy}
In this section, we study train paths carried by a train track. These train paths are naturally coded by the branches it passes through. In particular, we give a criterion in terms of the coding when two train paths are asymptotic at infinity. We also construct finite train paths of arbitrarily small transverse measures whose coding words do not appear in the leaf of a fixed measured lamination.

\subsection{Symbolic coding of points at infinity}
In this subsection, we use symbolic coding of train tracks to prove Theorem~\ref{thm:tail2}, which gives a criterion for the convergence of two train paths. The result is likely known (or at least unsurprising) to experts. However, we were unable to pinpoint a specific reference, so we included a detailed proof here.

Let $\tau=(V_\tau,E_\tau)$ be a train track carrying $\mathcal{L}$. Suppose $|E_\tau|=d$ and label the branches of $\tau$ by $b_1^{\tau}, \dots, b_d^{\tau}$. By listing the branches it traverses, we can describe a (finite or infinite) train path $r$ of $\tau$ by a (finite or infinite) word $w_{\tau}(r)$ with letters in the alphabet $B^\tau:=\{ b_i^{\tau} , 1 \leq i \leq d\}$. We can also assign such a word to a ray or a curve carried by $\tau$ if we consider the corresponding train path. 

Assume further $\tau$ is connected and birecurrent. Let $\tilde\tau=\pi^{-1}(\tau)\subset\mathbb{H}^2$. We say a point $p$ on the circle at infinity $S^1_\infty$ is \emph{reached} by a train path of $\tau$ if some lift of the path to $\tilde\tau$ converges to $p$. Two infinite train paths are said to converge at infinity if some lifts of the paths to $\tilde\tau$ reach the same point at infinity. We have
\begin{thm}\label{thm:tail2}
Two train paths of a train track $\tau$ converge at infinity if and only if the corresponding words have the same tail. 
\end{thm}
\begin{proof}
One direction is obvious. For the other direction, let $X_1$ be the smallest subsurface containing $\tau$. Suppose $r_1$ and $r_2$ are two converging train paths. By definition, there exist lifts $\widetilde{r_1}$ and $\widetilde{r_2}$ of $r_1$ and $r_2$ to $\widetilde\tau$ so that they converge to the same point $Q$ at infinity. Assume the starting point of $\widetilde{r_i}$ is $P_i$ for $i=1,2$. We view $\widetilde{r_i}$ as an \emph{oriented} path from $P_i$ to $Q$. 

 It is easy to see that, in fact, $\widetilde{r_1}$ and $\widetilde{r_2}$ are contained in a single connected component $\widetilde{X_1}$ of $\pi^{-1}(\Int(X_1))$. We aim to prove that $\widetilde{r_1}$ and $\widetilde{r_2}$ share the same branches after a point. 
\begin{figure}[htp]
    \centering
    \includegraphics[width=0.46\linewidth]{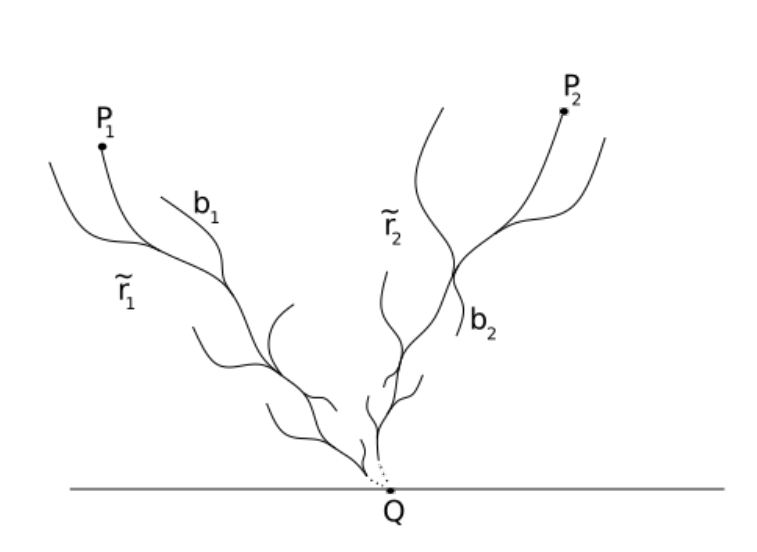}
    \caption{Converging train paths}
    \label{tr1}
\end{figure}

Suppose otherwise. We will repeatedly use the following fact: there is no embedded bigon in $\mathbb{H} \cup S_{\infty}^1$ whose boundary is contained in $r_1 \cup r_2$ \cite[Prop.\ 1.5.2]{com.tr}. We prove the following sequence of claims:
\begin{clmA} 
$\widetilde{r_1}$ and $\widetilde{r_2}$ are disjoint.
\end{clmA}
Indeed, between any intersection and $Q$, $\widetilde{r_1}$ and $\widetilde{r_2}$ bound a nonempty bigon. Therefore, we may assume $\widetilde{r_2}$ is on the left of $\widetilde{r_1}$ as we go toward $Q$ (see Figure~\ref{tr1}).

\begin{clmA}
We can connect $P_1$ and $P_2$ in $\widetilde{X_1}$ by a sequence of branches of $\widetilde{\tau}$.
\end{clmA}
Let $\widetilde{\tau}_*$ be the connected component of $\widetilde\tau$ containing $P_1$. The geodesic segment $\overline{P_1P_2}$ projects to a geodesic segment contained in $X_1$. Each complementary region of $\tau$ in $X_1$ is a disk, a cylinder with one boundary component in $\partial X_1$, or a cylinder with two boundary components formed by branches of $\tau$, at least one of them non-smooth. If the projection of $\overline{P_1P_2}$ intersects the core curve of a cylinder of this last type, then $Q$ must be the endpoint of a lift of the core curve. Since we can represent train paths by smooth curves whose geodesic curvature is uniformly bounded above by a small constant, we may assume $\widetilde{r_1}$ and $\widetilde{r_2}$ are within bounded distance from each other. Now both $r_1$ and $r_2$ are recurrent, so we obtain an immersed cylinder in $X_1$ whose boundary components are train paths. But this lifts to a bigon with both vertices at infinity.

Therefore, we can homotope the projection of $\overline{P_1P_2}$ rel endpoints to a (possibly non-smooth) path formed by the branches of $\tau$. This lifts to a path between $P_1$ and $P_2$ formed by branches of $\widetilde{\tau}_*$. Note that this may not necessarily be a train path. From now on, by $\overline{P_1P_2}$ we mean this path consisting of branches.

Let $b$ be a branch attached to $\widetilde{r_i}$. Then $b$ is called an \emph{incoming branch} (resp.\ \emph{outgoing branch}) if it is smoothly connected to the tail (resp.\ head) of a branch of $\widetilde{r_i}$ (recall that $\widetilde{r_i}$ is oriented, and so are its branches). In Figure~\ref{tr1}, for example, $b_1$ is an incoming branch and $b_2$ is an outgoing branch. 

\begin{clmA} \label{inf}
There are either infinitely many branches attached to $\widetilde{r_1}$ on its left, or infinitely many branches attached to $\widetilde{r_2}$ on its right.
\end{clmA}
Suppose otherwise. Then there is no branch on the left of $\widetilde{r_1}$ nor on the right of $\widetilde{r_2}$ after a point. This implies that the projection of $\widetilde{r_i}$ to $X_1$ after that point is a closed curve $\gamma_i$. Moreover, $\gamma_1$ and $\gamma_2$ must be homotopic. The region bounded between them is a cylinder with smooth boundary, which cannot appear for a train track.

Without loss of generality, we may assume there are infinitely many branches attached to $\widetilde{r_1}$ on its left. Let $T \subset \mathbb{H} \cup S_{\infty}^1$ be the region bounded by $\widetilde{r_1}, \widetilde{r_2}$ and $\overline{P_1P_2}$. Given a branch $b$ inside $T$ attached to $\widetilde{r_1}$, we can extend it (from the end not attached to $\widetilde{r_1}$) to a train path $r_b$ until we hit the boundary of $T$. We have the following two cases.

\noindent\textbf{Case 1.}\quad There are infinitely many incoming branches. 
    
Given an incoming branch $b$, if $r_{b}$ hits $\widetilde{r_2}$, then $r_b$ and portions of $\widetilde{r_1}$ and $\widetilde{r_2}$ form a bigon with $Q$ as a vertex (see Figure~\ref{bgn1.a}), which is impossible. Moreover, given two incoming branches $b_1, b_2$, we must have $r_{b_1} \cap r_{b_2}= \varnothing$, for otherwise we would have a bigon (see Figure~\ref{bgn1.b}). 
\begin{figure}[htp]
\begin{subfigure}{.33\textwidth}
  \centering
  \includegraphics[width=.8\linewidth]{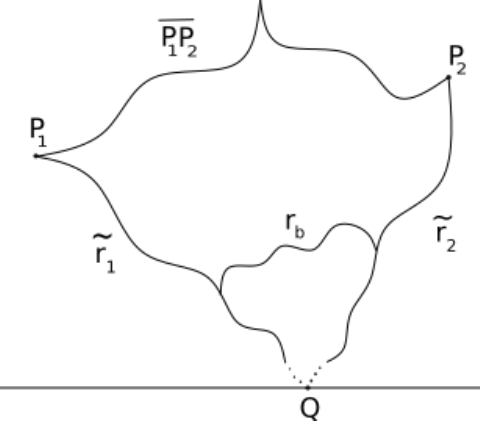}  
  \caption{}
  \label{bgn1.a}
\end{subfigure}
\begin{subfigure}{.33\textwidth}
  \centering
  \includegraphics[width=.8\linewidth]{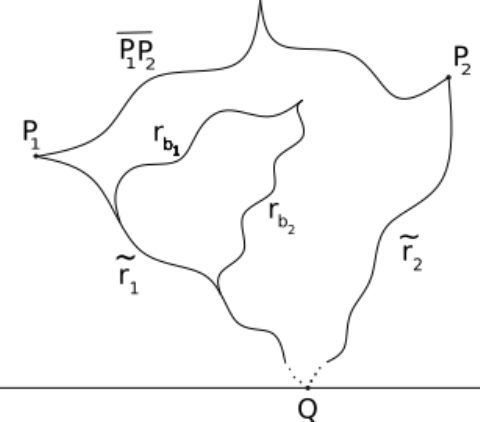}  
  \caption{}
  \label{bgn1.b}
\end{subfigure}
\begin{subfigure}{.33\textwidth}
  \centering
  \includegraphics[width=.8\linewidth]{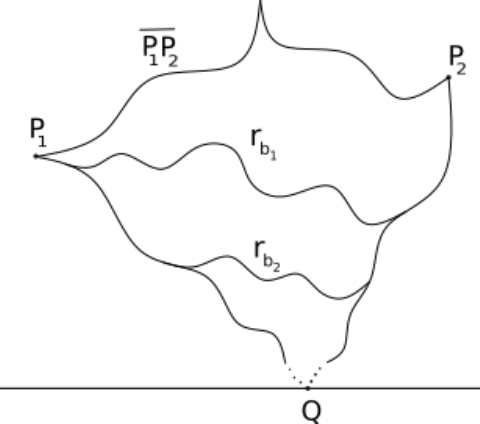}  
  \caption{}
  \label{bgn2}
\end{subfigure}
\caption{Different types of bigons}
\label{bgn1}
\end{figure}

Therefore all the extended train paths hit $\overline{P_1P_2}$. But this is impossible, as $\overline{P_1P_2}$ consists of finitely many branches.

\noindent\textbf{Case 2.}\quad There are infinitely many outgoing branches.

The arguments are similar. Given two outgoing branches $b_1, b_2$, we must have $r_{b_1} \cap r_{b_2}= \varnothing$, for otherwise we have a bigon. Moreover, if $r_{b_1}$ and $r_{b_2}$ both hit $\widetilde{r_2}$ then either they bound a bigon (see Figure~\ref{bgn2}) or one of them bounds a bigon with $Q$ as a vertex. Therefore, similar to the previous case, all the extended train paths hit $\overline{P_1P_2}$, again impossible.
\end{proof}
Consequently, for every point $Q$ at infinity that is reachable by $\tau$, we may assign a symbolic coding by choosing any train path reaching $Q$. This coding is only well-defined up to the equivalence relation of having the same tail and is $\pi_1(X)$-invariant (i.e.\ the equivalence class of words we can assign for $Q$ is the same as that of $\gamma\cdot Q$ for any $\gamma\in\pi_1(X)$). This is very much reminiscent of the classical cutting sequences for the modular surface, or boundary expansion in general (see e.g.\ \cite{simple}).

\subsection{Entropy and closed curves carried by a train track}\label{sec:carry}
In this subsection, we show that there exist inadmissible words with arbitrarily small transverse measure.

Let $\mathcal{L}$ be a minimal non-atomic measured lamination and  $\{\tau_n\}$ a sequence of train tracks approximating $\mathcal{L}$ as described in \S\ref{sec:background}. Note that each branch of $\tau_n$ has positive transverse measure $<1/2^n$. For each $n$, we fix a homotopy $\phi_n$ sending $\tau_n$ to a subset of $\tau_1$. Via $\phi_n$, each closed train path carried by $\tau_n$ determines a closed train path of $\tau_1$. The itinerary of the path gives, up to cyclic reordering, a word with letters in $B^{\tau_1}$, which we call the word of the path. Words of leaves can be defined similarly. A word (and the corresponding train path) is said to be \emph{admissible} if it is a subword of the word of some leaf of $\mathcal{L}$. Otherwise, it is called \emph{inadmissible}.

We emphasize that to make sure words of train paths from different train tracks are comparable, we always \emph{take the path to $\tau_1$ and record the itinerary there}, with letters in $B^{\tau_1}$.

Let $\mathcal{A}_\tau(T)$ be the set of admissible words of length at most $T$. We have:

\begin{prop}\label{prop:ent.poly}
The number $|\mathcal{A}_\tau(T)|$ of admissible words of length $\le T$ has polynomial growth in $T$. More precisely, $$|\mathcal{A}_\tau(T)|\le 4|E_{\tau}|^2T^{|E_{\tau}|-|V_\tau|+2}.$$
\end{prop}

\begin{proof}
Consider a small neighborhood of $\tau$ on the surface. Then a simple train path of $\tau$ can be represented by a path in this neighborhood with no self-intersection. From this representation, we can describe each simple train path $\gamma$ of $\tau$ by the following information:
\begin{itemize}[topsep=0mm, itemsep=0mm]
    \item A function $f: E_{\tau} \to \mathbb{N}\cup\{0\}$ that assigns to each branch $e$ the number of times $\gamma$ passes through $e$;
    \item Two branch endpoints $P,Q$. These points present the endpoints of $\gamma$. $P$ specifies a vertex and a branch connecting to that vertex. The branch is the one that contains the subarc of $\tau$ with endpoint $P$. Similarly, $Q$ is chosen.
    \item Numbers $0<s,k\le T$. Consider subarcs of $\gamma$ in one branch's neighborhood. Since the subarcs are not intersecting, we can number them from $1$ to $r$ (r is the number of subarcs in the branch). Based on their order in the rectangular neighborhood of the branch. The numbers $s,k$ are assigned to the endpoints of $\gamma$.
\end{itemize}   
We can uniquely construct $\gamma$ from these data (if any $\gamma$ exists with these properties) as follows. For each $e$, we consider $f(e)$ parallel arcs in the neighborhood of $e$. At each vertex, we connect all incoming and outgoing arcs, except the two endpoints $P,Q$ of $\gamma$, without making any intersection. Therefore, there is a unique way to connect them. See Figure \ref{fig: traintrack.connect.arcs}. 

\begin{figure}[htp]
    \centering
    \includegraphics{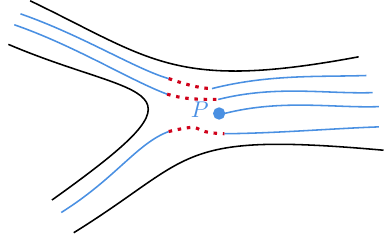}
    \captionsetup{width=.8\linewidth}
    \caption{Function $f$ assigns numbers $1,2,4$ to these arcs. $P$ is on the third arc inside the right branch, so $s=3$. The dotted lines show how uniquely we can connect arcs with no self-intersection. }
    \label{fig: traintrack.connect.arcs}
\end{figure}

Therefore, the map from $\mathcal{A}_{\tau}(T)$ to the set of multiples $(f, P,Q,s,k)$ is injective. The cardinality of this set of multiples is at most (the number of functions $f$)$\times 2|E_\tau|\times 2|E_\tau| \times T \times T$. 

The function $f$ satisfies the following equation for each switch $v\in V_\tau$:
\begin{equation}\label{f.eq}
f(e_1)+\dots+f(e_i)-f(e'_1)-\dots-f(e'_j)=m
\end{equation}
where $e_1,\dots,e_i$ are incoming branches at $v$ and $e'_1, \dots, e'_j$ are outgoing branches, for a fixed compatible orientation of branches connected to $v$ and $m \in \{ -2,-1,0,1,2 \}$ depending on the number of endpoints we have on these branches (on the side connected to $v$).

Note that the value of $f$ on each branch is less than $T$. Equations \ref{f.eq} are independent (see \cite[Lamma 2.1.1]{com.tr}). Therefore, the number of functions $f$ satisfying these properties is at most $T^{|E_{\tau}|-|V_{\tau}|}$.
\end{proof}

 Let $b_n$ be a branch of $\tau_n$, and $\mathcal{K}_{b_n}$ the set of finite words of closed train paths containing $b_n$. Its subset $\mathcal{K}_{b_n}(T) \subset \mathcal{K}_{b_n}$ contains words of length $<T$.
 \begin{lm}\label{lm:2admisible}
 There are at least two admissible words in $\mathcal{K}_{b_n}$ whose train paths are first return paths from $b_n$ to $b_n$.
 \end{lm}
 
\begin{proof}
Equivalently, we want to show that admissible words in $\mathcal{K}_{b_n}$ are not generated by a single word $w$. Assume by contradiction that they are. Then the only possible infinite admissible word is a concatenation of copies of $w$, since every leaf of $\mathcal{L}$ is recurrent and comes back to $b_n$ infinitely many times. By Theorem~\ref{thm:tail2}, all leaves converge to the closed geodesic that corresponds to $w$, a contradiction! 
\end{proof}

\begin{prop} \label{prop:ent.exp}
 The number $|\mathcal{K}_{b_n}(T)|$ of closed train paths containing $b_n$ with length $<T$ has exponential growth in $T$.
\end{prop}
\begin{proof}
Let $a_1, \dots, a_k$ be the words of admissible first return paths from $b_n$ to $b_n$ in $\tau_n$. Note that $k\ge2$ by Lemma~\ref{lm:2admisible}. Any words constructed from these blocks are in $\mathcal{K}_{b_n}$. Therefore, $|\mathcal{K}_{b_n}(T)|> k^{cT}$ for some constant $c$.
\end{proof}

\begin{cor} \label{cor:inadmissible}
There is an inadmissible word $w \in \mathcal{K}_{b_n}$ with transverse measure $< 1/2^n$. 
\end{cor}
\begin{proof}
Let $a_1, \dots, a_k$ be as above. By Propositions~\ref{prop:ent.poly} and \ref{prop:ent.exp}, there must be an inadmissible word among words constructed from the blocks $a_1,\ldots,a_k$.
 
Let $w_n=a_{i_1}a_{i_2}\dots a_{i_s}$ be an inadmissible word constructed from the blocks with the smallest number of blocks. Then $a_{i_1}a_{i_2}\dots a_{i_{s-1}}$ is admissible, and thus can be represented by a segment of a leaf. The word $a_{i_s}$ can also be represented by a segment of a leaf; connecting these two segments possibly creates a crossing over the branch $b_n$. Thus what we obtain is a segment $l_n$ representing $w_n$ with transverse measure less than $1/2^n$, as the crossing happens in a branch of $\tau_n$.
\end{proof}

\paragraph{Interpretation by entropy}
As a punchline for the results in this section, we note that the entropy of leaves of a measured lamination is zero while the entropy of train paths is positive. More precisely, let $\mathcal{A}$ be the set of bi-infinite words of leaves, and $\mathcal{W}_n$ the set of bi-infinite words constructed from $a_1,\ldots,a_k$ as in the proof of Proposition~\ref{prop:ent.exp}. Then together with the shift map $\sigma$ on words, $\mathcal{A}$ and $\mathcal{W}_n$ are subshifts of finite type. The entropy of a subshift of finite type is given by the exponential growth rate in $T$ of the number of words of length $T$ (see e.g.\ \cite[Prop.~3.2.5]{katok1997introduction}). Thus $(\mathcal{A},\sigma)$ has zero entropy while $(\mathcal{W}_n,\sigma)$ has positive entropy $h_n>0$ by Propositions~\ref{prop:ent.poly} and \ref{prop:ent.exp} respectively. Moreover, in fact, $h_n\to 0$ as $n \to \infty$, since $\mathcal{A}=\bigcap_{n=1}^\infty\mathcal{W}_n$.

\section{The halo of a measured lamination}\label{sec:general_proof}

In this section, we discuss some properties of the halo of a measured lamination, and give a proof of the main result Theorem~\ref{mainprime}.

The following lemma showcases the difficulty of distinguishing exotic rays from leaves: asymptotically, the exotic ray stays very close to a leaf for longer and longer time.
\begin{lm}\label{lem:angle}
    Let $X$ be a hyperbolic surface of finite area, and $\mathcal{L}$ a measured lamination on $X$. Suppose $r:[0,\infty)\to X$ is an exotic ray for $\mathcal{L}$. If $r$ intersects a leaf at the point $r(t)$, let $\theta(t)$ be the angle formed by the ray and the leaf. Then $\theta(t)\to0$ as $t\to\infty$.
\end{lm}
Note that by definition, $r$ intersects the leaves recurrently, so the statement above is not vacuous.
\begin{proof}
    Fix $\epsilon>0$. Let $\mathcal{E}_\epsilon$ be the set of all unit tangent vectors based at a point on a leaf of $\mathcal{L}$ forming an angle $\ge\epsilon$ with the tangent direction of the leaf. Since $\mathcal{L}$ is compact, this set is also compact.
    
    For any $v\in T_1(X)$, let $\phi_v:[-1,1]\to X$ be a geodesic segment with $\phi_v'(0)=v$. Then the transverse measure of the segment $I(v):=I(\mathcal{L},\phi_v([-1,1]))$ is continuous in $v$. Hence $I(v)$ reaches its minimum on $\mathcal{E}_\epsilon$. Clearly for any vector $v$ in $\mathcal{E}_\epsilon$, $I(v)>0$, so we conclude that the minimum of $I(v)$ over $\mathcal{E}_\epsilon$ is strictly positive.

    If $\theta(t)$ does not tend to zero as $t\to\infty$, then there exists an $\epsilon>0$, a infinite sequence $t_i\to\infty$ so that $r$ intersects a leaf at $r(t_i)$ and $r'(t_i)\in\mathcal{E}_\epsilon$. But then the total transverse measure of $r$ goes to infinity, a contradiction!
\end{proof}

Recall that given a measured lamination $\mathcal{M}$ of $\mathbb{H}^2$, the halo of $\mathcal{M}$, denoted by $h\mathcal{M}$ is the subset of $S^1$ consisting of endpoints of exotic rays. The following proposition states that \emph{any} geodesic ray ending in $h\mathcal{M}$ is in fact exotic:

\begin{prop}\label{prop:point_at_infinity}
Suppose $r_1:[0,\infty)\to \mathbb{H}$ is a piecewise smooth ray and $r_2$ a geodesic ray. Assume further $r_1$ and $r_2$ reach the same point at infinity. If $I(\mathcal{M},r_1)<\infty$ then $I(\mathcal{M},r_2)<\infty$.
\end{prop}
\begin{proof}
Let $r_0$ be the (finite) geodesic segment between $r_2(0)$ and $r_1(0)$. Consider leaves intersecting $r_2$ but not $r_1$. Since $r_2$ is geodesic, each of these leaves intersects $r_2$ only once, and must also intersect $r_0$. In particular, $I(\mathcal{M},r_2)\le I(\mathcal{M},r_1)+I(\mathcal{M},r_0)<\infty$, as desired.
\end{proof}

As a consequence, if $r_1$ and $r_2$ are asymptotic geodesic rays, then $I(\mathcal{M},r_1)<\infty$ if and only if $I(\mathcal{M},r_2)<\infty$. In particular, any geodesic ray asymptotic to a leaf of $\mathcal{M}$ has finite intersection number, as any leaf has intersection number $0$.

When $\mathcal{M}=\widetilde{\mathcal{L}}$, the lift of a measured lamination $\mathcal{L}$ on $X$ to $\mathbb{H}^2$, we have the following consequence of the proposition above:
\begin{cor}\label{cor:dense}
The set of exotic vectors in $T_1(X)$ is empty or dense.
\end{cor}
\begin{proof}
The halo $h\mathcal{L}$ is $\pi_1(X)$ invariant, so if it is nonempty, it is also dense. Assume it is nonempty. For any $p\in X$, choose a covering map $\mathbb{H}^2\cong\mathbb{D}\to X$ so that a lift of $p$ is at the origin in the unit disk $\mathbb{D}$. The set of tangent vectors at the origin pointing towards a point in $h\mathcal{L}$ is dense. This is enough for our desired result.
\end{proof}
Another consequence is that varying the hyperbolic structure will not change the fact that $I(\mathcal{L},r)<\infty$, after straightening $r$ in the new hyperbolic structure. In particular, we may define the \emph{halo} of a measured foliation, an entirely topological object, by endowing the surface with any hyperbolic structure.

We are now in a position to prove the main result Theorem~\ref{mainprime}:
\begin{proof}[Proof of Theorem~\ref{mainprime}]
If $\mathcal{L}$ is a multicurve, it is easy to see that $h\mathcal{L}=\varnothing$. On the other hand, if $\mathcal{L}$ is not a multicurve, then it has an infinite minimal component. An exotic ray for this component that is also bounded away from the other components is clearly an exotic ray for $\mathcal{L}$. To prove Theorem~\ref{mainprime}, it is thus sufficient to construct exotic rays that remain in a small neighborhood of this minimal component of $\mathcal{L}$. Let $\tau_n$ be a sequence of train tracks approximating this minimal component of $\mathcal{L}$ as described in \S\ref{sec:background}. We may further assume $\tau_1$ is disjoint from the other components of $\mathcal{L}$. The rays we construct come from train paths of $\tau_1$ and are thus disjoint from the other components. From now on, for simplicity, we may safely assume $\mathcal{L}$ is minimal.

We construct an exotic ray by gluing a sequence of inadmissible words of $\tau_1$, the sum of whose transverse measures is finite. We fix a homotopy $\phi_n$ which sends $\tau_n$ to a subset of $\tau_1$ and $\phi_n(\tau_n) \subset \phi_{n-1}(\tau_{n-1})$.

Choose a branch $b_n \in E_{\tau_n}$, so that under the map from $\tau_n$ to $\tau_{n-1}$, a portion of $b_n$ is mapped to $b_{n-1}$.
From Corollary \ref{cor:inadmissible}, there is a finite inadmissible train path of $\tau_n$ starting and ending at $b_n$, which can be represented by a segment $l_n$ of transverse measure $<1/2^n$.

If we glue the segments $l_n$ for $n\ge1$ together, we have a ray of finite transverse measure. Indeed, since both $l_n$ and $l_{n-1}$ can be represented by train paths starting and ending in $b_{n-1}$, connecting them creates a transverse measure of at most $1/2^{n-1}$. So the total transverse measure $<2\sum1/2^{n}<\infty$. This ray is exotic. Indeed, by construction the ray intersects the lamination infinitely many times; moreover, its word $w_1w_2w_3\cdots$ contains inadmissible subwords in any tail, and by Theorem~\ref{thm:tail2}, this ray does not converge to any leaf of $\mathcal{L}$.

Finally, by passing to a subsequence, we may assume the transverse measures of $l_n$ satisfy $I(\mathcal{L},l_n)/3>\sum_{n+1}^\infty I(\mathcal{L},l_k)$. For any subsequence $\{n_k\}$ of $\mathbb{N}$, we may glue the segments $l_{n_k}$ instead. Note that $w_{n_1}w_{n_2}\cdots$ and $w_{n'_1}w_{n'_2}\cdots$ have the same tail if and only if the sequences $\{n_k\}$ and $\{n'_k\}$ have the same tail. Indeed, our assumption $I(\mathcal{L},l_n)/3>\sum_{n+1}^\infty I(\mathcal{L},l_k)$ means that if the two sequences have different tails, the tails of the corresponding words have different transverse measure. Given any subsequence $\{n_k\}$, there are at most countably many subsequences of $\mathbb{N}$ having the same tail. So there are uncountably many of subsequences not having the same tail, implying that there are uncountably many exotic rays with different endpoints.
\end{proof}
\begin{rmk}
It is certainly possible to directly construct an exotic ray using \emph{leaf approximations} to be introduced in Part~\ref{part:qf}. However, the coding from train tracks provides a systematic way of determining asymptotic behavior, which is useful for proving uncountability above.
\end{rmk}

\section{Complementary results}\label{sec:complementary}
In this section, we study the general behavior of intersection of a geodesic ray with $\mathcal{L}$ and prove Theorems \ref{thm:linear} and \ref{thm: sublinear}.

Given a geodesic ray $r:[0,\infty)\to X$, let $I_{\mathcal{L},r}(t)$ be the intersection number of the geodesic segment $r([0,t])$ with $\mathcal{L}$. Clearly, $I_{\mathcal{L},r}(t)$ is a non-negative and non-decreasing continuous function of $t$. Given two non-negative functions $f,g$ defined on $[0,\infty)$, we write $f=O(g)$ if there exist constants $t_0,c>0$ so that $f(t)\le cg(t)$ for all $t\ge t_0$, and write $f\asymp g$ if $f=O(g)$ and $g=O(f)$. We have:
\begin{thm}\label{thm:linear}
For any measured lamination $\mathcal{L}$ and geodesic ray $r$ on $X$, $I_{\mathcal{L},r}(t)=O(t)$. Moreover, for almost every choice of initial vector $r'(0)$ with respect to the Liouville measure, $I_{\mathcal{L},r}(t)\asymp t$. More precisely, $\lim_{t \rightarrow \infty} I_{\mathcal{L},r}(t)/{t}= \ell(\mathcal{L})/(\pi^2(2g-2+n))$. 
\end{thm}
Here $\ell(\mathcal{L})$ is the \emph{length} of $\mathcal{L}$; for a precise definition see \S \ref{sec:background}. This theorem implies that the set of exotic vectors (i.e.\ initial vectors of exotic rays) has measure zero. However, recall that it is dense in $T_1(X)$ by Corollary~\ref{cor:dense}.

We also show that any growth rate between finite and linear is achieved:
\begin{thm}\label{thm: sublinear}
Suppose the measured lamination $\mathcal{L}$ is not a multicurve. Given a continuous increasing function $f:\mathbb{R}_+\to\mathbb{R}_+$ such that $f(0)=0$ and $\lim_{t\to\infty}f(t)/t=0$, there exists a geodesic ray $r:\mathbb{R}_+\to X$ such that $I_{\mathcal{L},r}\asymp f$.
\end{thm}
We remark that this does not imply Theorem~\ref{main1}, as here we only guarantee the existence of a geodesic ray with a specific growth rate, \emph{without} any restriction on its endpoint.

We start with the first part of Theorem~\ref{thm:linear}:
\begin{prop}\label{Prop: linr}
For any geodesic ray $r:[0,\infty)\to X$, $I_{\mathcal{L},r}(t)=O(t)$.
\end{prop}
\begin{proof}
If $X$ has cusps, there exists a cuspidal neighborhood for each cusp so that the measured lamination is contained in $X^\circ$, the complement of these neighborhoods. Note that any geodesic segment contained in these neighborhoods has zero transverse measure. If $X$ has no cusp, we set $X^\circ =X$.

The set of geodesic segments on $X$ of length $a$ and intersecting $X^\circ$ is homeomorphic to $T_1(X^\circ)$ and is hence compact. Therefore, the transverse measure of such a segment is less than a constant $b$. If we split $r$ into segments of length $a$, then each segment is either contained in a cuspidal neighborhood or intersects $X^\circ$. Therefore
$$I_{\mathcal{L},r}(t) < b \lceil \frac{t}{a} \rceil.$$
This implies $I_{\mathcal{L},r}(t)=O(t)$, as desired.
\end{proof}

It would be interesting to see what initial vectors give exactly linear growth. It is obvious that if the image of $r$ is a closed curve, then the growth is linear (or constantly zero if the curve is disjoint from $\mathcal{L}$). Moreover, we have
\begin{prop}\label{prop:leaf}
Assume $r:[0, \infty) \rightarrow X$ is an eventually simple geodesic ray; more precisely, there exists $t>0$ such that $r:[t, \infty) \to X$ does not intersect itself on $X$. Then $I_{\mathcal{L},r}(t)\asymp t$ or $I_{\mathcal{L},r}(t)\asymp 1$.
\end{prop}
\begin{proof}
If $r$ converges to a closed curve then the growth rate is linear. Otherwise, $r$ converges to a leaf of a lamination $\mathcal{L}_1$. So we might as well assume $r$ is part of a leaf of a minimal lamination $\mathcal{L}_1$.

If $\mathcal{L}_1$ is a minimal component of $\mathcal{L}$, we already know $I_{\mathcal{L},r}(t)\asymp 1$. If $\mathcal{L}_1$ and $\mathcal{L}$ are disjoint, then clearly $I(\mathcal{L},r)<\infty$ so $I_{\mathcal{L},r}(t)\asymp1$ as well. Thus we may assume $\mathcal{L}$ and $\mathcal{L}_1$ intersect transversely.

We claim that there exists a constant $C$ so that for any segment $L$ of length $C$ on a leaf of $\mathcal{L}_1$, $I(\mathcal{L},L)>0$. Suppose otherwise. Then there exists $C_i\to\infty$ so that a segment $L_i$ of length $C_i$ on a leaf of $\mathcal{L}_1$ is contained in a complementary region. We may as well assume the starting point $p_i$ of $L_i$ lies on a leaf of $\mathcal{L}$. Taking a subsequence if necessary, we may assume $p_i\to p$. Moreover, since each $L_i$ lies on a leaf of $\mathcal{L}_1$ and the length of $L_i$ goes to $\infty$, $p$ lies on a leaf of $\mathcal{L}_1$ and the half leaf starting at $p$ is contained entirely in a complementary region of $\mathcal{L}$. So either this half leaf of $\mathcal{L}_1$ is asymptotic to a leaf of $\mathcal{L}$, or is contained entirely in a subsurface disjoint from $\mathcal{L}$. Either contradicts our assumption that $\mathcal{L}$ and $\mathcal{L}_1$ intersect transversely.

By compactness, there exists $c>0$ so that any segment $L$ of length $C$ on a leaf of $\mathcal{L}_1$ has $I(\mathcal{L},L)>c$. Dividing the geodesic ray $r$ into segments of length $C$, we conclude that $I_{\mathcal{L},r}(t)>ct$ with a possibly smaller constant $c>0$. Together with Proposition~\ref{Prop: linr}, we have $I_{\mathcal{L},r}(t)\asymp t$.
\end{proof}

It turns out that linear growth is typical with respect to the Liouville measure $\lambda$:

\begin{thm}\label{thm:generic}
For a.e.\ geodesic ray $r$, we have $I_{\mathcal{L},r}(t) \asymp t$. Moreover,
$$\lim_{t \rightarrow \infty}\frac{I_{\mathcal{L},r}(t)}t= \frac{1}{\pi^2(2g-2+n)}\ell(\mathcal{L}).$$
\end{thm}
\begin{proof}
Recall that we denote the geodesic flow on $T_1(X)$ by $\phi_t$, which is ergodic with respect to the Liouville measure $\lambda$ \cite{erg.g.f}. Moreover, $\lambda/(\pi^2(2g-2+n))$ is a probability measure on $T_1(X)$. Define a function $f$ on $T_1(X)$ as follows: $f(v)$ is the transverse measure (with respect to $\mathcal{L}$) of the geodesic segment of unit length with initial tangent vector $v$. By continuity of $f$ and Birkhoff's ergodic theorem, for almost every $v \in T_1(X)$ we have:
\begin{equation}\label{erg}
\lim_{T \rightarrow \infty} \frac{1}{T} \int_{0}^{T} f(\phi_t(v)) dt= c, 
\end{equation}
where $c$ is a constant equal to $1/(\pi^2(2g-2+n)) \int_{T_1(X)} f d\lambda$.

We write $f \simeq g$ if the difference between $f$ and $g$ is less than a constant.

\begin{clm}\label{int.f}
The integral of $f$ along a geodesic segment is roughly the total transverse measure of the segment. In other words, we have:
$$\int_{0}^{T} f(\phi_t(v)) dt \simeq I_{\mathcal{L},r}(T).$$
\end{clm}
\begin{proof}
It is enough to prove the claim for each minimal component of $\mathcal{L}$.

If a component is a simple closed geodesic $\gamma$, $I_{\mathcal{L},r}(T)$ gives the number of intersections of $r([0,T])$ with $\gamma$, while the integral can be larger if $r([T,T+1])$ also intersects $\gamma$. Since the number of intersections of a unit length geodesic segment with $\gamma$ is bounded above, we have the desired relation.

If a component is not a simple closed geodesic, we can divide the ray into segments of length $1/n$ (for a natural number $n$ large enough), and note that
$$n\cdot I(\mathcal{L},r([1,T]))\le\sum_{k=0}^{\lfloor T/n\rfloor}f(\phi_{k/n}(v))\le n\cdot I_{\mathcal{L},r}(T+1).$$
This is because the transverse measure of each small segment of length $1/n$ is counted $n$ times, except for those included in the initial segment of length $1$ and an additional segment of length $1$ at the end. Since $f$ is continuous, hence integrable, dividing by $n$ and letting $n\to\infty$, we conclude that
$$I(\mathcal{L},r([1,T]))\le\int_0^Tf(\phi_t(v))dt\le I_{\mathcal{L},r}(T+1).$$
Since the transverse measure of geodesic segments of length at most one is uniformly bounded, we have
$$\int_{0}^{T} f(\phi_t(v)) dt \simeq I_{\mathcal{L},r}(T),$$
as desired.
\end{proof}

Claim~\ref{int.f} and Equation~\ref{erg} imply $\lim_{t\to\infty}I_{\mathcal{L},r}(t)/t=c$. In order to find $c$ we show for a sequence $T_i \rightarrow \infty$, $\lim \limits_{i \rightarrow \infty} I_{\mathcal{L},r}(T_i)/T_i=\ell(\mathcal{L})$.

Let $\alpha_v^T=\frac1T(\phi_t(v))_*dt$ be $1/T$ of the pushforward of the Lebesgue measure $dt$ on $[0,T]$ via $\phi_t(v)$. From Equation~\ref{erg}, we can see that for a generic $v\in T_1(X)$, $\alpha_v^T$ converges to $\lambda/(\pi^2(2g-2+n))$ when $T \rightarrow \infty$. Moreover, if $v$ is generic, the geodesic ray $\phi_t(v)$ is recurrent, and so there exists a sequence $T_i\to\infty$ such that $\phi_{T_i}(v)\to v$. By Anosov closing lemma \cite{erg.g.f}, there is a closed geodesic $\gamma_i$ very close to the geodesic segment from $v$ to $\phi_{T_i}(v)$. Therefore, $\frac{1}{T_i}\gamma_i$ converges to $\lambda/(\pi^2(2g-2+n))$ as a sequence of geodesic currents. Moreover, we have
$$\frac{1}{T_i} \int_{0}^{T_i} f(\phi_t(v)) dt \simeq \frac{1}{T_i} i(\gamma_i,\mathcal{L})$$
when $T_i$ is large enough.
On the other hand, by continuity of the intersection number we have: 
$$
\lim_{T_i \rightarrow \infty} \frac{1}{T_i}i(\gamma_i, \mathcal{L})=i(\lambda, \mathcal{L})/(\pi^2(2g-2+n))= \ell(\mathcal{L})/(\pi^2(2g-2+n)).
$$
Therefore $I_{\mathcal{L},r}(t) \simeq t\ell(\mathcal{L})/(\pi^2(2g-2+n))$.
\end{proof}
\begin{rmk}
The statement of Theorem~\ref{thm:generic} can be interpreted as an extension of Crofton formula to measured laminations. Crofton formula relates the expected number of times a random line interesting a curve $\gamma$ on Euclidean plane to the length of $\gamma$ (see \cite{intgeobook}). 
\end{rmk}
Proposition~\ref{Prop: linr} and Theorem~\ref{thm:generic} imply Theorem~\ref{thm:linear}. From this result, we conclude the collection of exotic vectors has measure zero, even though it is dense by Corollary~\ref{cor:dense}.

Finally, we give a proof of Theorem~\ref{thm: sublinear}, using some elementary hyperbolic geometry. Note that we may assume $\mathcal{L}$ contains no atoms, for otherwise we can replace $\mathcal{L}$ with its non-atomic part, and the argument below produces a geodesic ray disjoint from the atoms of $\mathcal{L}$.

We construct $r$ in the following way. Start from a point on $\tilde{\mathcal{L}}$ and go along the leaves. At some points, we take a small jump (by taking a segment transverse to $\tilde{\mathcal{L}}$) to the right of the leaf which we are on, then we land on another point of $\tilde{\mathcal{L}}$ and continue going along the leaves. If we choose the times to jump and segments properly then we can obtain a ray with the growth rate of transverse measure $\asymp f$. In preparation for a rigorous argument, we need the following lemma:
\begin{lm}\label{lm:jump}
\begin{enumerate}[label=\normalfont{(\arabic*)}, topsep=0mm, itemsep=0mm]
    \item\label{item: constants} For any $M>0$, there exists positive constants $\ell_1<\ell_2$, $\tau_1<\tau_2$ satisfying the following: for any point $x$ on a leaf of $\tilde{\mathcal{L}}$, there exists a point $x_0$ on the same leaf with $d(x_0,x)\le M$ so that on either side of the leaf, there exists a geodesic segment $s$ satisfying:
    \begin{enumerate}[label=\normalfont{(\alph*)}, topsep=0mm, itemsep=0mm]
        \item segment $s$ starts at $x_0$, orthogonal to the leaf and ends on another leaf;
        \item length of $s$ is in $(\ell_1,\ell_2)$;
        \item transverse measure of $s$ is in $(\tau_1,\tau_2)$.
    \end{enumerate}
    \item\label{item: cross_ratio} Furthermore, for each segment $s$ in \emph{(1)}, the endpoints of the two leaves containing $\partial s$ has cross ratio bounded away from $0$ and $\infty$.
\end{enumerate}
\end{lm}
\begin{proof}
For \ref{item: constants}, if $x$ lies on a leaf that is not the boundary of a component of $\mathbb{H}^2-\tilde{\mathcal{L}}$, then any segment starting at $x$ and perpendicular to the leaf has a positive transverse measure. Otherwise, the segment may lie completely inside a complementary region.

Consider a leaf $\tilde{l}$ bounding a component $\mathcal{P}$ of $\mathbb{H}^2-\tilde{\mathcal{L}}$ that is an ideal polygon. Then there are only finitely many points $x'\in\tilde{l}$ where the geodesic orthogonal to $x'$ lie completely inside $\mathcal{P}$, see Figure~\ref{subfig:avoidance_ideal}. Choose disjoint closed intervals $I(x')\subset\tilde{l}$ so that each $x'$ lies in the interior of $I(x')$. Note that by construction, for any point $x\in\tilde{l}-\cup_{x'}I(x')$, the geodesic orthogonal to $x$ going into $\mathcal{P}$ will intersect the lamination again, and since we take away a definite interval containing $x'$, we have an upper bound $\ell_{\tilde{l}}$ on the length of the shortest orthogonal geodesic meeting $\tilde{\mathcal{L}}$ again.

Consider a leaf $l$ on the boundary of a non-simply-connected complementary region $\mathcal{R}$. We can find a simple closed geodesic $\gamma$ of $\mathcal{R}$ so that cutting along $\gamma$ we obtain a crown whose boundary contains $l$ (see Figure~\ref{subfig:crown_double_crown} and \ref{subfig:nonelementary}). A lift $\tilde{\mathcal{R}}$ of $\mathcal{R}$ is shown in Figure~\ref{subfig:avoidance_cylinder}, where the lift $\tilde\gamma$ of the geodesic $\gamma$ is marked in blue.
\begin{figure}[htp]
    \centering
    \begin{minipage}{0.44\linewidth}
        \centering
        \begin{subfigure}[htp]{\linewidth}
         \captionsetup{width=.8\linewidth}
        \centering
        \includegraphics[width=\linewidth]{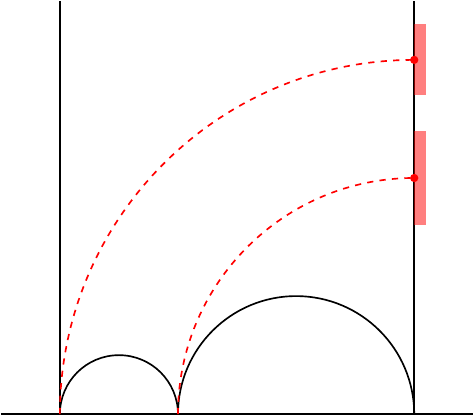}
    \caption{Avoidance regions for a boundary geodesic of an ideal polygon}
    \label{subfig:avoidance_ideal} 
    \end{subfigure}
    \end{minipage}
    \begin{minipage}{0.55\linewidth}
        \centering
        \begin{subfigure}[htp]{\linewidth}
        \captionsetup{width=.8\linewidth}
        \centering
        \includegraphics[width=\linewidth]{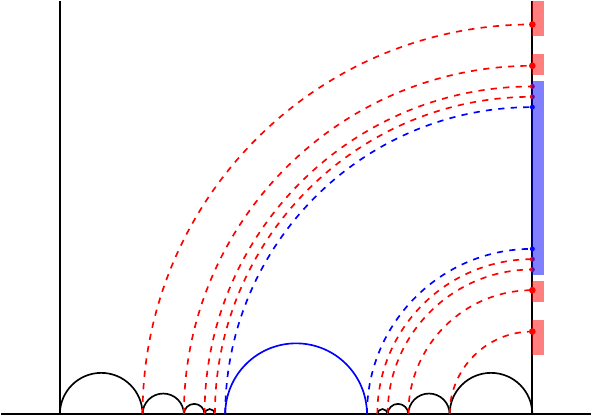}
        \caption{Avoidance regions for a boundary geodesic of a non-simply connected complementary region}
        \label{subfig:avoidance_cylinder} 
    \end{subfigure}
    \end{minipage}
\caption{}
\label{fig:avoidance}
\end{figure}

Consider a lift $\tilde{l}$ of $l$ on the boundary of $\tilde{\mathcal{R}}$. There are infinitely many points $x'\in\tilde{l}$ where the geodesic orthogonal to $x'$ and going into $\tilde{\mathcal{R}}$ lie completely inside $\tilde{\mathcal{R}}$; they accumulate on two points $x_1$ and $x_2$ so that the orthogonal geodesics there are asymptotic to the lift $\tilde\gamma$, see Figure~\ref{subfig:avoidance_cylinder}.

Choose a closed interval $I(x_1,x_2)\subset\tilde{l}$ containing both $x_1$ and $x_2$ in its interior, so that its endpoints are not among the points $x'$ mentioned in the previous paragraph. Then there are only finitely many $x'$ on $\tilde{l}$ not in $I(x_1,x_2)$. Choose disjoint intervals $I(x')$ (also disjoint from $I(x_1,x_2)$ so that each of these finitely many $x'$ lies in the interior of $I(x')$. As in the case of ideal polygon, for any point outside of $I(x_1,x_2)$ and finitely many of $I(x')$, the geodesic orthogonal to $x$ going into $\tilde{\mathcal{R}}$ intersects the lamination again, and we also have an upper bound $\ell_{\tilde l}$ on the length of the shortest orthogonal geodesic meeting $\tilde{\mathcal{L}}$ again.

Now for each $\tilde l$ bounding a complementary component, we obtain an upper bound $\ell_{\tilde l}$ as explained above. The construction is invariant under the action of the fundamental group of the surface on $\mathbb{H}$. Since there are only finitely many complementary regions on the surface, and each complementary region is bounded by finitely many leaves, we can thus choose a uniform $\ell_2>\ell_{\tilde{l}}$ for any $\tilde l$.

The intervals constructed above map down to the surface, and they are disjoint and embedded on the surface by construction. Again, we only have finitely many such intervals $I$ on the surface. For each interval $I$, choose a nearby segment of leaf of roughly the same length to form two opposite sides of a thin quadrilateral. By choose the segment close enough (and we can always choose an arbitrarily close one as we can assume $\mathcal{L}$ is non-atomic), these quadrilaterals can be chosen to be disjoint and embedded. The lifts of the quadrilaterals to $\mathbb{H}$ are also disjoint. Some of these quadrilaterals are colored in red or blue in Figure~\ref{fig:avoidance}.

By construction, segments of leaves contained in each quadrilateral have bounded length. Since there are only finitely many quadrilaterals (up to the action of the fundamental group), we can find $M>0$ so that any leaf segment contained in a quadrilateral has length $\le 2M$.

Now for any point $x$ on a leaf of $\tilde{\mathcal{L}}$, we can always find $x_0$ not contained in any quadrilateral on the same leaf distance $\le M$ away from $x$ so that any geodesic segment starting at $x_0$, perpendicular to the leaf on either side, and of length $\ell_2$ has positive transverse measure. Indeed, if we cannot find any such point, then the leaf is covered by quadrilaterals. But they are disjoint, a contradiction.

By compactness, orthogonal geodesic segments of length $\ell_2$ have positive transverse measure $\le\tau_2$. We cannot find a sequence of orthogonal segments with arbitrarily small transverse measure, since the starting points have to enter the quadrilaterals otherwise. So we have a lower bound $\tau_1$ on the transverse measure as well.

Finally, by cutting off the end part contained completely in a complementary region, we can assume the endpoint is on a leaf. The lengths of these modified segments are bounded below by $\ell_1>0$, as the transverse measure is bounded below. This concludes \ref{item: constants}.

For \ref{item: cross_ratio}, suppose otherwise. Then there exists a sequence of points $x_i$ with geodesic segments $L_i$ starting at $x_i\in\tilde{\mathcal{L}}$ and ending at $y_i\in\tilde{\mathcal{L}}$ so that the cross ratio tends to $0$ or $\infty$. Passing to a subsequence and replacing $x_i$ by an element in the orbit $\pi_1(X)\cdot x_i$ if necessary, we may assume $x_i\to x_0\in\mathcal{L}$, and $y_i\to y_0\in\mathcal{L}$. It is clear the corresponding cross ratio also tends to the limit cross ratio. In particular, by assumption, the leaves which $x_0$ and $y_0$ lie on share a point at infinity. But this means that the geodesic segment $x_0y_0$ is contained in a complementary region and thus has transverse measure $0$. This contradicts the fact that the transverse measure of these segments is bounded below.
\end{proof}
For simplicity, we will call a geodesic segment described in the previous lemma a \emph{connecting segment}. We are now in a position to give a proof of Theorem~\ref{thm: sublinear}:
\begin{proof}[Proof of Theorem~\ref{thm: sublinear}]
Let $\{d_n\}$ be a sequence of positive numbers so that $f(d_1+\cdots+d_n)=n$. By assumption $d_n\to\infty$ as $n\to\infty$. Start with any point $p$ on a leaf of the lamination $\tilde{\mathcal{L}}$. Choose a direction and go along the leaf for a distance of $d_1$, and then turn to the left and then go along a connecting segment predicted by Lemma~\ref{lm:jump}. One may need to go forward or back along the leaf for a distance $\le M$ first, as per Lemma~\ref{lm:jump}. At the end of the segment, turn right and go along the new leaf for a distance of $d_2$, and repeat the process; See the red trajectory in Figure~\ref{fig:construction} for a schematic picture.
\begin{figure}[htp]
    \centering
    \includegraphics[width=0.8\linewidth]{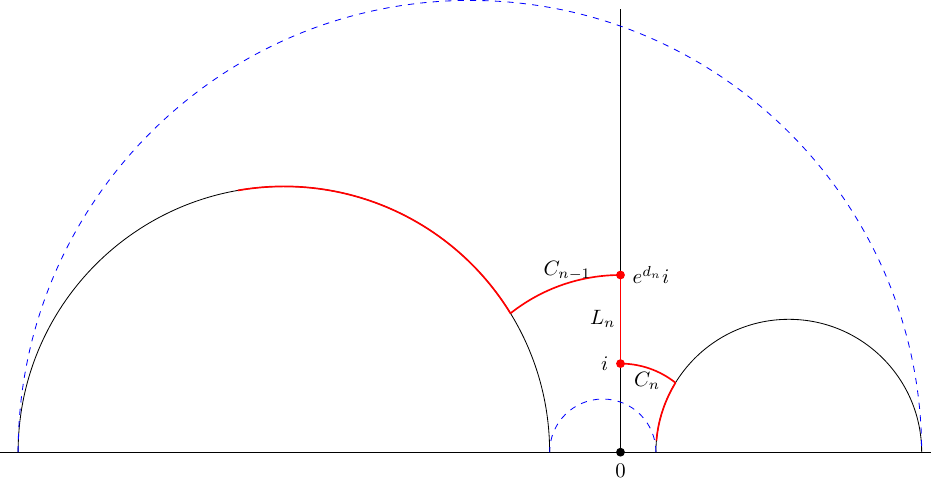}
    \caption{Construction of the geodesic ray}
    \label{fig:construction}
\end{figure}

It is clear that the transverse measure of this piecewise geodesic ray has the desired growth rate. It remains to show that straightening the geodesic does not change this growth rate; since we have:

$$
\frac{\tau_1}{2} =\frac{(n/2)\tau_1}{n} \leq \frac{\textit{transverse measure after length }d_1+\dots+d_n}{f(d_1+\dots+d_n)} \leq \frac{n\tau_2}{n}=\tau_2.
$$

Note that we use the following fact: if the continuous increasing sublinear functions $f_1,f_2$ satisfy $f_1(a_n)=f(b_n)=n$ and $a_n\asymp b_n$, then we have $f_1\asymp f_2$.

We use the upper half plane model. Assume a leaf is the imaginary axis, and one endpoint of a connecting segment is at $i$, then the leaf containing the other endpoint is a half circle whose endpoints are bounded away from $0$ and $\infty$ by Part \ref{item: cross_ratio} of Lemma~\ref{lm:jump}, and also bounded away from each other as the length of the connecting segment is bounded above.
Choose constants $0<m_1<m_2$ so that the endpoints of the circles in either description above lie between distance $m_1$ and $m_2$ from the origin.

Now return to our construction. Denote the $n$-th segment on the leaf by $L_n$, and $n$-th connecting segment by $C_n$. Note that the length of $L_n$ is roughly $d_n$, up to a bounded constant $M$. Put $L_n$ on the imaginary axis, ending at $i$; see Figure~\ref{fig:construction}. The starting point of $L_n$ is at $e^{d_n}i$. Hence the circle on the left half plane (i.e. the leaf $L_{n-1}$ lies on) has both ends contained in the interval $[-m_2e^{d_n},-m_1e^{d_n}]$. After straightening, the geodesic will cross both circles in Figure~\ref{fig:construction}, and intersect the imaginary axis at $(0,y_n)$. The extreme cases are:
\begin{itemize}[topsep=0mm, itemsep=0mm]
\item The straightened geodesic intersects the circle on the left very far into the left endpoints, and the circle on the right very far into the right endpoints (see the larger dotted blue circle in Figure~\ref{fig:construction}), and hence $y_n\le\sqrt{e^{d_n}m_2\cdot m_2}\le e^{d_n}m_2$.
\item The straightened geodesic intersects the circle on the left very far into the right endpoints, and the circle on the right very far into the left endpoints (see the smaller dotted blue circle in Figure~\ref{fig:construction}), and hence $y_n\ge\sqrt{e^{d_n}m_1\cdot m_1}\ge m_1$.
\end{itemize}
Therefore the straightened geodesic intersects the leaf containing $L_n$ within a bounded distance ($\le\max\{|\log m_2|,|\log m_1|\}$) of $L_n$.

The leaves containing $L_n$'s and the connecting segments cut the straightened geodesic into segments. Each of these segments starts or ends on a leaf. 

These segments have the same transverse measure as the subset of the constructed path to which they correspond (each of them is a subset of $L_n$ and $C_n$). Therefore, it suffices to show that these segments and their corresponding subsets of the constructed path have roughly the same length too.

We have the following possibilities
\begin{enumerate}[topsep=0mm, itemsep=0mm]
    \item The segment starts at a point on $L_n$ close to the endpoint of $L_n$. This segment then must end on $C_n$, or a point on the leaf containing $L_{n+1}$, outside $L_{n+1}$ in the backward direction along the leaf. Either way this segment has bounded length, so we can safely ignore the segment for the purpose of growth rates.
    \item The segment starts on the leaf containing $L_n$ sufficiently away from the endpoint of $L_n$. Then we have four cases pictured in Figure~\ref{fig:triangle_inequality}, where the segment is denoted by $\Pi$, and using triangle inequality, one can show that $l(\Pi)\asymp l(L_n')$ where $L_n'=L_n$ if $\Pi$ starts at a point not on $L_n$ (first and second in Figure~\ref{fig:triangle_inequality}), and $L_n'$ is the portion between the starting point of $\Pi$ and endpoint of $L_n$ if $\Pi$ starts at a point on $L_n$ (third and fourth in Figure~\ref{fig:triangle_inequality}). Indeed, for the first one, recall that the starting point of $\Pi$ lies within bounded distance $m=\max\{|\log m_2|,|\log m_1|\}$ of $L_n$. So we have by triangle inequality
    $$l(\Pi)\le m+l(L_n)+l(C_n)\le l(L_n)+m+\ell_2\le Cl(L_n)$$
    for some constant $C>1$ independent of $n$, since $l(L_n)\asymp d_n\to\infty$. On the other hand, clearly $l(\Pi)\ge l(L_n)$. The remaining four cases can be dealt with similarly.
    \item The segment ends at a point on $L_n$ close to the starting point of $L_n$. Similarly, This segment then must start on $C_{n-1}$, or a point on the leaf containing $L_{n-1}$, outside $L_{n-1}$ in the forward direction along the leaf. Again this segment has a bounded length.
    \item The segment ends on the leaf containing $L_n$ sufficiently away from the starting point of $L_n$. This case can be dealt with entirely similarly to Case 2.
\end{enumerate}
\begin{figure}[htp]
\begin{minipage}{0.245\linewidth}
    \begin{subfigure}[htp]{\linewidth}
        \includegraphics[width=\linewidth]{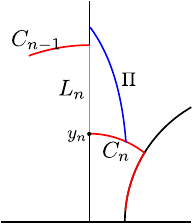}
    \end{subfigure}
\end{minipage}
\begin{minipage}{0.245\linewidth}
    \begin{subfigure}[htp]{\linewidth}
        \includegraphics[width=\linewidth]{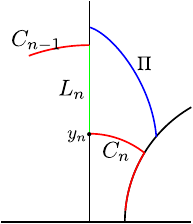}
    \end{subfigure}
\end{minipage}
\begin{minipage}{0.245\linewidth}
    \begin{subfigure}[htp]{\linewidth}
        \includegraphics[width=\linewidth]{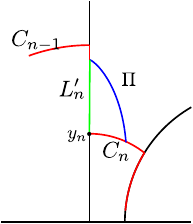}
    \end{subfigure}
\end{minipage}
\begin{minipage}{0.245\linewidth}
    \begin{subfigure}[htp]{\linewidth}
        \includegraphics[width=\linewidth]{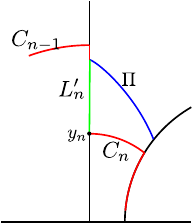}
    \end{subfigure}
\end{minipage}
\caption{The starting point of the blue segment is sufficiently far away from the endpoint $y_n$ of $L_n$. In all cases, the blue segment has length comparable to the green segment}
\label{fig:triangle_inequality}
\end{figure}
From these observations, we conclude that the straightened geodesic ray also has the prescribed growth rate, as desired.
\end{proof}

\part{Geodesic planes in a geometrically finite end}\label{part:qf}

\section{Convex core boundary of quasifuchsian manifolds}\label{sec: quasifuchsian}
\paragraph{Quasifuchsian manifolds}
Let $X=\Pi\backslash\mathbb{H}^2$ be a complete hyperbolic surface of finite area, possibly with punctures. We may view $\Pi\subset\psl(2,\mathbb{R})\subset\psl(2,\mathbb{C})$ as a Kleinian group. Let $M_X=\Pi\backslash\mathbb{H}^3$ be the corresponding 3-manifold. Note that the convex core of $M_X$ is a hyperbolic surface isometric to $X$.

A quasifuchsian manifold is a hyperbolic 3-manifold $M=\Gamma\backslash\mathbb{H}^3$ quasi-isometric to $M_X$. Note that this definition is independent of the choice of $X$, since $M_X$ and $M_Y$ are quasi-isometric for any pair of finite-area complete hyperbolic surfaces $X,Y$ of the same topological type. The manifold $M$ is geometrically finite, and its convex core boundary consists of two hyperbolic surfaces (homeomorphic to $X$) bent along a bending lamination. See below for details.

The space of quasifuchsian manifolds quasi-isometric to $M_X$ (up to isometry) is biholomorphic to $\mathcal{T}(X)\times\mathcal{T}(\bar{X})$, which corresponds to the two conformal structures (with opposite orientations) at infinity given by the action of $\Gamma$ on the two components $\Omega_\pm$ of the domain of discontinuity.

\paragraph{Convex core boundary}
We now describe the convex core boundary of a quasifuchsian manifold, following \cite{bending}. Proofs are omitted, but we quote the relevant results in \cite{bending} for those who wish to refer to them.

Let $\Lambda\subset S^2$ be a closed subset of the sphere at infinity, and let $\hull(\Lambda)$ be its convex hull in $\mathbb{H}^3$. A \emph{support plane} $P$ of a closed convex subset $E$ of $\mathbb{H}^3$ is an isometrically embedded totally geodesic hyperbolic plane such that $P\cap E\neq\varnothing$ but $P\cap\Int(E)=\varnothing$. Given any point $x\in\partial E$, there exists a support plane $P$ containing $x$ \cite[Lemma II.1.4.5]{bending}. When $E=\hull(\Lambda)$, let $P$ be a support plane for $E$ with boundary circle $C$; then $P\cap E=\hull(C\cap\Lambda)$ \cite[Lemma II.1.6.2]{bending}. In particular, $P\cap E$ is either a complete geodesic (when $|C\cap\Lambda|=2$), or a 2-dimensional submanifold of $P$, bounded by geodesics with endpoints in $C\cap\Lambda$ (when $|C\cap\Lambda|>2$) \cite[Cor. II.1.6.3]{bending}. We call the former a \emph{bending line}, and the latter a \emph{flat piece}. A geodesic boundary component of a flat piece is also called a bending line, although there may or may not exist a support plane whose intersection with $E$ is precisely that geodesic. Note that bending lines are disjoint and hence form a geodesic lamination on $\partial E$.

Let $l$ be a bending line. The set of geodesic planes containing $l$ is naturally identified with $S^1$, by fixing a plane to be the base and parametrizing the other planes by rotating the base plane. It is clear that support planes for $l$ form a closed connected subset of $S^1$. We call a geodesic plane corresponding to a boundary point of this closed subset an \emph{extreme support plane}. In particular, either $l$ has a unique extreme support plane (and thus has a unique support plane), or two extreme support planes.

Given a point $x$ on a bending line $l$, there exists an open neighborhood $U$ in $\partial E$ containing $x$ so that if two bending lines $l_1,l_2$ meet $U$, any support plane to $l_1$ meets any support plane to $l_2$ \cite[Lemma II.1.8.3]{bending}, and we call any such intersection a \emph{ridge line}. The exterior dihedral angle\footnote{One of the four angles formed by two intersecting support planes contains $E$, and either of its two complementary angles is an exterior dihedral angle.} formed by the support planes at a ridge line is called a \emph{ridge angle}.

If a ridge line meets a bending line then they are equal \cite[Lemma II.1.9.2]{bending}. Otherwise, there is a ``gap" between the ridge line and $\partial E$, and we can add a support plane to fill part of the gap; see Figure~\ref{fig: raising_the_roof}. The original ridge angle is replaced by two ridge angles. We can apply this process inductively to obtain better and better finite approximations of $\partial E$ locally. Note that each of these approximations consists of a chain of intersecting support planes, and the bending angle of this chain is the sum of all ridge angles formed by neighboring support planes in the chain.
\begin{figure}[htp]
    \centering
    \captionsetup{width=.8\linewidth}
    \includegraphics[width=0.4\linewidth]{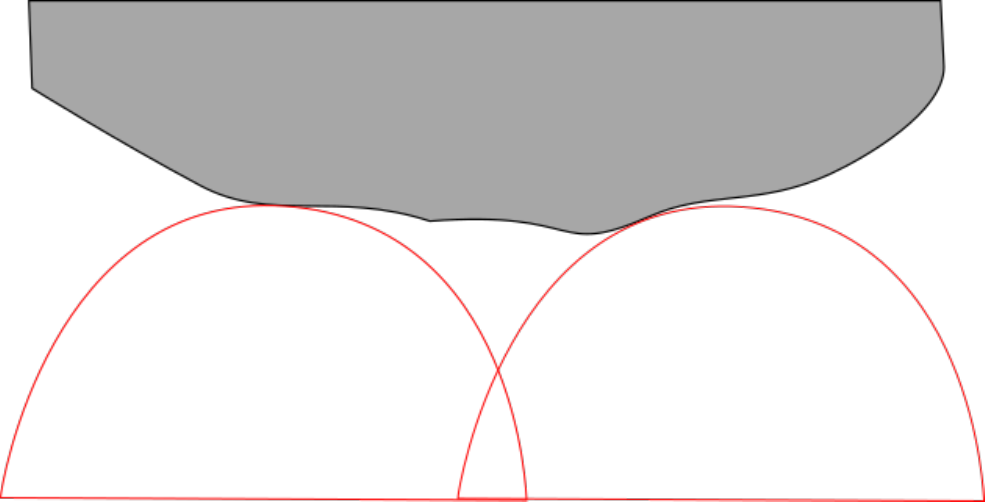}
    \hspace{0.05\linewidth}
    \includegraphics[width=0.4\linewidth]{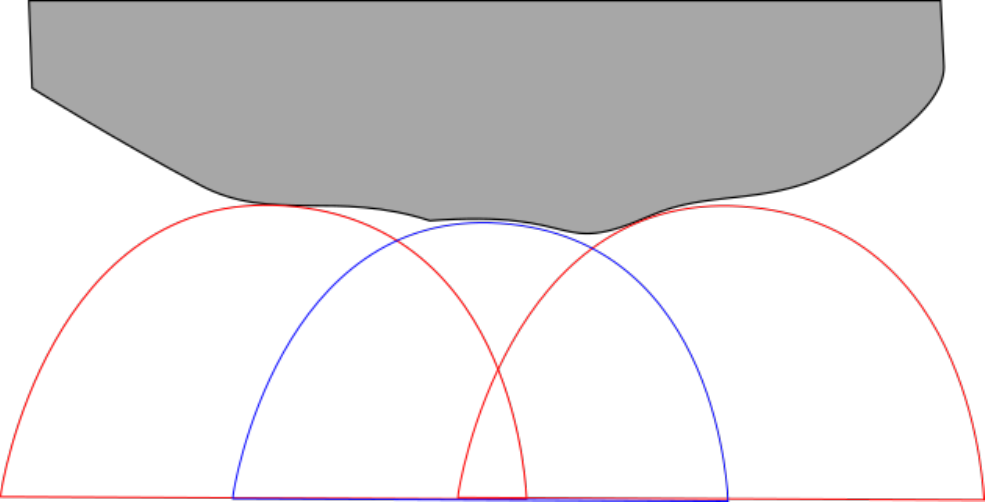}
    \caption{Filling the gap, or raising the roofs. If there is a gap between the convex core boundary and two support planes, we can fill in a third to cover part of the gap. This can be done inductively to get a better and better approximation of the convex core boundary.}
    \label{fig: raising_the_roof}
\end{figure}

Let $I$ be a small segment contained in $U$ so that the endpoints $x_1,x_2$ of $I$ lie on flat pieces and $I$ intersects the bending lines transversely. Let $P_1$ and $P_2$ be two support planes containing $x_1$ and $x_2$ respectively. They intersect in a ridge line. We then have a sequence $\{S_n\}$ of chains of support planes, where $S_n$ is obtained from $S_{n-1}$ by ``filling the gaps" described above. Assume that $S_n\cap I$ becomes dense as $n\to\infty$. The \emph{bending measure} of $I$ is then defined as
$$\mu(I)=\lim_{n\to\infty}(\text{bending angle of }S_n).$$
This is well defined and gives a transverse measure on the geodesic lamination of all bending lines \cite[Thm II.1.11.3]{bending}.

We remark that $\partial E$ has a hyperbolic metric induced from $\mathbb{H}^3$; as a matter of fact, each component of $\partial E$ is a complete hyperbolic $2$-manifold, isometrically covered by $\mathbb{H}^2$ with respect to this metric.

Finally, note that when $\Lambda$ is the limit set of a quasifuchsian manifold $M\cong\Gamma\backslash\mathbb{H}^3$, the boundary of $\hull(\Lambda)$ and its bending lamination are invariant under $\Gamma$, and thus give a pair of pleated hyperbolic surfaces that bound $\core(M)$.

\paragraph{Behavior of support planes}
We now prove some propositions about the behavior of support planes that are straightforward consequences of the general description of the convex core boundary above. More detailed and precise descriptions are presented in later sections.

Throughout this part, let $P$ be a support plane for an end $E_+$ of the quasifuchsian manifold $M\cong\Gamma\backslash\mathbb{H}^3$, $\Lambda$ its limit set, $\Omega_+$ the corresponding component of domain of discontinuity. Also let $\tilde E_+$ be the preimage of $E_+$ under the covering map $\mathbb{H}^3\to M$, $\tilde P$ a lift of $P$ to $\mathbb{H}^3$, and $C$ the corresponding boundary circle. Suppose that the component of $\partial\core(M)$ corresponding to the end $E_+$ is isometric to a hyperbolic surface $X$, bent along a geodesic lamination $\mathcal{L}\subset X$ with bending measure $\mu$. We write $\mathcal{L}=\mathcal{L}_a\cup\mathcal{L}_m$, where $\mathcal{L}_a$ consists of simple closed geodesics in $\mathcal{L}$ (the \emph{atomic} part of $\mathcal{L}$), and $\mathcal{L}_m$ consists of non-atomic components of $\mathcal{L}$.

\begin{prop}\label{prop: support_flat}
Suppose that $\tilde P\cap\hull(\Lambda)$ is a flat piece. Then $P$ intersects precisely one connected component $X'$ of $X-\mathcal{L}_a$, and the closure of $P$ contains $\cup_{\text{Plane } Q\text{ supports }X'}Q$.
\end{prop}
\begin{proof}
First assume that $X'$ contains only one minimal component of $\mathcal{L}_m$. Let $Q$ be any plane that supports $X'$. Then either the lift $\tilde Q$ intersects $\hull(\Lambda)$ in a bending line, or a flat piece. In the former case, there exists a sequence of points $x_i$ on the boundary of flat pieces $\gamma_i\tilde P\cap\hull(\Lambda)$, such that $x_i\to x\in \tilde Q\cap\hull(\Lambda)$ (this is because a geodesic in a component of $\mathcal{L}_m$ is dense in that component). Hence $\gamma_iP\to Q$ as each piece $\gamma_iP$ has to be the extreme supporting plane at $x_i$ and we can apply \cite[Lemma II.1.7.3]{bending}. In the latter case, either $Q=\gamma P$ for some $\gamma\in\Gamma$, or $Q$ supports a flat piece in the same connected component. Since the flat piece on $Q$ is bounded by leaves, we can choose $x\in\partial (\tilde Q\cap\hull(\Lambda))$ and $x_i$ on the boundary of some sequences of flat pieces $\gamma_i\tilde P\cap\hull(\Lambda)$ so that $x_i\to x$ similarly as above. Then the same argument shows that $Q$ is in the orbit closure as well.

For the general case, we may divide $X'$ into subsurfaces using disjoint simple closed geodesics, so that each subsurface contains one minimal component. Two subsurfaces sharing a simple closed geodesic also share a flat piece containing that geodesic. By connectedness and repeatedly using the argument above, we have the desired result.
\end{proof}
\begin{prop}\label{prop: support_leaf}
Suppose $\tilde P\cap\hull(\Lambda)$ is a bending geodesic. Then either $P$ is closed (when the geodesic corresponds to an atom in $\mathcal{L}_a$), or the closure of $P$ contains a support plane to a flat piece (when the geodesic corresponds to a leaf in a minimal component of $\mathcal{L}_m$).
\end{prop}
\begin{proof}
For the former case, a boundary circle $C$ of $P$ is contained in the closure of $\Omega_+$. This circle is stabilized by a hyperbolic element $\gamma$. Choose a fundamental domain $F$ in $C$ for the action of this hyperbolic element, consisting of two circular arcs. Note that $F$ is contained in $\Omega_+$, where the action of $\Gamma$ is properly discontinuous. Let $\{\gamma_i\}\subset\Gamma$ be a sequence so that $\gamma_iC\to C'$ for some circle $C'$. Fix $x\in C'\backslash\Lambda$. Then we can find $y_i\in C\backslash\Lambda$ so that $\gamma_iy_i\to x$. As $F$ is a fundamental domain for the action of $\gamma$ on $C$, there exist points $x_i\in F$ and integers $n_i$ so that $\gamma^{n_i}x_i=y_i$. Thus $\gamma_i\gamma^{n_i}x_i\to x\in C'\backslash\Lambda$. By proper discontinuity, $\gamma_i\gamma^{n_i}=\delta\in\Gamma$ for all $i$ sufficiently large. In particular, $C'=\delta C$, and hence the orbit $\Gamma\cdot C$ is indeed closed.

For the latter case, note that every minimal component contains a leaf that bound a flat piece. Since the leaf is dense in the corresponding minimal component, there exists a sequence of points $x_i\in\gamma_iP\cap\hull{(\Lambda)}$ such that $x_i\to x$, where $x$ is a point on the boundary of a flat piece. Hence by \cite[Lemma II.1.7.3]{bending}, $\gamma_iP$ limits to the plane supporting that flat piece.
\end{proof}

As a corollary, we have the first statement of Part~\ref{item:thm:mimial_roof_dense} of Theorem~\ref{thm: minimal_roof}.
\begin{cor}\label{cor: support_closure}
Suppose $\mathcal{L}_a=\varnothing$. Let $P$ be a support plane. Then $\overline{P}=E_+$
\end{cor}
\begin{proof}
By the two previous propositions, in this case, $\overline{P}$ contains all support planes. Let $C_1$ and $C_2$ be boundary circles of two support planes $\tilde P_1$ and $\tilde P_2$ of $\partial\tilde E_+$ in the universal cover. Choose any points $p_1\in\tilde P_1\cap\partial \tilde E_+$ and $p_2\in\tilde P_2\cap\partial\tilde E_+$, and let $r:[0,1]\to\partial\tilde E_+$ be any continuous path between them. For each $t\in[0,1]$, there exists a unique support plane $\tilde P_t$ containing $r(t)$. The corresponding boundary circle $C_t$ varies continuously with respect to $t$. Hence the support planes sweep through the end, and hence the union of all of them is $E_+$.
\end{proof}

\begin{prop}\label{prop: separate}
For $k=1,2$, let $X_k$ be either a connected component of $X-\mathcal{L}_a$ or a simple closed geodesic in $\mathcal{L}_a$, and $P_k$ be a support plane so that $P_k\cap\core(M)\subset X_i$. If $X_1\neq X_2$, then the closure of $P_1$ does not contain $P_2$.
\end{prop}
\begin{proof}
If $X_1$ is a simple closed geodesic, the statement follows from the previous proposition. Fix a lift $\tilde{P_2}$ of $P_2$. Any lift of $P_1$ is either disjoint from $\tilde{P_2}$ or intersects $\tilde{P_2}$ with a ridge angle bounded below. Indeed, any path from $P_1$ to $P_2$ must pass through an atom in $\mathcal{L}_a$, and the bending measure is defined as the infimum of the sum of the ridge angles of any chain of support planes along the path. In particular, the orbit of a lift of $P_1$ does not limit to $\tilde{P_2}$.
\end{proof}

\section{Geodesic planes outside convex core I: purely atomic bending}\label{sec:geodesic_planes_atomic}
In this section, we completely classify the behavior of geodesic planes outside the convex core when the bending lamination is purely atomic. That is, we prove Proposition~\ref{prop: atomic}, as a combination of Propositions~\ref{prop:horodisks}, \ref{prop: atomic_support_flat}, \ref{prop: atomic_horocycle_geodesic}, and \ref{prop: atomic_horocycle_flat}.

We start with a discussion of the behavior of geodesic planes whose boundary circles meet the limit set exactly at a parabolic fixed point (i.e. the fixed point of a parabolic element in $\Gamma$). Note that such a plane only exists if the quasifuchsian manifold of interest is quasi-isometric to $M_X$ for some hyperbolic surface $X$ with punctures. For simplicity, such a geodesic plane is called a \emph{parabolic plane} (and its boundary circles are called \emph{parabolic circles}).

Note that every parabolic circle is shadowed by a support circle. Indeed, parabolic fixed points correspond to cusps of the quasifuchsian manifold, and on the boundary $\partial E_+$, each cusp has a neighborhood contained in a flat piece; see Figure~\ref{subfig:punctured_polygon}.

Our first goal is to show that a parabolic plane is always closed, regardless of the bending lamination. This gives Part~\ref{item:asy_parabolic} of Proposition~\ref{prop: atomic} and Part~\ref{item:thm:minimal_roof_parabolic} of Theorem~\ref{thm: minimal_roof}. We adopt the notations from the last section.
\begin{prop}\label{prop:horodisks}
    Let $\tilde P\subset\tilde E_+$ be a geodesic plane, and $C$ its boundary circle. Suppose $C\cap\Lambda$ is a parabolic fixed point. Then the corresponding plane $P$ is closed in $E_+$
\end{prop}
\begin{proof}
    Let $\gamma$ be the parabolic element whose fixed point is $p=C\cap\Lambda$. The plane $\tilde P$ contains horocycles based at $p$, which are invariant under the action of the $\gamma$. Each horocycle bounds a horodisk, and it follows easily from the description of cuspidal neighborhoods of hyperbolic 3-manifolds that the projection of any horodisk $\mathcal{D}$ under $\pi:\mathbb{H}^3\to M$ is closed (see e.g. \cite[Thm.~2.24]{hyperbolic_manifolds}).

    Now a fundamental domain for $\tilde{P}-\mathcal{D}$ is compact in $\mathbb{H}^3\cup\Omega$, on which the action of $\Gamma$ is properly discontinuous (see Figure~\ref{fig:lune}). An argument similar to that of Proposition~\ref{prop: support_leaf} implies that $\pi(\tilde{P}-\mathcal{D})$ in $M$ is closed. Hence $P=\pi(\tilde P-\mathcal{D})\cup\pi(\mathcal{D})$ is closed.
\end{proof}
\begin{figure}[htp]
    \captionsetup{width=0.6\linewidth}
    \centering
    \includegraphics[width=0.4\linewidth]{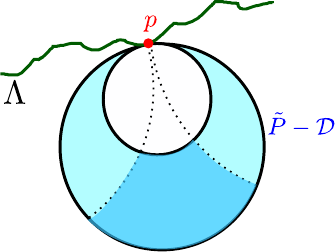}
    \caption{A schematic picture for the proof of Proposition~\ref{prop:horodisks}. The region $\tilde P-\mathcal{D}$ has a crescent shape, and a fundamental domain for the action of $\gamma$ is shaded darker. Note that the fundamental domain is compact and bounded away from $\Lambda$.}
    \label{fig:lune}
\end{figure}
\begin{rmk}\label{rmk:parabolic_support}
    We remark that a parabolic circle $C$ is always shadowed by a support circle. Indeed, if $p=C\cap\Lambda$ is a parabolic fixed point, then it corresponds to a cusp of the surface $\partial E_+$. The bending lamination on $\partial E_+$ is contained in a compact set, and thus the cusp always belongs to a flat piece (see e.g.~Figure~\ref{subfig:punctured_polygon} in the introduction). The corresponding support plane has a boundary circle $C'$ meeting $\Lambda$ at $p$, which must be tangent to $C$ at $p$ (cf.\ Lemma~\ref{lm:angle_means_atom} below).
\end{rmk}

We next consider the behavior of support planes when the bending lamination is purely atomic.
\begin{prop}\label{prop: atomic_support_flat}
If the bending lamination is purely atomic, i.e. $\mathcal{L}_m=\varnothing$, then any support plane $P$ to a flat piece is closed.
\end{prop}
\begin{proof}
Since the bending lamination is purely atomic, each component of $\partial E_+-\mathcal{L}_a$ is a nonelementary hyperbolic surfaces with geodesic boundary. The stabilizer $\stab_\Gamma(\tilde P)$, being isomorphic to the fundamental group of one such component, is a geometrically finite nonelementary Fuchsian group of second kind, and is actually convex cocompact if the flat piece does not contain any cusps (see e.g.\ \cite[Thm.~5.1]{MMO2} and \cite[Lem.~3.7]{acy_geom_finite}).

In the convex cocompact case, a fundamental domain $F$ of $\stab_\Gamma(\tilde P)$ in $\tilde P\cup C$ is compact in $\mathbb{H}^3\cup\Omega$, on which the action of $\Gamma$ is properly discontinuous. An argument similar to the first case of Proposition~\ref{prop: support_leaf} then gives the result: if $\gamma_i C\to C'$, then we can find $x_i\in F\cap C$ and $\tau_i\in\stab_\Gamma(\tilde P)$ so that $\gamma_i\tau_ix_i\to x\in C'\backslash\Lambda$, and thus by proper discontinuity, $\gamma_i\tau_i=\delta\in\Gamma$ for all $i$ sufficiently large, and hence $C'=\delta C$.

More generally, we can take away from $\tilde P$ finitely many orbits of small enough horodisks (so that these horodisks project down to disjoint cuspidal neighborhoods on $\partial E_+$) similar to the proof Propostion~\ref{prop:horodisks}. The remaining portion has a compact fundamental domain for the action of $\stab_\Gamma(\tilde P)$. We can then argue as above to obtain the desired result.
\end{proof}
Since in this case, a support plane either supports a flat piece or an atom in $\mathcal{L}_a$, this proposition and Proposition~\ref{prop: support_leaf} imply that all support planes are closed, giving \ref{item:cylinder} and \ref{item:nonelementary} of Proposition~\ref{prop: atomic}. It remains to consider the case when $P$ is an asymptotic plane, or equivalently $|C\cap\Lambda|=1$.

The following lemma is true for an arbitrary bending lamination (not necessarily purely atomic):
\begin{lm}\label{lm:angle_means_atom}
Let $C_1,C_2$ be two non-tangent circles contained in $\overline{\Omega_+}$ such that $C_1\cap\Lambda=C_2\cap\Lambda=p$, then $p$ is a fixed point of a hyperbolic element in $\Gamma$ corresponding to an atom in $\mathcal{L}_a$.
\end{lm}
\begin{proof}
Indeed, note that $\overline{\Omega_+}$ contains a circle $C'$ such that $C'\cap\Lambda=p$ and $C'$ is tangent to neither $C_1$ nor $C_2$ (see Figure~\ref{fig:tangent circles}). Moreover, $C'$ must be shadowed by a support plane; if we enlarge the circle $C'$ (but make sure that it remains tangent to $C'$ at $p$), then it must first meet the limit set elsewhere (note that \emph{a priori}, it is not immediately clear that it meets the limit set elsewhere at only one point). For the corresponding plane, we have $P'\cap\partial E_+=\hull(C'\cap\Lambda)$ by \cite[Lemma II.1.6.2]{bending}. As $p$ is an isolated point in $C'\cap\Lambda$, it must be the end point of a leaf.

Note that by rotating the circle $C'$ slightly around $p$, we find another support plane containing a leaf ending at $p$. By the description of the shape of the convex core boundary in the last section, if $p$ is the end point of a lift of a leaf in a minimal component of $\mathcal{L}_m$, then there is only one unique support circle meeting $\Lambda$ at $p$.

Hence $p$ is the endpoint of some lift of an atom, as desired.
\end{proof}
\begin{figure}[htp]
\centering
\includegraphics{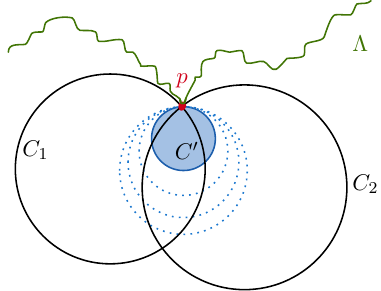}
\caption{It is possible to enlarge $C'$ without crossing the limit set}
\label{fig:tangent circles}
\end{figure}

\begin{prop}\label{prop: atomic_horocycle_geodesic}
Let $p$ be a point satisfying the conditions of the previous lemma, and let $C$ be a circle touching the limit set only at $p$. Let $\gamma\in\Gamma$ be a hyperbolic element with $p$ as the repelling fixed point. Suppose $\gamma^iC\to C'$ as $i\to\infty$ in $\mathcal{C}$. Then $\overline{\Gamma\cdot C}=\Gamma\cdot C\cup\Gamma\cdot C'$.
\end{prop}
\begin{proof}

It suffices to show that $\overline{\Gamma\cdot C}\subset\Gamma\cdot C\cup\Gamma\cdot C'$. We may assume $\Omega_+$ is a bounded domain. Then $C$ is contained in the disk bounded by $C'$. Suppose $\gamma_i C\to C_{\lim}$. Then $\gamma_iC'$ also limits to a circle $C_{\lim}'$ (passing to a subsequence if necessary), and $C_{\lim}$ is contained in the disk bounded by $C_{\lim}'$. By Proposition~\ref{prop: support_leaf}, when $i$ is sufficiently large, $\gamma_i\gamma^{n_i}=\delta$ for some integer $n_i$, and $C_{\lim}'=\delta C'$. It then follows that either $C_{\lim}=\delta\gamma^nC\in\Gamma\cdot C$ for some integer $n$ (when $n_i$ is bounded above), or $C_{\lim}=\delta C'\in\Gamma\cdot C'$, as desired.
\end{proof}

On the other hand, if $p$ is not the endpoint of a lift of an atom, any two circles $C_1$ and $C_2$ contained in $\overline{\Omega_+}$ meeting the limit set only at $p$ are tangent to each other at $p$. In particular, if we assume $\Omega_+$ is a bounded domain (which we may without loss of generality by applying a M\"obuis transform), there exists such a circle of the largest radius. We will call this the \emph{osculating circle} at $p$. Note that in the introduction, we call the geodesic plane corresponding to such a circle a roof. We will discuss more about roofs and osculating circles in the next section. In the case of purely atomic bending, we have:
\begin{lm}\label{lm: atomic_roof}
Suppose $\mathcal{L}_m=\varnothing$. Then every roof is a support plane.
\end{lm}
\begin{proof}
Equivalently, every osculating circle meets the limit set at $2$ or more points. Let $C$ be an osculating circle at $p$, and $\tilde P$ the corresponding geodesic plane in $\mathbb{H}^3$. Take a circle $C'$ passing through $p$ and separate the limit set (this means that the limit set intersects both components of $S^2-C'$), and $\tilde P'$ the corresponding geodesic plane. Let $l=\tilde P\cap \tilde P'$ and $l'$ be the intersection of $\tilde P'$ with the boundary of $\hull(\Lambda)$ on this end.
\begin{figure}[htp]
\centering
\includegraphics{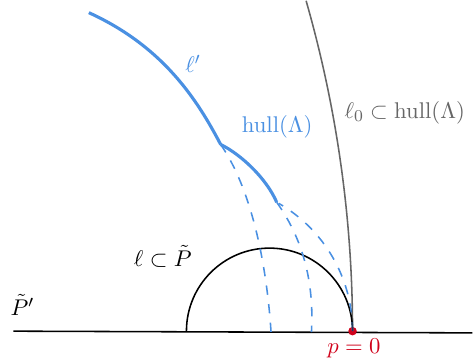}
\caption{A vertical slice of the convex core}
\label{fig:core_slice}
\end{figure}

Note that $l'$ consists of countably many geodesic arcs that come from the intersection of $\tilde P'$ with support planes. We map $\tilde P'$ to the standard upper half plane so that $p$ is the origin, $l$ is a half circle to the left of the imaginary axis, and $l'$ is a piecewise geodesic above $l$ (see Figure~\ref{fig:core_slice}).

As mentioned after the proof of Proposition~\ref{prop: atomic_support_flat}, when $\mathcal{L}_m=\varnothing$, the $\Gamma$-orbit of every support plane is discrete, since the corresponding plane in the manifold is closed (for the relation between closed orbits and discreteness, see \cite[\S2]{MMO1}). If $l'$ consists of infinitely many geodesic segments, then by the discreteness of the orbit of the supporting planes, the radii of the half circles of these geodesic segments tend to zero. On the other hand, these half circles all lie to the left of the imaginary axis, for otherwise $p$ is contained in the domain of discontinuity. But this is impossible, as then $l'$ would eventually lie below $l$. 
\end{proof}
In the language we introduced, this means that every asymptotic plane is shadowed by a support plane when the bending lamination is purely atomic. If the support plane intersects the convex core boundary in a simple closed geodesic, Proposition~\ref{prop: atomic_horocycle_geodesic} implies that the closure of the asymptotic plane is the union of itself and the support plane. If the support plane intersects the convex core boundary in a flat piece, we have
\begin{prop}\label{prop: atomic_horocycle_flat}
Let $\mathcal{L}_m=\varnothing$. Suppose $C\subset\overline{\Omega_+}$ and $C\cap\Lambda=p$, where $p$ is not a parabolic fixed point. Suppose furthermore that $C$ is contained in the closed disk $D$ bounded by $C'$, a boundary circle of a flat piece. Then $\overline{\Gamma\cdot C}=\cup_{\gamma\in\Gamma}\mathcal{H}(\gamma D)$.
\end{prop}
Here $\mathcal{H}(D):=\cup_{p\in\Lambda\cap D}\{C_0:C_0\text{ tangent to $\partial D$ at $p$ and $C_0\subset\overline{D}$}\}\subset\mathcal{C}$ for any open or closed disk $D$ contained in $\overline{\Omega_+}$.

\begin{proof}
Note that $\stab_\Gamma(D)$ is a geometrically finite nonelementary Fuchsian group of second kind. Applying Dal'bo's theorem \cite{horocycle}, since $p$ is not a parabolic fixed point, we have $\overline{\Gamma\cdot C}\supset\cup_{\gamma\in\Gamma}\mathcal{H}(\gamma D)$. The other inclusion is similar to the proof of Proposition~\ref{prop: atomic_horocycle_geodesic}.
\end{proof}

Propositions~\ref{prop: support_leaf}, \ref{prop:horodisks}, \ref{prop: atomic_support_flat}, \ref{prop: atomic_horocycle_geodesic} and \ref{prop: atomic_horocycle_flat} together with Lemma~\ref{lm: atomic_roof} now give Proposition~\ref{prop: atomic}.

\section{Exotic roofs}\label{sec:special_roofs}
In this section, we study exotic roofs. Whether they exist or not leads to different consequences for the classification of geodesic planes, as we will see in \S\ref{sec:geodesic_planes_irrational}.
\subsection{Exotic roofs and exotic rays: a necessary condition}\label{sec: roof}
In this subsection, we aim to relate the existence of exotic roofs to the existence of exotic rays and prove Theorem~\ref{prop:exotic_circle_necessary}, giving a necessary condition for the existence of exotic roofs.

We start with the following observation:
\begin{lm}\label{lem:finite_bending}
Let $C\subset\overline{\Omega_+}$ so that $C\cap\Lambda=p$. Let $C'$ be any circle passing through $p$ that separates $\Lambda$, and $\tilde P'$ the corresponding geodesic plane in $\mathbb{H}^3$. Let $l'$ be the intersection of $\tilde P'$ with $\partial\hull(\Lambda)$ on this end. Then any ray contained in $l'$ tending to $p$ has finite bending.
\end{lm}
Before giving the proof, we first set up the picture and define what we mean by ``finite bending". Let $\tilde P$ be the geodesic plane corresponding to $C$, and $l=\tilde P'\cap\tilde P$. Since $C'$ separates $\Lambda$, it must intersects $\Lambda$ in a different point $q$ from $p$. The geodesic $l_0$ connecting $p$ and $q$ is then contained in $\tilde P'$ and entirely in $\hull(\Lambda)$.

As in the last section, we can map $\tilde P'$ to the upper half plane $\mathbb{H}^2$ isometrically, with $p$ at the origin and $l$ a half circle to the left of the imaginary axis. Then $l'$ is a piecewise geodesic sandwiched between $l$ and $l_0$ (see Figure~\ref{fig:core_slice}). Each geodesic piece of $l'$ is part of the intersection of $\tilde P'$ and a support plane to a flat piece; we call this intersection a \emph{support line}. In our picture, any support line is a half circle to the left of the imaginary axis as well, with both endpoints on the negative real line. Moreover, two support lines intersect in a point above $l$.

Any finite segment $I$ of $l'$ can be approximated by a sequence $\{S_n\}$ of intersecting chains of support lines, by ``raising the roofs" just as discussed in \S\ref{sec: quasifuchsian}; see Figure~\ref{fig: raising_the_roof}. If $S_n\cap I$ becomes dense in $I$, we similarly define the \emph{bending} of $I$ to be the limit $\lim_{n\to\infty}(\text{bending angle of $S_n$})$. A ray in $l'$ is said to have \emph{finite bending} if any finite portion of it has bending uniformly bounded above.

\begin{proof}[Proof of Lemma~\ref{lem:finite_bending}]
Recall that $l'$ is a path in $\tilde P'$ ending in $p$ sandwiched between $l$ and $l_0$. Let $l''$ be a geodesic segment connecting a point on $l'$ and a point on $l$. Then $l,l',l''$ form a finite area triangle, with two geodesic boundary components and one with bending. By Gauss-Bonnet, the bending has to be finite.
\end{proof}

Note that this does not immediately imply that the ray has finite transverse measure with respect to the bending lamination, as the bending of the ray may be much smaller than the bending measure, depending on the angle the ray intersects the leaves of the lamination. On the other hand, for a purely atomic bending, Lemma~\ref{lm: atomic_roof} implies that the ray $l'$ is either eventually bounded away from any atom, or asymptotic to an atom. In either case, it does have finite transverse measure.

Finally, we use the description of the convex core boundary to finish the proof of Theorem~\ref{prop:exotic_circle_necessary}:

\begin{proof}[Proof of Theorem~\ref{prop:exotic_circle_necessary}]
We adopt the same setup and notations as in the discussion preceding the proof of the previous lemma.

The support lines of $l'$ limit to a half circle to the left of the imaginary axis with one endpoint at the origin. This limiting half circle is either $l$ or a circle lying above $l$. The corresponding support planes limit to planes with boundary circles meeting the limit set at $p$. Since $p$ is not the endpoint of any leaf, all these boundary circles must be tangent to $C$ at $p$. Moreover, as $C$ is an osculating circle, these limit circles must be exactly $C$.

We claim that the ray does have finite transverse measure with respect to the bending lamination $\mathcal{L}$ in $X$. Indeed, fix a support plane $\tilde Q$ so that $\tilde Q\cap\tilde P'$ intersects $l$ to the right of the peak of $l$. Let $x$ be a point on the intersection of this support line and $l'$. The support plane to every point post $x$ on $l'$ towards $p$ intersects $\tilde Q$. In particular, the transverse measure of $l'$ from $x$ to $p$ is bounded above by the bending angle $\theta$ formed by $\tilde Q$ and $\tilde P$.
\begin{figure}[htp]
    \centering
    \includegraphics[width=0.7\linewidth]{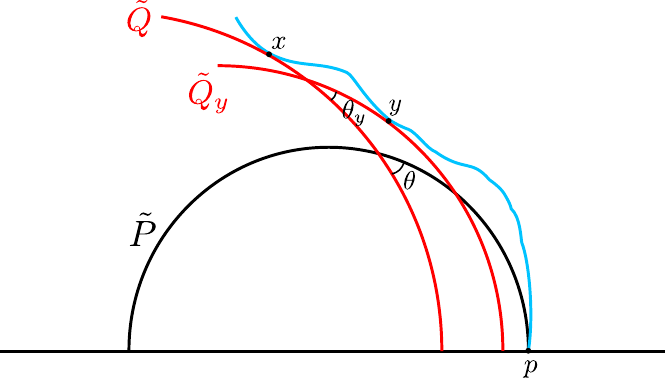}
    \caption{The bending measure between $x$ and $y$ is bounded above by $\theta$}
    \label{fig:angle_approx}
\end{figure}

Indeed, let $y$ be any point post $x$ on $l'$ towards $p$, and let $\tilde Q_y$ be the support plane at $y$. Suppose the ridge angle formed by $\tilde Q$ and $\tilde Q_y$ is $\theta_y$ (see Figure~\ref{fig:angle_approx}). Then by basic hyperbolic geometry, $\theta>\theta_y$. On the other hand, by \cite{bending}, the transverse measure of the segment between $x$ and $y$ with respect to the bending lamination is bounded above by $\theta_y$, and hence by $\theta$. Since $y$ is arbitrary, the total transverse measure of the ray is bounded above by $\theta$ as well.

These discussions then imply that $p$ lies in the halo of the lamination.
\end{proof}

From the second paragraph of the proof above, we extract following lemma, which will be helpful in \S\ref{sec:geodesic_planes_irrational}.
\begin{lm}\label{lm:exotic_support_limit}
    Any exotic circle is a limit of support circles. Equivalently, any exotic roof is a limit of support planes.
\end{lm}

\subsection{Constructing exotic roofs: a sufficient condition}\label{sec:sufficient}
Throughout this section, we assume that the bending lamination on $\partial E_+$ is minimal and irrational. Under this assumption, for any point $x\in\partial E_+$, there exists a unique support plane to $\partial E_+$ at $x$ (see \S\ref{sec: quasifuchsian}). Generally, we remark that the arguments in this subsection apply to a minimal component of the bending lamination. Let $X\cong\partial E_+$ be the hyperbolic surface isometric to $\partial E_+$, and $(\mathcal{L},\mu)$ the corresponding measured geodesic lamination.

Recall that the complement of $\mathcal{L}$ in $X$ consists of finitely many components. For simplicity, assume they are all ideal triangles. All ideal triangles are isomorphic; in Figure~\ref{fig:ideal_triangle}, we put the ideal points at infinity at $0$, $1$, and $\infty$. Consider the horizontal segment crossing from one side of the ideal triangle to the other side which has length $1$ (marked red in the figure). We call such a horizontal segment a \emph{unit crossing}. Note that there are three unit crossings, each for an ideal vertex. We refer to this vertex as the \emph{vertex of the crossing}.
\begin{figure}[htp]
\centering
\includegraphics{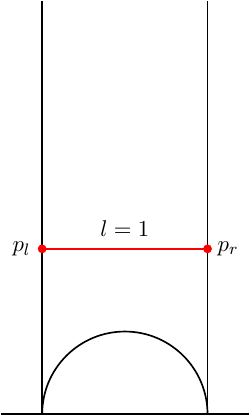}
\caption{An ideal triangle in the upper half plane}
\label{fig:ideal_triangle}
\end{figure}

Fix a complementary ideal triangle $T$ and a unit crossing on $T$ whose left and right endpoints are $p_l,p_r$. A \emph{sequence of leaf approximations} for $(\mathcal{L},\mu)$ consists of leaf segments $L_n$ so that
\begin{enumerate}[topsep=0mm, itemsep=0mm]
    \item Each leaf segment starts at the left endpoint $p_l$, and goes away from the vertex of the crossing;
    \item Each leaf segment $L_n$ ends at a point $p_n$ close to the right endpoint $p_r$;
    \item The length $d_n$ of the geodesic segment $L_n$ goes to infinity as $n\to\infty$;
    \item The transverse measure $\alpha_n$ of the geodesic segment $l_n$ connecting $p_n$ and $p_r$ goes to $0$.
\end{enumerate}
Let $C_n$ be the closed curve formed by concatenating $L_n$, $l_n$, and the unit crossing. Note that $C_n$ has length roughly $d_n$ and transverse measure $\alpha_n$, and in the space of geodesic currents, $C_n/d_n\to\mathcal{L}$. It is often easier to think about these objects in the universal cover.

A sequence of leaf approximations $\{L_n\}$ is called
\begin{itemize}[topsep=0mm, itemsep=0mm]
    \item \emph{exotic} if connecting $C_n$'s at $p_l$ in order and then straightening it gives an exotic ray;
    \item \emph{$f(d)$-separating} for some non-decreasing function $f$ if there exist constants $C>c>0$ so that the hyperbolic length of $l_n$ is between $ce^{-d_n}f(d_n)$ and $Ce^{-d_n}f(d_n)$; for convenience, \emph{at least} (resp.~\emph{at most}) $f(d)$-separating means only the lower bound (resp.~upper bound) holds;
    \item \emph{$g(d)$-good} if the transverse measure $\alpha_n$ is $O(g(d_n))$;
    \item \emph{CESAG} (short for compatibly exotic, separating, and good) if it is exotic, $f(d)$-separating and $g(d)$-good for some $f(d), g(d)$ so that $f(d_n)g(d_n)\to 0$.
\end{itemize}
To conceptualize the separating condition, we note the following lemma:
\begin{lm}
    An exotic sequence of leaf approximations is at most $e^d$-separating.
\end{lm}
\begin{proof}
For the straightened ray to be exotic, we need the total transverse measure to be finite. Thus the transverse measure of $l_n$ goes to zero, and hence the hyperbolic length of $l_n$ goes to zero; at most $e^{d}$-separation follows.
\end{proof}
We remark that it is easy to construct an exotic sequence of leaf approximations. Indeed, in our construction of exotic rays in Part I (see the proof of Theorem~\ref{mainprime}), each inadmissible word can be represented by a leaf segment, followed by a crossing, and then another leaf segment. Gluing these words together may require additional crossings, but clearly we can obtain a sequence of leaf approximations.

We have the following sufficient condition for the existence of an exotic roof:
\begin{thm}\label{thm: sufficient}
Let $(\mathcal{L},\mu)$ be the bending lamination for an end $E_+$ of a quasifuchsian manifold $M$. Suppose $(\mathcal{L},\mu)$ is minimal and irrational, and has a CESAG sequence of leaf approximations, then there exist uncountably many exotic roofs for $E_+$.
\end{thm}

We again remark that a similar result holds if a minimal component of $\mathcal{L}$ has a CESAG sequence of leaf approximations. Since any exotic sequence of leaf approximations is at most $e^{d}-$separating, any exotic and good leaf approximation as defined in the introduction (i.e. taking $\alpha_n=o(e^{-d_n})$) is CESAG.

The exotic roof in the proof of Theorem~\ref{thm: sufficient} will be constructed as the limit of support planes to lifts of leaf approximations. In particular, there is a geodesic ray going to infinity towards the unique point where the exotic roof meets the limit set, constructed by concatenating lifts of leaf approximations. The exotic property guarantees that this gives an exotic ray. As in Part~\ref{part:halo}, the abundance of choices gives uncountability. 

\paragraph{Infinitesimal approximation of bending measures}
For each segment $l_n$ on $X$, consider the corresponding segment on $\partial E_+\cong X$, still denoted by $l_n$. Fix any lift $\tilde T$ of $T$, and let $\tilde l_n$ be the corresponding lift of $l_n$, with endpoints $\tilde p_r$ and $\tilde p_n$. Suppose the support planes containing $\tilde p_r$ and $\tilde p_n$ are $\tilde P_0$ (this also contains $\tilde T$) and $\tilde P_n$, and assume the ridge angle formed by the two support planes is $\beta_n$ (one needs to assume that $\alpha_n$ is small enough so that the two support planes do intersect; but this is always possible by passing to a subsequence). The following lemma implies that $\beta_n/\alpha_n\to1$ as $n\to\infty$.
\begin{lm}
    Let $\partial\tilde E_+$ be the component of $\partial\hull(\Lambda)$ covering $\partial E_+$. For any $x\in\partial\tilde E_+$, let $\tilde P_x$ be the support plane at $x$. For any $x,y\in\partial\tilde E_+$ close enough so that support planes intersect, let $\beta_{x,y}$ be the ridge angle formed by $\tilde P_x$ and $\tilde P_y$, and $\alpha_{x,y}$ be the bending measure of the geodesic connecting $x$ and $y$. Then $\beta_{x,y}/\alpha_{x,y}\to 1$ as $y\to x$ in the induced metric on $\partial E_+$.
\end{lm}
\begin{proof}
    Let $S_i$ be a chain of support planes between $x$ and $y$ with $S_0=\tilde P_x$ and $S_N=\tilde P_y$, obtained via the process of ``raising the roof" as discussed in \S\ref{sec: quasifuchsian}; see Figure~\ref{fig:infinitesimal}. Note that for $i=1,\ldots,N-2$, the three geodesic planes $S_i,S_{i+1},S_N$ pairwisely intersect but do not all intersect at a common point (by \cite[Lemma II.1.9.2]{bending}), so there exists a unique geodesic plane $P_i$ orthogonal to all three (see e.g.\ \cite[Lemma II.1.10.1]{bending}). Then the geodesics $S_i\cap P_i, S_{i+1}\cap P_i, S_N\cap P_i$ form a geodesic triangle, which we denote by $T_i$.
    
    \begin{figure}[htp]
    \centering
    \includegraphics[width=0.7\linewidth]{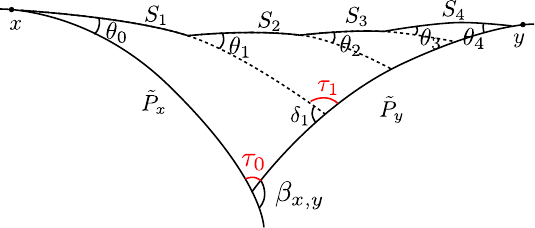}
    \caption{The angles formed by a chain of support planes}
    \label{fig:infinitesimal}
    \end{figure}
    
    Let the bending angle formed by $S_i$ and $S_{i+1}$ be $\theta_i$. In the triangle $T_i$, one of the angle is equal to $\theta_i$. Let the angle formed by $S_i\cap P_i$ and $S_N\cap P_i$ be $\tau_i$, and the angle formed by $S_{i+1}\cap P_i$ and $S_N\cap P_i$ be $\delta_{i+1}$. Then clearly $\delta_{i+1}+\tau_{i+1}=\pi$, since they are the two complementary dihedral angles  formed by $S_{i+1}$ and $S_N$. Moreover, $\theta_{N-1}=\delta_{N-1}$, and $\beta_{x,y}+\tau_0=\pi$.
    
    By elementary hyperbolic geometry $\text{Area}(T_i)=\pi-\theta_i-\tau_i-\delta_{i+1}$. So
    \begin{equation}
        \sum_{i=1}^{N-2}\text{Area}(T_i)=(N-2)\pi-\sum_{i=1}^{N-2}\theta_i-\tau_0-(N-3)\pi-\delta_{N-1}=\beta_{x,y}-\sum_{i=1}^{N-1}\theta_i.\tag{$\ast$}
    \end{equation}

    Let $l_x$ be the length of the side of the triangle $T_0$ lying on $\tilde P_x$, and $l_y$ be the sum of the lengths of the sides of the triangles $T_i$ lying on $\tilde P_y$. Although the triangles $T_i$ do not necessarily lie on the same geodesic plane, since the angles at the same ridge line are complementary, we can map them to a single hyperbolic plane isometrically so that the sides on $\tilde P_y$ match up to a geodesic segment of length $l_y$. That is, we may treat angles and sides in Figure~\ref{fig:infinitesimal} as if they lie on the same plane.
    
    Consider the hyperbolic triangle with an angle $\tau_0$, whose two adjacent sides have lengths $l_x$ and $l_y$. Then by construction $\sum_{i=1}^{N-2}\text{Area}(T_i)$ is bounded above by the area of this triangle, and by ($\ast$) so is $\beta_{x,y}-\sum_{i=1}^{N-1}\theta_i$.

    Taking a sequence of chains of support plane to approximate the bending measure of the segment between $x$ and $y$ (i.e.\ we have $\sum_{i=1}^{N-1}\theta_i\to\alpha_{x,y}$), we have $l_x$ tends to the distance $d_x$ between $x$ and the ridge line $\tilde P_x\cap\tilde P_y$, and $l_y$ is bounded above by the distance $d_y$ between $y$ and the same ridge line. Let $T_{x,y}$ be the triangle with an angle $\tau_0$ whose two adjacent sides have lengths $d_x$ and $d_y$. The by the arguments above,
    $$\beta_{x,y}-\alpha_{x,y}\le\text{Area}(T_{x,y})$$

    When $x,y$ are close enough, we can approximate the area of $T_{x,y}$ by a Euclidean triangle with the same angle $\tau_0$ whose two adjacent sides have lengths $d_x$ and $d_y$. That is, there exists a constant $C>0$ so that $\text{Area}(T_{x,y})\le Cd_xd_y\sin\tau_0=Cd_xd_y\sin\beta_{x,y}$. Hence
    $$1\ge\frac{\alpha_{x,y}}{\beta_{x,y}}=1-\frac{\beta_{x,y}-\alpha_{x,y}}{\beta_{x,y}}\ge 1-Cd_xd_y\frac{\sin\beta_{x,y}}{\beta_{x,y}}.$$
    When $y\to x$ in the induced metric on $\partial E_+$, we have $d_x,d_y,\beta_{x,y}\to0$. So the right-hand side in the inequality above goes to $1$ as well. Thus $\beta_{x,y}/\alpha_{x,y}\to 1$ as $y\to x$, as desired.
\end{proof}

\paragraph{Some lemmas in hyperbolic geometry}
We start with a few lemmas approximating the size of circles and spheres. For this we use the upper half space model throughout the section. Given a totally geodesic plane $\tilde P$ in $\mathbb{H}^3$, let $r(\tilde P)$ be the (Euclidean) radius of the corresponding half sphere. Let $\mathcal{G}$ be a complete geodesic on $\tilde P$, and $r(\mathcal{G})$ the radius of the corresponding half circle. Necessarily $r(\mathcal{G})\le r(\tilde P)$. Let $O_{\mathcal{G}}$ be the highest point on $\mathcal{G}$ with respect to the Euclidean coordinates in the upper half space model. We orient $\mathcal{G}$ so that along the direction of $\mathcal{G}$, the smaller portion of $\tilde P-\mathcal{G}$ is on the right (if $r(\tilde P)=r(\mathcal{G})$, then choose either direction).

Let $x_d$ be the point on $\mathcal{G}$ of hyperbolic distance $d$ from $O_{\mathcal{G}}$ in the chosen direction. We can find a unique ideal triangle $T_{d,k}$ and $k$-crossing\footnote{Recall that we have defined ``unit crossing"; an $k$-crossing is a crossing with length $k$ instead of $1$.} so that $x_d$ is the right endpoint of the crossing, and the vertex of the crossing is the starting point (at infinity) of the oriented complete geodesic $\mathcal{G}$. Suppose the left endpoint lies on a side $\mathcal{G}_{d,k}$ of $T_{d,k}$. We note that $\mathcal{G}_{d,k}=\mathcal{G}_{d-\log k,1}$. We have
\begin{lm}
The radius $r(\mathcal{G}_{d,k})=r(\mathcal{G})/\sqrt{1+2ke^{-d}\sin\theta+k^2e^{-2d}}$, where
$$\sin\theta=\sqrt{1-r(\mathcal{G})^2/r(\tilde P)^2}.$$
In particular,
$$r(\mathcal{G}_{d,k})\ge r(\mathcal{G})/(1+ke^{-d}).$$
\end{lm}
\begin{proof}
Consider the M\"obius transform $\alpha$ sending $\infty,0,-1$ to $-r(\tilde P)e^{i\theta},r(\tilde P)e^{-i\theta},-r(\tilde P)i$. The desired circle $\mathcal{G}_{d,k}$ is the image via $\alpha$ of the vertical line over $-ke^{-d}$. It is then routine to calculate its radius.
\end{proof}

Now suppose $\mathcal{G}_{d,k}'$ is the side on the right end of an $k$-crossing from $x_d$ (so the orientation of $\mathcal{G}$ needs to be reversed), we have
\begin{lm}\label{lm: hypergeom2}
The radius $r(\mathcal{G}_{d,k}')=r(\mathcal{G})/\sqrt{1+2ke^{d}\sin\theta+ke^{2d}}$. In particular
$$r(\mathcal{G}_{d,k}')\ge r(\mathcal{G})/(1+ke^{d}).$$
\end{lm}
\begin{proof}
This follows from the previous lemma and symmetry.
\end{proof}

Let $x_d'$ be the right end of the $k$-crossing in the lemma above. Consider a complete geodesic $\mathcal{G}_{d,k}''$ passing through $x_d'$ that does not cross $\mathcal{G}$ (in particular $\mathcal{G}_{d,k}'$ is such a geodesic). Then we have
\begin{lm}\label{lm: hypergeom3}
Suppose $k<1$. Then for all $d$ large enough,
$$r(\mathcal{G}_{d,k}'')\ge r(\mathcal{G}_{d,k}')\ge r(\mathcal{G})/(1+ke^{d}).$$
\end{lm}
\begin{proof}
Indeed, this follows from the fact that when $k<1$ and $d$ large enough, the end of the $k$-crossing is between the ``peak" of $\mathcal{G}_{d,k}'$ and the common endpoint of $\mathcal{G}$ and $\mathcal{G}_{d,k}'$. To see this, note that when $d$ is large enough, the configuration is approximated by Figure~\ref{subfigure: plane} in the upper half plane. Then the claim easily follows from the fact that $k<1$.
\end{proof}

\begin{figure}[htp]
\centering
\begin{subfigure}[t]{0.45\linewidth}
\centering
\includegraphics{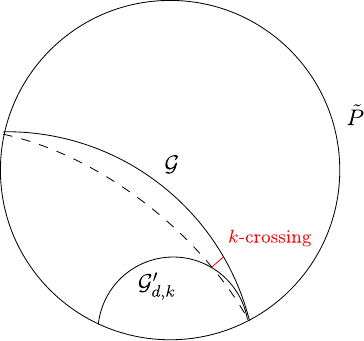}
\caption{The configuration in plane $\tilde P$}
\label{subfig: circle}
\end{subfigure}
\begin{subfigure}[t]{0.45\linewidth}
\captionsetup{width=\linewidth}
\includegraphics{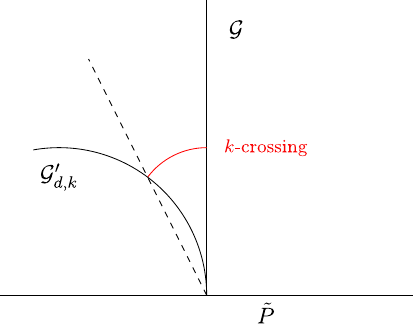}
\caption{When $d$ is large enough, the configuration can be approximated by this picture in upper half plane}
\label{subfigure: plane}
\end{subfigure}
\caption{An illustration of Lemma~\ref{lm: hypergeom3}}
\end{figure}

The next lemma relates the radius of a plane bent from another at an angle $\alpha$. Let $\tilde P$ and $\mathcal{G}$ be the same as above. Let $\tilde P'$ be a geodesic plane containing $\mathcal{G}$, so that the portion of $\tilde P'$ to the right of $\mathcal{G}$ and the portion of $\tilde P$ to the left of $\mathcal{G}$ form an angle of $\pi-\alpha$. Then
\begin{lm}
The radius $r(\tilde P')=\dfrac{r(\mathcal{G})}{\sin\alpha\sin\theta+\cos\alpha\cos\theta}=\dfrac{r(\tilde P)r(\mathcal{G})}{\sqrt{r(\tilde P)^2-r(\mathcal{G})^2}\sin\alpha+r(\mathcal{G})\cos\alpha}$
\end{lm}
\begin{proof}
This is simple trigonometry.
\end{proof}

\paragraph{Construction of exotic roofs}
We now go back to our setting. We will inductively choose a subsequence of the leaf approximations, so that the plane $\tilde P_n$ containing the lift of $L_n$ has radius bounded below by a uniform constant. Then $\tilde P_n\to \tilde P$ a geodesic plane whose boundary at infinity touches the limit set at a point in the halo of the lamination. Finally, as in the case of exotic rays, the abundance of choices gives uncountably many. 

We start by lifting $T$ to an ideal triangle $\tilde T_1$ contained in a plane $\tilde P_1$ of Euclidean radius $r_1$, so that the corresponding lift $\tilde p_l$ of $p_l$ lies at the top of the sphere. Denote the side of $\tilde T_1$ that $\tilde p_l$ lies on by $\mathcal{G}_1$. Let $n_1$ be a number large enough to be determined later. Take the lift $\tilde L_{n_1}$ of $L_{n_1}$ starting at $\tilde p_l$ and let $\tilde l_{n_1}$ be the corresponding lift of $l_{n_1}$ starting at the endpoint of $\tilde L_{n_1}$. Then $\tilde l_{n_1}$ ends at a side $\mathcal{G}_1'$ of an ideal triangle (this is simply another lift of $T$). We claim
\begin{equation}\label{eq: radius_lower_bound}
    r(\mathcal{G}_1')\ge \frac{r(\mathcal{G}_1)}{1+Cf(d_{n_1})}
\end{equation}
if $n_1$ is large enough. Here $C$ is the constant appearing in the definition of $f(d)$-separation. Indeed, this follows from Lemma~\ref{lm: hypergeom3} and the fact that the length of $l_{n_1}$ is bounded above by $Ce^{-d_n}f(d_n)$.

In the convex core boundary, the corresponding lift of $T$ lies in a sphere $\tilde P_2$ bent from $\tilde P_1$ in a ridge line $\mathcal{G}_{1,2}$ between $\mathcal{G}_1$ and $\mathcal{G}_1'$. Note that $r(\mathcal{G}_{1,2})$ satisfies the same lower bound (\ref{eq: radius_lower_bound}). Indeed, it crosses $\tilde l_{n_1}$ but not $\mathcal{G}_1$, so Lemma~\ref{lm: hypergeom3} applies. We have
$$r_2=r(\tilde P_2)=r_1\frac{r(\mathcal{G}_{1,2})}{\sqrt{r_1^2-r(\mathcal{G}_{1,2})^2}\sin\beta_{n_1}+r(\mathcal{G}_{1,2})\cos\beta_{n_1}}$$
Since $\sin(\beta_{n_1})\approx \beta_{n_1}\le D\alpha_{n_1}$, $\alpha_{n_1}=O(g(d_{n_1}))$, and $r(\mathcal{G}_{1,2})$ satisfies lower bound (\ref{eq: radius_lower_bound}), the ratio on the right-hand side is bounded below by a number that goes to $1$ as $n_1\to\infty$ (here we use the assumption $f(d)g(d)\to0$). Thus by choosing $n_1$ even larger, we may arrange $r_2>r_1/2$.

In $\tilde P_2$, denotes the corresponding lift of $T$ by $\tilde T_2$, its left side $\mathcal{G}_2$ and right side $\mathcal{G}_1''$. Note that $r(\mathcal{G}_1'')\approx r(\mathcal{G}_1')$ (actually slightly larger), but we do not need any estimates on its radius.

Note that the starting point is not at the top of the circle, but already to the side. Suppose the distance to the top of the circle is $d_2$. By the same arguments as above, we can choose $n_2$ large enough so that the corresponding lift of $l_{n_2}$ ends at a side $\mathcal{G}_2'$ of an ideal triangle satisfying
$$r(\mathcal{G}_2')\ge\frac{r(\mathcal{G}_2)}{1+Ce^{d_2}f(d_{n_2})}.$$
As in the previous part, we get a sphere $\tilde P_3$ bent from $\tilde P_2$ in a ridge line $\mathcal{G}_{2,3}$ between $\mathcal{G}_2$ and $\mathcal{G}_2'$, and
$$r_3=r(\tilde P_3)=r_2\frac{r(\mathcal{G}_{2,3})}{\sqrt{r_2^2-r(\mathcal{G}_{2,3})^2}\sin\beta_{n_2}+r(\mathcal{G}_{2,3})\cos\beta_{n_2}}$$
Again, by choosing $n_2$ even larger, we can get $r_3>r_1/2$.

It is clear we can do this inductively, so we obtain a sequence of support planes $\tilde P_k$ with radius bounded below. They limit to an exotic roof as explained at the start of the section. This concludes the proof of Theorem~\ref{thm: sufficient}.

\section{Geodesic planes outside convex core II: atom-free bending}\label{sec:geodesic_planes_irrational}
In this section, we turn to the case where $\mathcal{L}$ has no atom, i.e. $\mathcal{L}_a=\varnothing$, and prove Theorem~\ref{thm: minimal_roof}.
We start with the following lemma:
\begin{lm}\label{lm: contains_support}
Let $P$ be a non-parabolic asymptotic plane. Then the closure of $P$ contains a support plane.
\end{lm}
Recall that a geodesic plane is parabolic if its boundary circles meet the limit set exactly at a parabolic fixed point. We remark that in the case of purely atomic bending, we already know this lemma due to the complete classification in \S\ref{sec:geodesic_planes_atomic}.
\begin{proof}
Let $\tilde P$ be any lift of $P$, and $C$ its boundary circle. Let $r_1:[0,\infty)\to\mathbb{H}^3$ be a geodesic ray contained in $\hull(\Lambda)$ whose endpoint is $C\cap\Lambda$, and $r_2:[0,\infty)\to\mathbb{H}^3$ be a geodesic ray on $\tilde P$ ending at $C\cap \Lambda$. Since $r_1$ and $r_2$ ends at the same point, up to reparametrization, we may assume the hyperbolic distance between $r_1(t)$ and $r_2(t)$ goes to zero.

Since $p=C\cap\Lambda$ is not a parabolic fixed point, the projection of the geodesic $r_1(t)$ to $\core(M)$ returns to the thick part of $\core(M)$ infinitely often. Since $M$ is geometrically finite, the thick part of $\core(M)$ is compact. Hence there exists a compact subset $\Pi$ of $\hull(\Lambda)$ and a sequence of numbers $t_i\to\infty$ so that $r_1(t_i)\in\gamma_i\Pi$ for some $\gamma_i\in\Gamma$. Thus $\gamma_i^{-1}r_1(t_i)\in\Pi$ and $\gamma_i^{-1}r_2(t_i)$ limits on $\Pi$. By compactness, passing to a subsequence if necessary, we may assume $\gamma_i^{-1}r_2(t_i)\to x\in\Pi$.

Now $\gamma_i^{-1}\tilde P$ contains $\gamma_i^{-1}r_2(t_i)$, so $\gamma_i^{-1}\tilde P$ limits to a geodesic plane $P_{\lim}$ passing through $x$. This limit plane cannot meet the interior of $\hull(\Lambda)$, as none of $\gamma_i^{-1}\tilde P$ meets the interior. Since $x\in\hull(\Lambda)$, we must have $\varnothing\neq P_{\lim} \cap\hull(\Lambda)\subseteq\partial\hull(\Lambda)$, i.e.\ $P_{\lim}$ is a support plane.
\end{proof}
Thus it is important to understand the closure of support planes. Proposition~\ref{prop: support_flat} states that if the bending lamination has no atom, the closure of any support plane contains all support planes. And Corollary~\ref{cor: support_closure} states that then the closure of any support plane is the entire end $E_+$.

In the remaining part, we describe the closure of the corresponding circle in the space of circles, or equivalently, the closure of frames tangent to the plane in the frame bundle (see Equation (1.1) in \cite{MMO1} and the discussion around it). This provides finer details on the behavior of the closure.

We start with the following lemma.
\begin{lm}\label{lm: enclosed}
Let $C$ be a support circle. Then $\overline{\Gamma\cdot C}$ contains all enclosed asymptotic circles.
\end{lm}
Recall here that an asymptotic circle shadowed by a support plane is stemmed if it meets the limit set at the endpoint of a leaf, and enclosed otherwise.
\begin{proof}
Assume $C'$ is an enclosed asymptotic circle. We may assume $C'$ is shadowed by $C$, as $\overline{\Gamma\cdot C}$ contains all support planes when $\mathcal{L}_a=\varnothing$. The set $\overline{\Gamma\cdot C}$ contains a continuous family of support planes distinct from $C$ and limiting to $C$ (we may construct this continuous family along a path in $\partial\tilde E_+$; see the proof of Corollary~\ref{cor: support_closure}), and by \cite[Cor.~3.2]{MMO1}, if the stabilizer of $C$ in $\Gamma$ is nonelementary, we immediately conclude that $C'$ is contained in $\overline{\Gamma\cdot C}$ as well. As a matter of fact, even when the stabilizer of $C$ is elementary, arguing similarly as \cite[Cor.~3.2]{MMO1}, we can also conclude that $C'$ is contained in $\overline{\Gamma\cdot C}$.
\end{proof}

Whether $\overline{\Gamma\cdot C}$ contains other circles than support and enclosed ones depends on whether $E_+$ contains an exotic roof or not. We analyze the two cases separately.

\paragraph{If exotic roofs exist..}
Assume there exists an exotic roof. To avoid complications, let us also assume that $\mathcal{L}$ is minimal. The case of several minimal components is discussed at the end of this section. Our goal in this part is to prove the following:
\begin{thm}\label{prop: support_closure_frame}
Suppose the bending lamination $\mathcal{L}$ of the end $E_+$ is minimal and atom-free. Suppose also that $E_+$ contains an exotic roof. Let $C$ be either a support circle or a non-parabolic asymptotic circle. Then the closure of $\Gamma\cdot C$ consists of every support and asymptotic circle.
\end{thm}
Note that by Lemma~\ref{lm:exotic_support_limit}, if $C$ is an exotic circle (i.e.~a boundary circle of the exotic roof), then $C$ is the limit of a sequence $\{C_n\}$ of support circles. We have:
\begin{lm}\label{lm:special_roof_support_limit}
Let $C$ be an exotic circle, and $C'$ a circle contained in $\overline{\Omega_+}$ and tangent to $C$ at $p=C\cap\Lambda$. Then $C'$ is a limit of support circles.
\end{lm}
\begin{proof}
Let $x$ be the point on $C'$ farthest away from $C$, and set $d$ to be the distance between $x$ and $C$ (this distance may be taken as either Euclidean or spherical, as long as the side of $C$ contained in $\Omega_+$ is a small enough spherical cap).

Let $C_n$ be a sequence of support circles so that $C_n\to C$. Choose a point $p_n\in C_n\cap\Lambda$ so that $p_n$ is the endpoint of a bending line. Let $p_n'$ be the other end of the bending line. Clearly $p_n, p_n'\to p$. Let $C'_n$ be the circle contained in $\overline{\Omega_+}$ and tangent to $C_n$ at $p_n$, so that the point $x_n$ on $C'_n$ farthest away from $C_n$ has distance $d$ from $C_n$.

As in the proof of Corollary~\ref{cor: support_closure}, there is a continuous family of support circles starting from $C_n$ and moving away across $p_np_n'$. For example, let $\tilde P_n$ be the corresponding geodesic plane, and $q$ any point on the geodesic $p_np_n'$. Take a path in $\partial\tilde E_+$ starting at $q$ and leaving $\tilde P_n\cap\partial\tilde E_+$. Then the family of support planes along the path gives the desired family.

By choosing this path long enough, eventually the support plane to the point at the end of the path is disjoint from $\tilde P_n$. The corresponding circle is then disjoint from $C_n$. Thus we have a continuous family of circles starting at $C_n$, moving cross $p_np_n'$ and eventually disjoint from $C_n$. By continuity, we can find $C''_n$ in this family that passes through $x_n$ (see Figure~\ref{fig:cont_circles}). Since $C_n''$ intersects $p_np_n'$, its limit is tangent to $C$ at $p$. It also passes through $x$, and thus $C_n''\to C'$, as desired.
\end{proof}
\begin{figure}[htp]
    \centering
    \captionsetup{width=.6\textwidth}
    \includegraphics[width=.4\textwidth]{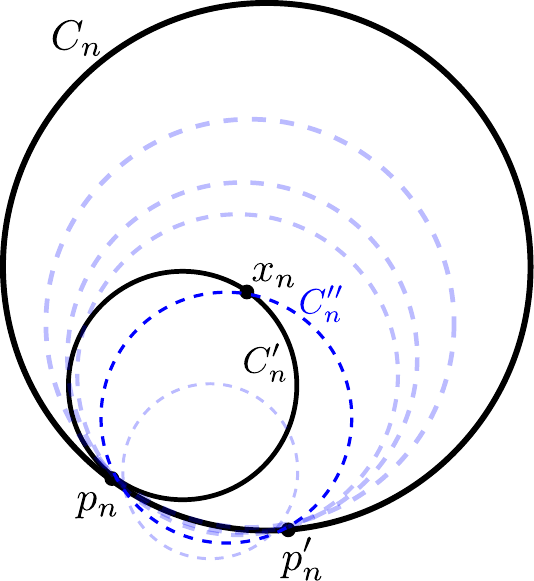}
    \caption{A continuous family of circles sweeping through the interior of $C_n$. One of the must pass through $x_n$.}
    \label{fig:cont_circles}
\end{figure}

\begin{rmk}\label{rmk: no.minimality}
Note that the only property of the exotic circle $C$ we used here is that there exists a sequence of support circles $C_n$ so that all points in $C_n\cap\Lambda$ limit to $C\cap\Lambda$. In particular, minimality of $\mathcal{L}$ is not needed here.
\end{rmk}
Therefore the $\Gamma$-orbit of a support plane of $\partial\tilde E_+$ limits on \emph{every} asymptotic plane shadowed by an exotic roof. Recall that Lemma~\ref{lm: contains_support} states that the orbit $\Gamma\cdot C$ of a non-parabolic asymptotic circle accumulates on a support plane. As every parabolic circle is shadowed by a support circle (see Remark~\ref{rmk:parabolic_support}), an exotic circle $C$ is not parabolic. We claim
\begin{lm}\label{lm:exotic_limit_to_leaf}
Let $C$ be an exotic circle. There exists a sequence $\{C_n\}\subset\Gamma\cdot C$ limiting to a support plane $C'$, so that $p_n=C_n\cap\Lambda$ limits to an endpoint of a leaf.
\end{lm}
\begin{proof}
The proof is a careful analysis supplementing Lemma~\ref{lm: contains_support}. Fix a geodesic ray $l$ on the exotic roof $P$ with boundary circle $C$ tending to the unique point $p=C\cap \Lambda$. Fix also a fundamental domain $F$ of the action of $\Gamma$ in $\partial\tilde E_+$. 

Let $r:[0,\infty)\to\partial E_+$ be an exotic ray in $\partial E_+$ ending at $p$. Then by Lemma~\ref{lem:angle}, we can find $t_n\to\infty$ so that $r(t_n)$ lies on a leaf, and the tangent vectors to the ray $r$ and the leaf at $r(t_n)$ form an angle $\theta(t_n)\to0$ as $t_n\to\infty$. Moreover, we can find $x_n\in l$ so that $d_{\mathbb{H}^3}(x_n,r(t_n))\to 0$, as $l$ and $r$ are asymptotic.

Then we can choose $\gamma_n\in\Gamma$, so that $\gamma_n x_n\in F$. Taking $n\to\infty$, we have $\gamma_n x_n$ tends to a point $x$ on the leaf, and the tangent vector along $\gamma_nl$ pointing towards $\gamma_np$ at $\gamma_nx_n$ tends to the tangent vector along the leaf at $x$. It then follows that $\gamma_np$ limits to the endpoint of the leaf, as desired.
\end{proof}

We are now ready to prove Theorem~\ref{prop: support_closure_frame}.
\begin{proof}[Proof of Theorem~\ref{prop: support_closure_frame}]
Since the orbit of every non-parabolic asymptotic circle accumulates on a support plane by Lemma~\ref{lm: contains_support}, we may assume $C$ is a support plane. Note that all enclosed asymptotic circles lie in the closure by Lemma~\ref{lm: enclosed}. So it suffices to show the same for any stemmed asymptotic circle.

By Lemma~\ref{lm:special_roof_support_limit}, $\overline{\Gamma\cdot C}$ contains every asymptotic circle shadowed by an exotic circle $C_e$. These circles form a one-dimensional family, all tangent to each other at the unique point $p_e$ in $C_e\cap\Lambda$. By the previous lemma, a sequence $\gamma_nC_e$ limits to a support circle $C'$ so that $\gamma_np_e$ tends to the endpoint $p_0$ of a leaf $l_0$. In particular, by considering the one-dimensional family of circles at $\gamma_np_e$, we conclude that all asymptotic circles at $p_0$ lie in the orbit closure as well. 

First assume the bending lamination $\mathcal{L}$ is nonorientable. Then for any leaf $l$ of $\tilde{\mathcal{L}}$ with one of its endpoints $p$, there exists $\gamma_n\in \Gamma$ so that $\gamma_nl_0\to l$ (in the Hausdorff topology on $\mathbb{H}^3\cup S^2$) and $\gamma_np_0\to p$. Indeed, fix $x\in l$, and a unit tangent vector $v_x$ of $l$ at $x$ pointing towards $p$. Since the lamination is minimal and nonorientable, every leaf in $\mathcal{L}$ is dense, and comes arbitrarily close to a fixed point in both directions. Thus there exists $x_n\in l_0$ and unit tangent vector $v_{x_n}$ of $l_0$ at $x_n$ pointing towards $p_0$, so that under the universal covering $\pi:\mathbb{H}^3\to M$, $\pi(x_n)\to\pi(x)$, and $\pi_*v_{x_n}\to\pi_*v_x$. Hence in the universal cover, we can find $\gamma_n$ so that $\gamma_nx_n\to x$ and $(\gamma_n)_*v_{x_n}\to v_x$. Therefore the forward endpoint $\gamma_np_0$ of $\gamma_nl_0$ tends to that of $l$, i.e.\ $p$; the same is true for the backward endpoint. This in turn implies that $\gamma_nl_0\to l$.

A similar argument as before then implies that all asymptotic circles at $p$ lie in the orbit closure, by considering the one-dimensional family of circles at $\gamma_np_0$.

When the lamination is orientable, the argument above gives asymptotic circles at only one endpoint for each leaf. For simplicity, call such endpoints \emph{marked}. On the other hand, the complement of the lamination $\mathcal{L}$ consists of even-sided ideal polygons, even-sided punctured ideal polygons, and subsurfaces decorated with crowns with even spikes. Let $\tilde Q$ be a lift of any such complementary region $Q$, and choose two leaves $l_1,l_2$ bounding $\tilde Q$ with a shared marked endpoint $p_0$, and let $p_1,p_2$ be their unmarked endpoints respectively. Let $q_0$ be another marked endpoint of $\tilde Q$ distinct from $p_0$. 

For any leaf $l$ with marked endpoint $p$ and unmarked endpoint $p'$, fix $x\in l$, and let $v_x$ be the unit tangent vector at $x$ pointing toward $p$. By minimality and orientability, there exists $x_n\in l_1$, and unit tangent vector $v_{x_n}$ pointing toward $p_0$ so that $\pi(x_n)\to x$ and $\pi_*v_{x_n}\to v_x$. In fact, since any half leaf is already dense in $\mathcal{L}$, the points $x_n$ may be chosen so that $x_n\to p_0$. Since $l_1,l_2$ share the endpoint $p_0$, we can then find points $y_n$ on $l_2$ so that $d(x_n,y_n)\to0$. Hence in the universal cover, we can find $\gamma_n$ so that $\gamma_nx_n\to x$, $\gamma_ny_n\to x$, and $(\gamma_n)_*v_{x_n}\to v_x$. Thus as before $\gamma_nl_1\to l$ and $\gamma_nl_2\to l$ (in the Hausdorff topology on $\mathbb{H}^3\cup S^2$). In particular the unmarked endpoints $\gamma_np_1\to p'$ and $\gamma_np_2\to p'$. Notice that $\gamma_nq_0$ is sandwiched between $\gamma_np_1$ and $\gamma_np_2$, so $\gamma_nq_0\to p'$ as well. 

As before, all asymptotic circles at $p'$ lie in the orbit closure. So we finally conclude that all asymptotic circles at the unmarked endpoints are also included in the orbit closure.
\end{proof}

Notice that the minimal condition is crucial here.

\paragraph{If exotic roofs do not exist..}
We now turn to the case where exotic roofs do not exist. Again we assume that the bending lamination is minimal. The goal is to show
\begin{thm}\label{thm: no_exotic_roof_closures}
Suppose the bending lamination $\mathcal{L}$ of the end $E_+$ is minimal and atom-free. Suppose also that $E_+$ contains no exotic roof. Then
\begin{enumerate}[label=\normalfont{(\arabic*)}, topsep=0mm, itemsep=0mm]
\item The closure of any support plane is the union $\mathcal{S}$ of all support planes and all enclosed asymptotic planes;
\item The closure of any non-parabolic asymptotic plane $C$ is $\Gamma\cdot C\cup\mathcal{S}$.
\end{enumerate}
\end{thm}
We start with the following proposition, which can be viewed as the converse to Theorem~\ref{prop: support_closure_frame}.
\begin{prop}\label{prop:exotic_plane_contrad}
Suppose the bending lamination $\mathcal{L}$ of the end $E_+$ is minimal and atom-free. Let $C_i$ be a sequence of support circles for $E_+$, and assume $C_i\to C$ where $C$ is a stemmed asymptotic circle. Then $E_+$ contains an exotic roof.
\end{prop}
\begin{proof}
Since $C$ is a limit of support circles, following the argument of Lemma~\ref{lm:special_roof_support_limit} we conclude that every asymptotic circle at $C\cap\Lambda$ is a limit of support circles (see Remark~\ref{rmk: no.minimality}). As $C$ is stemmed, the conclusion of Lemma~\ref{lm:exotic_limit_to_leaf} is automatically true. Then following the argument of Theorem~\ref{prop: support_closure_frame}, we then conclude that every asymptotic circle is the limit of a sequence of support circles. In the remainder of the proof, this is the key fact we will use.

We construct an exotic circle in a way similar to what we have done in the last section, as the limit of a sequence of support circles. The key difficulty of this strategy lies in bounding the radii of the support circles from below. We overcame this difficulty in the last section by adding the CESAG condition. Here we have an abundance of circles in the limit of support circles. We can extract a sequence of support circles with radii bounded below simply by the fact that they limit to asymptotic circles with definite radii. We will now explain the construction in detail.

For simplicity, assume the lamination $\mathcal{L}$ is also maximal, so the complement of $\mathcal{L}$ consists of ideal triangles. The general case is essentially the same, with $T$ below replaced by a complementary component.

Fix an ideal triangle $T$, and take one of its lifts $\tilde T$. Let $C_0$ be the boundary circle of the support plane for $\tilde T$. Choose one point $p$ of $C_0\cap \Lambda$, and let $C'_i$ be a sequence of circles contained in the closed disk bounded by $C_0$ and tangent to $C_0$ at $p$ so that $C'_i\to C_0$. There exists a sequence of support circles $C_{i,k}$ so that $C_{i,k}\to C'_i$. Passing to a subsequence if necessary, we may assume the corresponding flat pieces $\tilde T_{i,k}$ lie on the same side of $\tilde T$. As a matter of fact, we may assume these pieces lie on the left of $\tilde T$ when we view $p$ as lying on the left in reference to a unit crossing $l$ on $\tilde T$ (see definition in the previous section); see Figure~\ref{fig:noexotic}.
\begin{figure}[htp]
    \centering
\begin{subfigure}[t]{0.45\linewidth}
\centering
    \includegraphics[width=\linewidth]{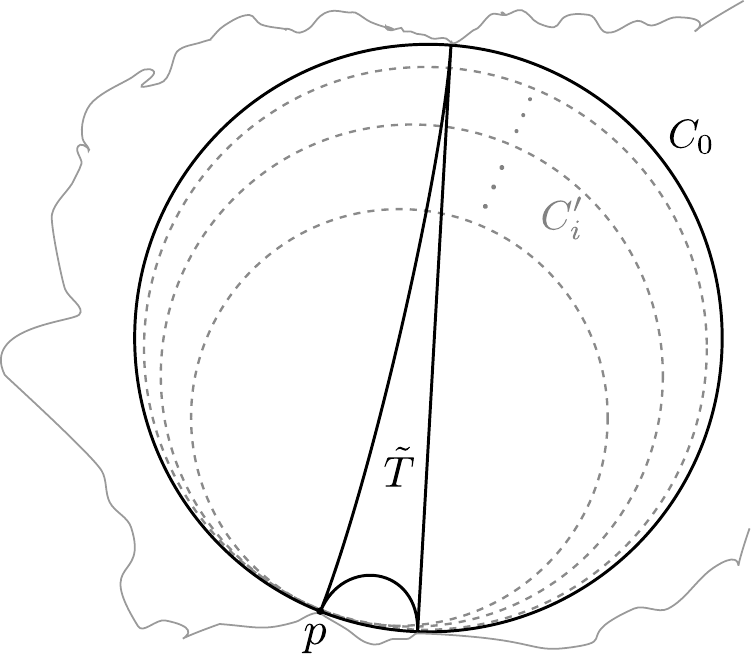}
    \caption{A sequence of asymptotic circles $C_i$ at $p$}
\end{subfigure}
\begin{subfigure}[t]{0.45\linewidth}
\centering
    \includegraphics[width=\linewidth]{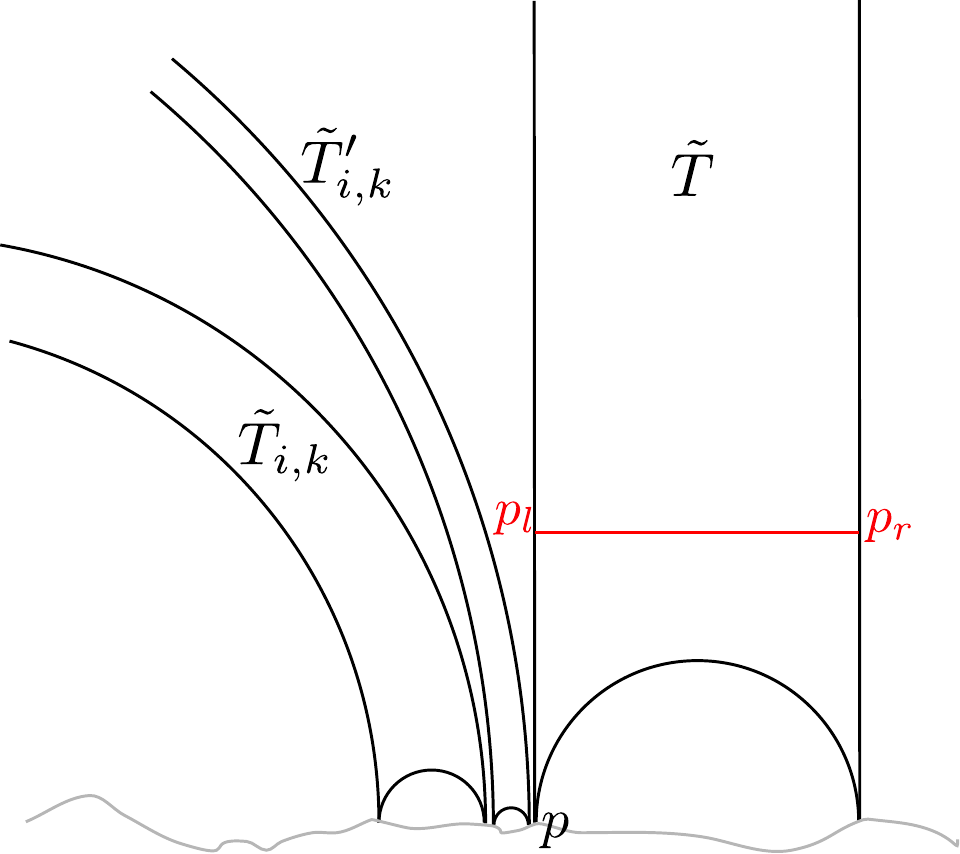}
    \caption{Arrangement of ideal triangles}
    \label{fig:noexotic2}
\end{subfigure}
    \caption{Illustration of the construction}
    \label{fig:noexotic}
\end{figure}

For a fixed $i$, since the angle between $C_{i,k}$ and $C_0$ goes to zero, the shortest geodesic on the convex core boundary between $\tilde T$ and $\tilde T_{i,k}$ goes to zero, whose bending measure goes to zero as well.

There exists another flat piece $\tilde T'_{i,k}$ between $\tilde T$ and $\tilde T_{i,k}$ so that $\tilde T'_{i,k}$ is an orbit of $\tilde T$ under $\Gamma$, and is oriented the same way as $\tilde T$. We may choose $\tilde T'_{i,k}$ as close to $\tilde T_{i,k}$. Indeed, by minimality, a geodesic ray contained in the leaf on the left side of $\tilde T$ (as in Figure~\ref{fig:noexotic2}) going away from $p$ is dense in $\mathcal{L}$ under projection to the convex core boundary; so we can find an orbit of $\tilde T$ in the universal cover as close to $\tilde{T}_{i,k}$ as we want. The corresponding support circle $C_{i,k}'$ of $\tilde{T}'_{i,k}$ thus has radius as close to $C_{i,k}$ as we want.

Let $p_l$ and $p_r$ be the left and right endpoints of the unit crossing $l$ on $\tilde T$ (see Figure~\ref{fig:noexotic2}). For a fixed $i$, The distance of $p_l$ to $\tilde T_{i,k}$ must go to infinity as $k\to\infty$. Indeed, otherwise we can find points $x_k\in\tilde T_{i,k}$ limiting to a point $x$ on the convex core boundary. The support planes at $x_k$, and the corresponding support circles $C_{i,k}$, then limit to the support plane at $x$ and the corresponding support circle respectively, contradicting the fact that $C_{i,k}$ limits to an asymptotic circle $C_i$. Thus by ignoring finitely many circles for each $i$ we may assume this distance is always $>d$ for every $k$, where $d$ is a fixed positive constant. In particular, the thin part of $\tilde T$ above the unit crossing $l$ in Figure~\ref{fig:noexotic2} have distance at least $d$ from the corresponding thin part of $\tilde T_{i,k}$ as well.

We choose a sequence of circles $\hat C_j$ inductively as follows. Suppose the radius of $C_0$ is $R$. We can choose $i_1$ large enough so that the radius of $C_{i_1}$ is $>R/2$. We can then choose $k_1$ large enough so that the radius of $C_{i_1,k_1}$ is $>R/2$, and thus we can choose $C_{i_1,k_1}'$ with radius $>R/2$. Let $\hat C_1=C_{i_1,k_1}'$. Recall that $\hat C_1=C_{i_1,k_1}'$ is in the $\Gamma$-orbit of $C_0$. Set $\hat C_1=\gamma_1C_0$. Inductively, given $\hat C_j=\gamma_jC_0$ with radius $>R/2$, we can choose $i_{j+1},k_{j+1}$ large enough so that $\gamma_jC_{i_{j+1},k_{j+1}}$ has radius $>R/2$, and then choose $C_{i_{j+1},k_{j+1}}'$ so that $\gamma_jC'_{i_{j+1},k_{j+1}}$ has radius $>R/2$. We set $\hat C_{j+1}=\gamma_jC'_{i_{j+1},k_{j+1}}$, which is again in the $\Gamma$-orbit of $C_0$, so $\hat C_{j+1}=\gamma_{j+1}C_0$ for some $\gamma_{j+1}$.

Passing to a subsequence if necessary, we may assume $\hat C_j\to \hat C$. We claim that $\hat C$ is (shadowed by) an exotic circle, i.e. $\hat C\cap\Lambda$ is contained in the halo $h\mathcal{L}$. Let $l_{i,k}$ be the shortest curve on the convex core boundary between $\tilde T$ and $\tilde T_{i,k}$. This curve crosses $\tilde T'_{i,k}$, and suppose it intersects the left side of $\tilde T_{i,k}$ at $p_{i,k}$. We construct a ray as follows. It starts at $p_l$, going along the left side of $\tilde T$ towards $p$ until it reaches the right endpoint of $l_{i_1,k_1}$; then it goes along $l_{i_1,k_1}$ until it reaches $p_{i_1,k_1}$. Recall that $\tilde T'_{i_1,k_1}=\gamma_1\tilde T$. The ray then goes along the left side of $\tilde T'_{i_1,k_1}$ towards $\gamma_1p$ until it passes $\gamma_1p_l$ and reaches the right endpoint of $\gamma_1l_{i_2,k_2}$. It then goes along $\gamma_1l_{i_2,k_2}$ until it reaches $\gamma_1p_{i_2,k_2}$ and thus arrives at the left side of $\gamma_1\tilde{T}'_{i_2,k_2}=\gamma_2\tilde T$. It is clear how we can continue this construction inductively. See Figure~\ref{fig:exotic_ray_path} for an illustration.
\begin{figure}[htp]
    \captionsetup{width=.6\textwidth}
    \centering
    \includegraphics[width=0.45\linewidth]{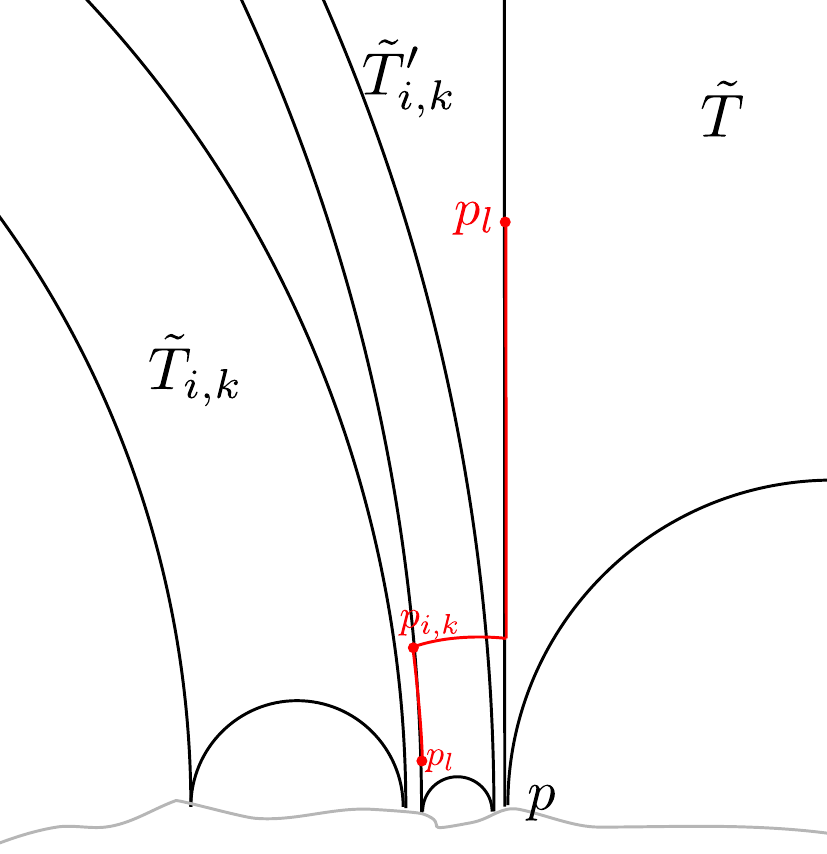}
    \caption{Construction of a ray limiting to $\hat C\cap\Lambda$. Note that the lower point labeled by $p_l$ is the image of the higher one by some $\gamma\in\Gamma$, so we can continue this construction inductively.}
    \label{fig:exotic_ray_path}
\end{figure}

Since $\hat C_j\to\hat C$, this ray limits at infinity to a point in $\hat C\cap\Lambda$. So to finish the proof, it suffices to show the ray is exotic. Suppose to the contrary that it limits at infinity to the endpoint $q$ of a leaf $F$. Note that by construction, the ray above starts from a flat piece $\tilde T$ and crosses flat pieces $\gamma_j\tilde T$ for all $j\ge1$. For any $j$, the flat piece $\gamma_j\tilde T$ cannot be in a position as in Figure~\ref{fig:notleafa}, since it that case the endpoint of the ray is bounded away from $q$ by an ideal point of $\gamma_j\tilde T$. So all of these pieces must all be aligned as in Figure~\ref{fig:notleafb}. This implies when $j$ is large enough, the spikes of $\gamma_j\tilde T$ and $\gamma_{j+1}\tilde T$ aligned along the leaf $F$ (in Figure~\ref{fig:notleafb}, these are the parts of ideal triangles stretching upwards) have distance going to zero. This contradicts our assumption that this distance is bounded below by a positive constant $d>0$.
\begin{figure}[htp]
    \centering
    \begin{subfigure}[t]{0.48\linewidth}
    \centering
         \includegraphics[width=\linewidth]{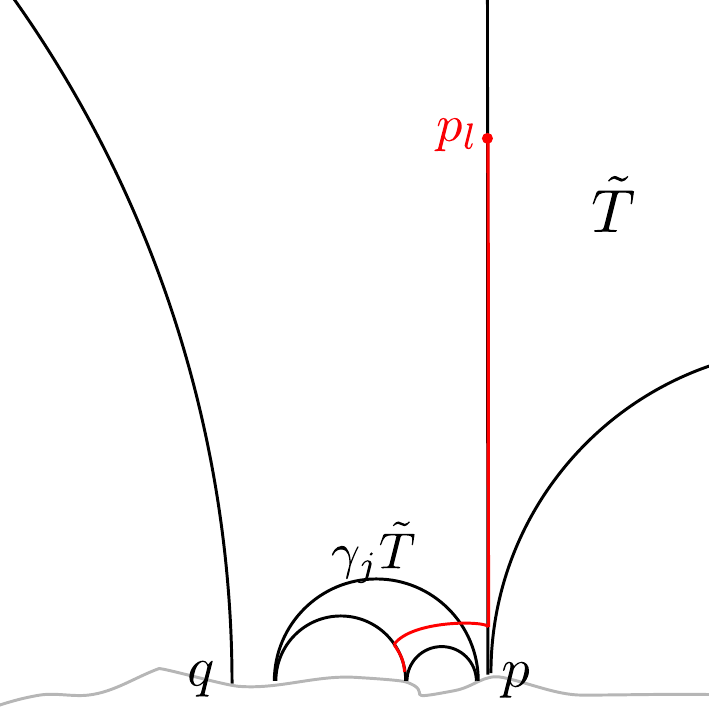}
    \caption{}
    \label{fig:notleafa}
    \end{subfigure}
    \begin{subfigure}[t]{0.48\linewidth}
    \centering
         \includegraphics[width=\linewidth]{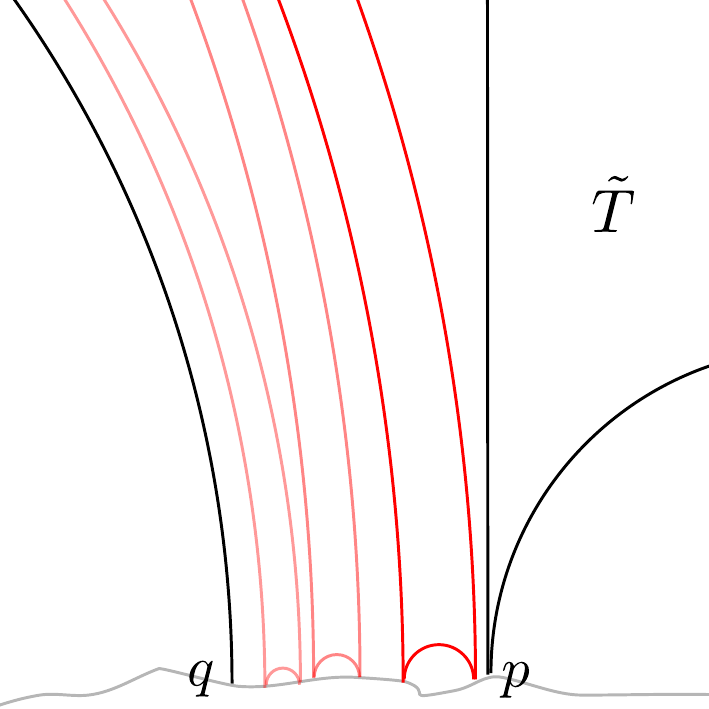}
    \caption{}
    \label{fig:notleafb}
    \end{subfigure}
    \caption{The configuration on the left cannot happen if the ray limits to the endpoint $q$ of a leaf, so all flat pieces $\gamma_j\tilde T$ must be aligned as on the right}
    \label{fig:notleaf}
\end{figure}
\end{proof}

With this proposition in hand, we are ready to prove Theorem~\ref{thm: no_exotic_roof_closures}:
\begin{proof}[Proof of Theorem~\ref{thm: no_exotic_roof_closures}]
For Part (1), any enclosed asymptotic circle lies in the closure by Lemma~\ref{lm: enclosed}. By the previous proposition, we immediately have that the union of all support planes and all enclosed asymptotic planes is closed.

For Part (2), it suffices to show that the set on the right-hand side is closed. Let $C_i=\gamma_iC$ be a sequence of circles in $\Gamma\cdot C$ limiting to $C_0$. Since $E_+$ contains no exotic roof, either $C_0$ is a support circle (and we are done), or $C_0$ is shadowed by a support circle $C_0'$. If $C_0$ is enclosed, then we are also done, so we may assume $C_0$ is stemmed. Let $C'$ be the support circle shadowing $C$. We must have $C_i'=\gamma_iC'\to C_0'$.

Let $K, K_0$ be the convex hull of $C'\cap\Lambda, C_0'\cap\Lambda$ in $\mathbb{H}^3$, i.e.\ they are flat pieces in the corresponding support planes. Since the Hausdorff limit of the sets $\gamma_iC'\cap\Lambda$ is contained in $C_0'\cap\Lambda$, we have $\gamma_iK$ either shrinks to a point, or limits on $K_0$. Note that $\gamma_iK$ is disjoint from $K_0$ so in the latter case, $\gamma_iK$ limits on the boundary of $K_0$.

If $\gamma_iK$ shrinks to a point, then $\gamma_iC'\cap\Lambda$ limits to a single point. Noting the remark after the proof of Lemma~\ref{lm:special_roof_support_limit}, we can follow the arguments of Lemma~\ref{lm:special_roof_support_limit} and Theorem~\ref{prop: support_closure_frame} to conclude that any asymptotic circle is the limit of a sequence of support circles. Hence the construction in the proof of Proposition~\ref{prop:exotic_plane_contrad} produces an exotic circle, a contradiction.
Thus $\gamma_iK$ limits on the boundary of $K_0$, since the Hausdorff limit of $\gamma_iC'\cap\Lambda$ in particular contains the unique point $p$ in $C_0\cap\Lambda$, $\gamma_iK$ only limits on the two boundary components of $K_0$ which end at $p$. Passing to a subsequence if necessary, we may assume $\gamma_iK$ limits on one single boundary component $l$ of $K_0$.

As each $\gamma_iK$ is isometric to $K$, this is only possible when the thick part of $\gamma_iK$ goes to infinity. So for one point $q$ in $C'\cap\Lambda$, $\gamma_iq$ limits to one endpoint of $l$, while the orbit of all other points in $C'\cap\Lambda$ limits to the other endpoint of $l$. It is then clear that either $\gamma_i C\to C_0'$ and so $C_0=C_0'$ (when $C\cap C'=q$), or the radius of $\gamma_i C$ goes to 0 (when $C\cap C'\neq q$). This gives the desired result.
\end{proof}
Combining Theorems~\ref{prop: support_closure_frame} and \ref{thm: no_exotic_roof_closures}, we have the remaining parts of Theorem~\ref{thm: minimal_roof}.

\paragraph{Several minimal components}
In general, write $\mathcal{L}=\sqcup_\alpha\mathcal{L}^\alpha$ into minimal components. let $X^\alpha$ be the smallest closed subsurface containing $\mathcal{L}^\alpha$. If $X^\alpha$ has any boundary component $\gamma$, there exists a connected component of $X^\alpha\backslash\mathcal{L}^\alpha$ containing $\gamma$. Geometrically, this connected component is a crown, bounded by leaves of $\mathcal{L}^\alpha$ other than $\gamma$.

For simplicity, by \emph{a support circle of $X^\alpha$} we mean a circle whose corresponding geodesic plane supports a lift $\tilde{X}^\alpha$ of $X^\alpha$.

We say a geodesic ray $r:[0,\infty)\to X$ \emph{involves} $\mathcal{L}^\alpha$ if it intersects $\mathcal{L}^\alpha$ recurrently, that is, $r([T,\infty))\cap\mathcal{L}^\alpha\neq\varnothing$ for all $T>0$. We also say a point $p$ in the halo $h\mathcal{L}$ \emph{involves} $\mathcal{L}^\alpha$ if any exotic ray ending at $p$ involves $\mathcal{L}^\alpha$. In Part I, we focused on constructing exotic rays that only involve one minimal component. Similarly, we say an exotic plane $P$ (and the corresponding exotic circle $C$) \emph{involves} $\mathcal{L}^\alpha$ if the unique point in $C\cap\Lambda$ involves $\mathcal{L}^\alpha$. We claim
\begin{prop}
    Let $C$ be an exotic circle involving $\mathcal{L}^\alpha$ (but possibly more). Then there exists an exotic circle involving only $\mathcal{L}^\alpha$.
\end{prop}
\begin{proof}
    The key ideas of the proof are already contained in previous discussions. Here we explain what needs to be modified.

    First note that by Proposition~\ref{prop: support_flat}, the closure of support planes to $X_\alpha$ contains all support planes. Therefore in the proof, whenever we choose a convergent sequence of circles, they can be chosen as support circles of $X_\alpha$. From Remark \ref{rmk: no.minimality}, we know that any asymptotic circle shadowed by $C$ is a limit of support circles of $X_{\alpha}$. As a result, we can see that the proof of Lemma~\ref{lm:exotic_limit_to_leaf} goes through without the minimality assumption, and we can actually make sure $p_n=C_n\cap\Lambda$ (as chosen in the proof of Lemma~\ref{lm:exotic_limit_to_leaf}) limits to an endpoint of a leaf in a lift $\tilde{\mathcal{L}}^\alpha$ of $\mathcal{L}^\alpha$. Indeed, since $C$ involves $\mathcal{L}^\alpha$, in the proof of Lemma~\ref{lm:exotic_limit_to_leaf}, we can choose $t_n$ such that $r(t_n)$ lies on a leaf of $\tilde{\mathcal{L}}^\alpha$.

    Now similarly to the proof of Theorem~\ref{prop: support_closure_frame}, we conclude that the closure of $\Gamma$-orbit of any support circle of $X_\alpha$ contains every asymptotic circle based at the endpoint of a leaf of any lift $\tilde{\mathcal{L}}^\alpha$ of $\mathcal{L}^\alpha$.

    Finally, arguing as in the proof of Proposition~\ref{prop:exotic_plane_contrad}, the construction there gives an exotic plane involving only $\mathcal{L}^\alpha$. The key here is to choose the distinguished point $p$ (described in the proof of Proposition~\ref{prop:exotic_plane_contrad}) as the endpoint of some lift $\tilde{\mathcal{L}}^\alpha$ to make sure the constructed ray only involves $\mathcal{L}^\alpha$.
\end{proof}
This means that we can check if exotic planes exist for each minimal component separately. If there are no exotic planes involving only $\mathcal{L}^\alpha$, then no exotic plane involves $\mathcal{L}^\alpha$ at all.

Finally, recall that if $X^\alpha$ has a boundary component, that component is part of a crown in $X^\alpha-\mathcal{L}^\alpha$. Note that any stemmed asymptotic plane for a lift of the crown can only be accessed as a limit of support or asymptotic planes from the side of a lift of $X^\alpha$ (see Figure~\ref{fig:qf_limit}).
\begin{figure}[htp]
    \centering
    \captionsetup{width=.8\linewidth}
   \includegraphics{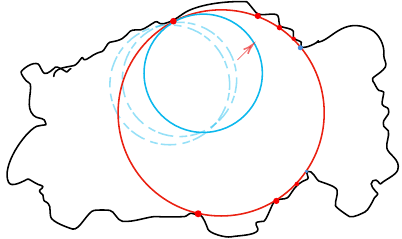}
\caption{A stemmed asymptotic circle as a limit of asymptotic circles. Note that if the bending lamination has more than one minimal component, stemmed asymptotic circles for boundary components of the subsurface containing a minimal component can only be accessed from the subsurface.}
\label{fig:qf_limit}
\end{figure}

These observations give Proposition~\ref{prop:several_comp}.

\section{Leaf approximations for punctured tori}\label{sec:tori}
In this section, we prove Theorem~\ref{thm:sewa}, i.e.\ every well-approximated irrational lamination on a punctured torus exhibits a CESAG sequence of leaf approximations. Combined with Theorem~\ref{thm: sufficient}, exotic roofs exist for quasifuchsian manifolds with such a lamination. The construction is given via a somewhat different symbolic coding scheme from \S\ref{sec:entropy}, based on flat structures instead of train tracks. We remark that this version of symbolic coding can be generalized to higher genera as well.

\paragraph{Measured laminations on a punctured torus}
We first review some facts about laminations on punctured tori. Given a hyperbolic torus $X$ with one cusp, let $\bar X$ be the torus with the cusp filled in (and denote the new point by $p$). Write $(\bar X,p)\cong(\mathbb{C},\Lambda)/\Lambda$ where $\Lambda=\mathbb{Z}\oplus\mathbb{Z}\tau$ for some $\tau\in\mathbb{H}$. For simplicity let $\tau=i$.

Any measured lamination on $X$ comes from a measured foliation on $\bar X$. More precisely, fix a real number $\theta$ and consider all lines in $\mathbb{C}$ of slope $\theta$. This gives a foliation on $\bar X$, and the transverse measure is given by (a constant multiple of) the length measure in the perpendicular direction. Straightening the lines with respect to the hyperbolic metric on $X$ gives a geodesic lamination, which is a simple closed curve if and only if $\theta$ is rational. For simplicity, we call the leaf of a lamination/foliation with slope $\theta$ a \emph{$\theta$-leaf}.

\paragraph{Symbolic coding of rays}
We can present $X$ as $(\mathbb{C}-\mathbb{Z}^2)/\mathbb{Z}^2$, where $\mathbb{C}-\mathbb{Z}^2$ is endowed with a unique complete hyperbolic metric. By symmetry, horizontal and vertical segments connecting neighboring integral points are complete geodesics in this metric. The unit square $S=[0,1]\times[0,1]$ minus the four vertices is isometric to an ideal quadrilateral. We refer to the left, bottom, right, top sides of $S$ as $I_l, I_b, I_r, I_t$.

Set $\bolda$ as the map $z\mapsto z+1$ and $\boldb$ as $z\mapsto z+i$. We consider $I_l$, $I_b$, $I_r$, $I_t$ as labelled by $\bar\bolda(:=\bolda^{-1})$, $\bar\boldb$, $\bolda$, $\boldb$, as say e.g. the square $\bolda S$ shares the side $I_r$ with $S$. Given a geodesic ray $r$ in $X$, its preimage via the map $S\to X$ consists of countably many geodesic segments $r_1,r_2,r_3,\ldots$, and each segment $r_i$ ends on one of the sides of $S$. Recording the label of the side $r_i$ lands on, we obtain a reduced infinite word $w_r$. Two geodesic rays are asymptotic if and only if the corresponding words have the same tail.

We may of course assign a finite word to a finite geodesic segment in the same way. If we restrict to geodesic segments contained in the thick part of $X$ (i.e. the complement of a small enough cuspidal neighborhood in $X$), each segment $r_i$ from one side to another side of $S$ has length bounded above and below (with the possible exception of the first and last segments). Thus there exist constants $c,C>0$ so that for any geodesic segment $r$,
$$c(l_w(w_r)-1)\le l_h(r)\le C(l_w(w_r)+1)$$
where $l_h$ denotes the hyperbolic length, and $l_w$ denotes the word length.

In the next few sections, we want to identify finite (and infinite) words that come from $
\theta$-leaf segments, which we call $\theta$-words. This is covered by the well-studied theory of Sturmian words (see e.g.\ \cite{sturmian1}, \cite{sturmian2}, and for an exposition, \cite{sturmian}). To better incorporate them into our discussions, however, we will give a self-contained exposition.

\paragraph{Continued fractions and rational approximations}
The correspondence between (ir)rational numbers and (ir)rational laminations on $X$ described above, and the end goal of constructing leaf approximations, naturally lead to consideration of rational approximations of irrational numbers. This is best phrased in terms of continued fractions, which we briefly recall now. For details, see e.g.\ \cite{continued_fraction}.

Let $\theta$ be a positive real number, and $[c_0;c_1,c_2,\ldots]$ its continued fraction, which has finite length if and only if $\theta$ is rational. Let $p_k/q_k:=[c_0;c_1,\ldots,c_k]$ be the $k$-th \emph{convergent} of $\theta$. Clearly $p_k/q_k\le\theta$ when $k$ is even, and $p_k/q_k\ge\theta$ when $k$ is odd. They are, in some sense, best approximations to $\theta$ (see, for example, \cite[\S II.6]{continued_fraction}):
\begin{thm}
For any integer $k\ge0$, we have $|q_k\theta-p_k|<1/q_{k+1}$. Moreover, $\min_{\substack{p,q\in\mathbb{Z}\\0<q\le q_k}}|q\theta-p|=|q_k\theta-p_k|$.
\end{thm}

As an application, we prove the following lemma, which suggests that a suitable $\theta$-leaf segment can be homotoped to a $p_k/q_k$-leaf segment when $k$ is large enough (and thus a finite $\theta$-word is always a $p_k/q_k$-word for $k$ sufficiently large):
\begin{lm}
Let $s\in(0,1)$ and set $\epsilon_k=\begin{cases}s&\text{$k$ is odd}\\1-s&\text{$k$ is even}\end{cases}$. If $1/q_{k+1}<\epsilon_k$, then there exists $s'\in(0,1)$ so that no integer lattice points lie in the region bounded by the lines $y=\theta x+s$, $y=(p_k/q_k)x+s'$, $x=0$, and $x=q_k$.
\end{lm}
\begin{proof}
We shall consider the case where $k$ is even; the other case is similar. In particular $p_k/q_k\le\theta$. For each integer $0\le l\le q_k-1$, let $y_l=s+(p_k/q_k)l \mod 1$. Then $y_0,\ldots,y_{q_k-1}$ are $q_k$ points of equal distance $1/q_k$ on $(0,1)$. Suppose $y_{l_0}$ is the largest of them. Let $s'=s+l_0(\theta-p_k/q_k)\in(0,1)$. Then when $x=l_0$, $y=\theta x+s=(p_k/q_k)x+s'$. Otherwise, when $l\neq l_0$, $1-y_l>1/q_k$. Since $\theta l+s$ and $(p_k/q_k)l+s'$ are both at most $1/q_k$ above $(p_k/q_k)l+s$, they have the same integral part.
\end{proof}

\paragraph{An algorithm for rational words}
It is thus useful to know rational words, i.e.\ words of simple closed geodesics. For this we quote \cite{simple_word}:

\begin{thm}\label{thm:simple_word}
Up to permutations of generators which interchange $\bolda$ and $\boldb$, $\bolda$ and $\bar\bolda$, or $\boldb$ and $\bar\boldb$, a simple word in $\pi_1(X)$ is up to cyclic permutations either $\bolda$ or $\bolda\bar\boldb\bar\bolda\boldb$ or
$$\bolda^{n_1}\boldb\bolda^{n_2}\boldb\cdots\bolda^{n_k}\boldb$$
where $\{n_1,\ldots,n_k\}\subset\{n,n+1\}$ for some $n\in\mathbb{Z}^+$.
\end{thm}
Note that the word $w=\bolda^{n_1}\boldb\bolda^{n_2}\boldb\cdots\bolda^{n_k}\boldb$ corresponds to $\bolda^{\sum_{j=1}^kn_j}\boldb^k$ in $\pi_1(\bar X)$, which gives a simple closed curve of slope $k/\sum_{j=1}^kn_j\in[1/(n+1),1/n]$. For simplicity, we write $w=(n_1,\ldots,n_k)$ (or $(n_1,\ldots,n_k)_{\bolda,\boldb}$ if we want to specify the letters), and call each segment $\bolda^{n_j}\boldb$ an \emph{$n_j$-block} of $w$.

We now describe an algorithm to produce (and exhaust) $r$-words for a rational number $r=p/q$. Without loss of generality, we may assume $r>1$; the case of $r<1$ can be obtained from that of $1/r$ by exchanging $\bolda$ and $\boldb$.

For any leaf of slope $r$ not passing through the lattice points, choose a start point on $I_l$, the left side of unit square $S$. The first few words must be $\boldb$, followed by an $\bolda$, which brings the segment back to $I_l$. Thus the word for this simple closed geodesic has the form described in Theorem~\ref{thm:simple_word} with $\bolda$, $\boldb$ interchanged. The integer $n$ specified in the theorem can be chosen as $\lfloor r\rfloor$.

It is easy to see that the number of $n$-blocks is given by $s=(n+1)q-p$, and the number of $(n+1)$-blocks is given by $t=p-nq$. Each block represents a sequence of segments in $S$, with the first segment starting on $I_l$ and the last ending on $I_r$. We have:
\begin{lm}
\begin{enumerate}[label=\normalfont{(\arabic*)}, topsep=0mm, itemsep=0mm]
    \item The starting points of the $(n+1)$-blocks are on top of those of the $n$-blocks, while the endpoints are switched.
    \item The relative position of the starting points of $n$-blocks is the same as the relative position of the endpoints; the same is true for $(n+1)$-blocks.
\end{enumerate}
\end{lm}
\begin{figure}[htp]
    \centering
    \includegraphics{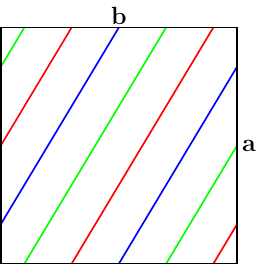}
    \caption{$1$- and $2$-blocks in $S$}
    \label{fig:blocks}
\end{figure}
The proof is trivial; see Figure~\ref{fig:blocks} for an example. This suggests the following algorithm to produce $r$-words for $r=p/q>1$:
\begin{alg}\label{alg:simple_word}
\begin{enumerate}[topsep=0mm, itemsep=0mm]
    \item Set $s=(n+1)q-p$ and $t=p-nq$. Suppose $sp:\{1,\ldots,q\}\to\{n,n+1\}$ and $ep:\{1,\ldots,q\}\to\{n,n+1\}$ record the ordered list of blocks at the starting points and the ending points respectively. That is, $sp(j)=n+1=ep(s+j)$ for $1\le j\le t$ and $sp(k+t)=n=ep(k)$ for $1\le k\le s$. Set $w$ to be the empty word. Let $l=l_1$ be any integer between $1$ and $s+t$.
    \item Set $w=w\boldb^{sp(l)}\bolda$. If $1\le l\le t$, set $l=s+l$; otherwise, set $l=l-t$.
    \item If $l=l_1$, output $w$; otherwise repeat Step 2.\qed
    \end{enumerate}
\end{alg}
As a matter of fact, different choices of $l_1$ simply yield cyclic rearrangements of blocks.

\paragraph{Irrational words}
The discussions above on rational words, together with the idea of rational approximations, immediately give the following proposition on irrational words:
\begin{prop}
Let $\theta=[c_0;c_1,c_2,\ldots]$ be a irrational number $>1$, $p_k/q_k$ its $k$-th convergent, and $n=\lfloor\theta\rfloor$. Let $w$ be the word of a half leaf of slope $\theta$ starting at a point on $I_l$ and not passing through any lattice points. Then
\begin{enumerate}[label=\normalfont{(\arabic*)}, topsep=0mm, itemsep=0mm]
    \item $w=\boldb^{n_1}\bolda\boldb^{n_2}\bolda\cdots=(n_1,n_2,\dots)$, where $n_j\in\{n,n+1\}$;
    \item For all $k$ large enough, $w(k):=(n_1,n_2,\ldots,n_{q_k})$ gives a $p_k/q_k$-word.
\end{enumerate}
Furthermore, any $p_k/q_k$-word can be extended to a $\theta$-word.
\end{prop}
We remark that similar to rational words, in the case $\theta\in(0,1)$, we can obtain $\theta$-words by exchanging $\bolda$ and $\boldb$ in $1/\theta$-words.

\paragraph{Inadmissible words}
A finite word is said to be ($\theta$-)\emph{inadmissible} if it does not appear in any $\theta$-word; otherwise, it is said to be ($\theta$-)\emph{admissible}. A word may be inadmissible even if it consists of the same number of blocks as a $p_k/q_k$-word. For example, if $\theta\in(n,n+1/2)$, no $p_k/q_k$-word contains two consecutive $(n+1)$-blocks. We have
\begin{lm}\label{lem:inadmissible}
There exists a $\theta$-inadmissible word consisting of $q_k$ blocks for all $k\ge 2$, so that by changing the last block from $n$ to $n+1$, or $n+1$ to $n$ it becomes admissible.
\end{lm}
\begin{proof}
We start with the case where $k$ is even. Let $w_k$ be the $p_k/q_k$-word starting at the lowest point on $I_l$; that is, the word we obtain from Algorithm~\ref{alg:simple_word} with $r=p_k/q_k$ and $l_1=q_k$. This word starts with an $n$-block and ends with an $n+1$-block. Let $w'_k$ be the word obtained from $w_k$ by changing the last block to an $n$-block.

We claim that $\bolda w_k'$ is not $\theta$-admissible. For this, it suffices to show that it is not $p_l/q_l$-admissible for all large $l$. Indeed, since $k$ is even, $p_k/q_k<p_l/q_l$ for all $l$ large. Moreover, $|q_k\frac{p_l}{q_l}-p_k|<1/q_{k+1}$.

Suppose $\theta\in(n+1/(m+1),n+1/m)$. Then $w_k$ and $w_k'$ start with $m$ pieces of $n$-blocks. Any $p_l/q_l$-word that starts with $m$ pieces of $n$-blocks and is preceded by an `$\bolda$' must start near the bottom of $I_l$. In particular, the first $q_k$ blocks of such a word is necessarily a $p_k/q_k$-word. Thus $w_k'$ cannot be part of a $p_l/q_l$ word, as desired.

When $k$ is odd, we set $w_k$ to be the $p_k/q_k$-word starting at the highest point on $I_l$ (that is, set $l_1=1$ in the algorithm). Then the word $w_k'$ obtained from $w_k$ by changing the last $n$-block to an $(n+1)$-block is inadmissible, by an entirely analogous argument.
\end{proof}

\paragraph{CESAG leaf approximations}
Let $y_\theta=-\theta\mod 1 \in (0,1)$. Then, any $\theta$-leaf starting at a point on $I_l$ above $iy_\theta$ begins with an $(n+1)$-block in its word, and any leaf starting below begins with an $n$-block. In the hyperbolic metric, the leaf $L_{y_\theta}$ starting at $y_\theta$ is straightened to a geodesic going into the cusp, and is contained in a complementary region of the corresponding measured lamination. This region is bounded by two complete geodesics, one from the limit of all leaves starting above $y_\theta$, and the other from the limit of all leaves starting below $y_\theta$. We denote them by $L_{\theta,+}$ and $L_{\theta,-}$ respectively.

Similarly, the leaf $L_0$ starting at $0$ is also straightened to a geodesic coming out of the cusp and leaves above (resp. below) it limit to $L_{\theta,+}$ (resp. $L_{\theta,-}$) as well.

Let $p_r$ (resp. $p_l$) be the lowest (resp. highest) intersection of $L_{\theta,+}$ (resp. $L_{\theta,-}$) with the complete geodesic $g_v$ given by the interval $(0,i)$ on the imaginary axis. Let $p_-$ be the first intersection (starting from $p_l$) of $L_{\theta,-}$ and the complete geodesic $g_h$ given by the interval $(0,1)$ on the real axis. 
\begin{figure}[htp]
    \centering
    \captionsetup{width=0.8\linewidth}
    \includegraphics[width=0.4\linewidth]{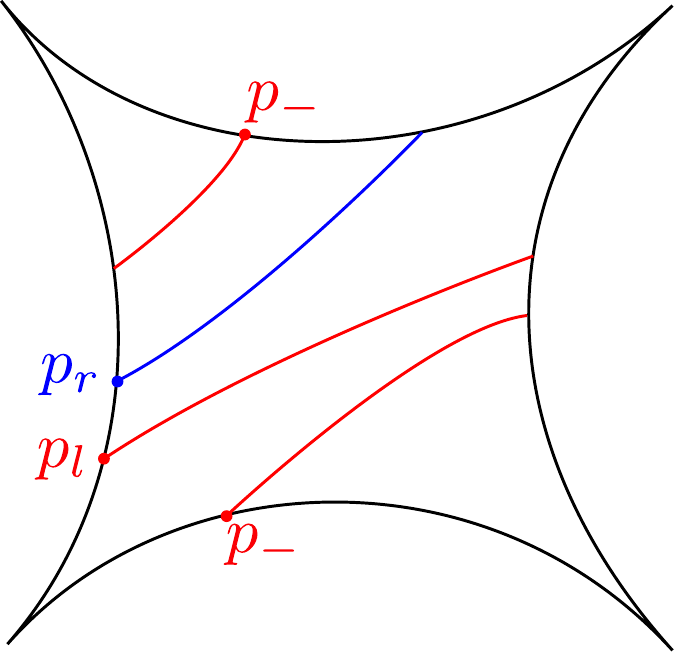}
    \caption{In the punctured torus, each leaf straighten to complete geodesics. Here we draw the hyperbolic picture of a fundamental domain. The blue segment is part of $L_{\theta,+}$ and the red segments are consecutive parts of $L_{\theta,-}$.}
\end{figure}

For each even number $k=2j\ge2$, let $w_k$ be the word constructed in the proof of Lemma~\ref{lem:inadmissible}. Let $L'_j$ be the segment on $L_{\theta,-}$ starting at $p_-$ and ending at a point $p_j'$ on $g_v$ whose word is exactly $w_k$, and $L_j''$ be the segment on $L_{\theta,-}$ following $L'_j$ ending at a point $p_j''$ on $g_v$ whose word is exactly first $q_k-1$ blocks of $w_k$. Let $L_c$ be the segment on $L_{\theta,-}$ starting at $p_l$ and ending at $p_-$. Finally, set $L_j$ to be $L_c$ followed by $L'_j$ and then $L_j''$.

Note that $p_j''$ and $p_r$ can be connected by a segment of transverse measure bounded above by
$$2cq_k(\theta-p_k/q_k)<2c/q_{k+1}.$$
where $c$ is a constant only depending on $\theta$. (Technically, for everything to work, we need to make sure that the words can actually be realized on the leaf; a sufficient condition is that $2q_k(\theta-p_k/q_k)<1/q_k$, which is always true if $c_{k+1}\ge2$ in the continued fraction.)

This gives a sequence of leaf approximations, with leaf length $d_j$ satisfying
$$cq_k\le d_j=l(L_j)\le Cq_k$$
where $c,C$ are constants depending only on $\theta$ and the hyperbolic metric of $X$.

It is not hard to see that the leaf approximation is $e^{Cd}$-separating, so to get a CESAG leaf approximation, we need $\exp(Cq_k)/q_{k+1}\to0$, for some fixed but unspecified constant $C$. This can be guaranteed by making $c_k$ much larger than $q_{k-1}$, e.g.\ we may take $c_k=\lceil e^{q_{k-1}^2}\rceil$ inductively.

In general, we only need a sequence of even (or odd) numbers $k_j\to\infty$ so that $\exp(Cq_{k_j})/q_{k_j+1}\to0$ for any fixed constant $C>0$ to make the construction work. This gives Theorem~\ref{thm:sewa}.

\paragraph{Some final remarks}
In the construction above we are only able to obtain exponentially separating leaf approximations, which is the main cause of restriction to well-approximated irrational numbers. One natural question is the following: does there exist a lamination with a CESAG leaf approximation that is sub-exponentially separating?

Finally, we want to remark on Corollary~\ref{cor:higher_genus}. In the construction above we approximated the length of a leaf segment using its word length, which only requires that in a fundamental domain, each letter represents a segment whose length is bounded above and below. This remains true in the cases considered in Corollary~\ref{cor:higher_genus}, so we have desired leaf approximations there as well.

\appendix
\section{Appendix: well-approximated irrational numbers are generic}\label{app:generic}
In this appendix, we show
\begin{thm}\label{thm:gdelta}
Well-approximated irrational numbers are generic in the sense of Baire category, i.e.\ they contain a dense $G_\delta$-set.
\end{thm}
Let $\mathcal{SE}$ be the set of well-approximated irrational numbers. We start with the lemma
\begin{lm}
The set $\mathcal{SE}$ is dense in $\mathbb{R}$.
\end{lm}
\begin{proof}
Let $r=[c_0;c_1,\ldots]\in\mathbb{R}$. Define $r_t=[c_0^{(t)};c_1^{(t)},\ldots]$ as follows. For $0\le k\le t$, set $c_k^{(t)}=c_k$ and for $k\ge t+1$, set $c_k^{(t)}=\lceil e^{(q^{(t)}_{k-1})^2}\rceil$, where $q^{(t)}_{k-1}$ is the denominator of the $(k-1)$-th convergent of $r_t$. Clearly, $r_t$ is well-approximated, and $r_t\to r$.
\end{proof}
Let $\mathcal{SE}_d$ be the set of real numbers $r=[c_0;c_1,\ldots]$ so that there exists an increasing sequence $1\le k_1<k_2<\cdots<k_d$ of integers of the same parity satisfying $c_{k_i}\ge \exp(q_{k_i-1}^2)$. Clearly
$$\mathcal{SE}_d\supset\mathcal{SE}_{d+1},\quad \bigcap_{d\ge1}\mathcal{SE}_d\subset\mathcal{SE}$$
Note that the well-approximated numbers $r_t$ chosen in the proof of the lemma above actually lie in $\bigcap_{d\ge1}\mathcal{SE}_d$, so each $\mathcal{SE}_d$ is dense. We have
\begin{lm}
The set $\mathcal{SE}_d$ is open.
\end{lm}
\begin{proof}
Indeed, for any $r\in\mathcal{SE}_d$, any number close enough shares the same first few digits in the continued fraction expansion. In particular, close enough numbers also share the same sequence of $d$ integers in the definition of $\mathcal{SE}_d$.
\end{proof}
These discussions immediately imply Theorem~\ref{thm:gdelta}.
\bibliographystyle{math}
\bibliography{biblio}
\end{document}